\documentclass[11pt]{amsart}
\pdfoutput=1

\usepackage[utf8]{inputenc}
\usepackage[english]{babel}
\usepackage{amssymb}
\usepackage{stmaryrd}
\usepackage{xcolor}
\usepackage{microtype}
\usepackage{mathtools}
\usepackage{enumitem}

\usepackage{lmodern}
\usepackage[margin=1in, marginparwidth=.75in]{geometry}

\usepackage[numbers,compress]{natbib}

\usepackage[bookmarksdepth=3]{hyperref}
\hypersetup{
  pdftitle={Finite element exterior calculus for time-dependent Hamiltonian partial differential equations},
  pdfauthor={Ari Stern and Enrico Zampa},
  pdfsubject={MSC 2020: 65P10 (Primary), 65M60, 37M15, 37K25 (Secondary)},
  bookmarksopen=false
}

\usepackage[capitalize,nameinlink]{cleveref}

\theoremstyle{plain} %--default
\newtheorem{theorem}             {Theorem}  [section]
\newtheorem{lemma}      [theorem]{Lemma}
\newtheorem{corollary}  [theorem]{Corollary}
\newtheorem{proposition}[theorem]{Proposition}

\theoremstyle{definition}
\newtheorem{definition} [theorem]{Definition}
\newtheorem{example}    [theorem]{Example}
\newtheorem{assumption} [theorem]{Assumption}

\theoremstyle{remark}
\newtheorem{remark}     [theorem]{Remark}

\AddToHook{env/lemma/begin}{\crefalias{theorem}{lemma}}
\AddToHook{env/corollary/begin}{\crefalias{theorem}{corollary}}
\AddToHook{env/proposition/begin}{\crefalias{theorem}{proposition}}
\AddToHook{env/definition/begin}{\crefalias{theorem}{definition}}
\AddToHook{env/example/begin}{\crefalias{theorem}{example}}
\AddToHook{env/assumption/begin}{\crefalias{theorem}{assumption}}
\AddToHook{env/remark/begin}{\crefalias{theorem}{remark}}

\newlist{assumptionenumerate}{enumerate}{1}
\setlist[assumptionenumerate]{label=(\roman*),ref=\theassumption(\roman*)}
\crefname{assumptionenumeratei}{Assumption}{Assumptions}
\crefrangelabelformat{assumptionenumeratei}{#3#1--(\crefstripprefix{#1}{#2}#4}

\newlist{lemmaenumerate}{enumerate}{1}
\setlist[lemmaenumerate]{label=(\roman*),ref=\thelemma(\roman*)}
\crefname{lemmaenumeratei}{Lemma}{Lemmas}
\crefrangelabelformat{lemmaenumeratei}{#3#1--(\crefstripprefix{#1}{#2}#4}

\newcommand{\ringhat}[1]{\mathchoice
  {\smash{\mathring{\widehat{#1}}}}%
  {\vphantom{{\mathring{\widehat{#1}}}}\smash{\mathring{\widehat{#1}}}}%
  {\vphantom{{\mathring{\widehat{#1}}}}\smash{\mathring{\widehat{#1}}}}%
  {\vphantom{{\mathring{\widehat{#1}}}}\smash{\mathring{\widehat{#1}}}}%
}

\DeclareMathSymbol{\star}{\mathord}{letters}{"3F}

\title[FEEC for time-dependent Hamiltonian PDEs]{Finite element exterior calculus for time-dependent Hamiltonian partial differential equations}
\author{Ari Stern}
\address{Department of Mathematics, Washington University in St.~Louis}
\email{stern@wustl.edu}

\author{Enrico Zampa}
\address{Department of Mathematics, University of Vienna}
\email{enrico.zampa@univie.ac.at}

\begin{document}

\begin{abstract}
  The success of symplectic integrators for Hamiltonian ODEs has led
  to a decades-long program of research seeking analogously
  structure-preserving numerical methods for Hamiltonian PDEs. In this
  paper, we construct a large class of such methods by combining
  finite element exterior calculus (FEEC) for spatial
  semidiscretization with symplectic integrators for time
  discretization. The resulting methods satisfy a local
  \emph{multisymplectic} conservation law in space and time, which
  generalizes the symplectic conservation law of Hamiltonian ODEs, and
  which carries finer information about Hamiltonian structure than
  other approaches based on global function spaces. We give particular
  attention to conforming FEEC methods and hybridizable discontinuous
  Galerkin (HDG) methods. The theory and methods are illustrated by
  application to the semilinear Hodge wave equation.
\end{abstract}

\maketitle

\section{Introduction}

\subsection{Background and motivation}

Hamiltonian ordinary differential equations (ODEs) and partial
differential equations (PDEs) are ubiquitous in physical
systems. Typically, it is infeasible to solve these equations exactly,
so numerical methods are used to compute approximate
solutions. However, while the exact solutions themselves may be out of
reach, the Hamiltonian structure leads the solutions to have certain
properties that \emph{can} be characterized exactly---symmetries,
conservation laws, etc.---and which it would be desirable for
numerical solutions to share. Over the last few decades, this has
motivated a major line of research into \emph{structure-preserving}
numerical methods, for which the numerical solutions share these
important features of the exact solutions.

Hamiltonian ODEs satisfy the \emph{symplectic conservation law}, which
many standard numerical integrators (e.g., explicit Runge--Kutta
methods) fail to preserve. The development of structure-preserving
\emph{symplectic integrators}, particularly since the 1980s, has led
to major advances in simulation of such systems, with numerical
advantages that have been well studied and documented in the decades
since \citep{SaCa1994,LeRe2004,HaLuWa2006}. This success story for
Hamiltonian ODEs naturally raises a longstanding question: How can we
construct similarly structure-preserving methods for Hamiltonian PDEs?

One approach begins by considering time-dependent Hamiltonian PDEs to
be Hamiltonian dynamics on an infinite-dimensional function
space. Semidiscretizing in space (also known as the ``method of
lines'') then gives an approximation to the infinite-dimensional
dynamics by a finite-dimensional system of ODEs, to which a symplectic
integrator may be applied. This approach is particularly amenable to
finite element semidiscretization, where the infinite-dimensional
function space is replaced by some finite element space. However,
there are some notable obstacles:
\begin{enumerate}
\item It leaves open the question of which semidiscretization methods
  are structure-preserving, in the sense that the resulting ODEs are
  also Hamiltonian.
\item The symplectic structure on a \emph{global} function space does
  not fully capture the finer-scale \emph{local} structure of
  Hamiltonian PDEs. In particular, the symplectic conservation law is
  global, but there are also local conservation laws one would like a
  numerical method to preserve.
\end{enumerate}
There has been considerable work on the first issue, and various
(globally) Hamiltonian finite element methods have been studied for
problems including linear hyperbolic systems \citep{XuVdVBo2008},
surface waves \citep{BrFeVdV2017}, the wave equation and Maxwell's
equations \citep{SaCiNgPe17,SaDuCo22,SaVa24}, and the linearized
shallow water equations \citep{NuSa2025}. The second obstacle, however,
is more fundamental.

A different approach---the one we take in this paper---is to start
from the local Hamiltonian perspective. This originated from
independent work in 1935 of \citet{deDonder1935} and \citet{Weyl1935},
who developed a canonical form for certain Hamiltonian PDEs that
involves finite-dimensional partial derivatives rather than the
infinite-dimensional functional derivatives of the global Hamiltonian
approach (i.e., ordinary calculus rather than calculus of
variations). This theory has continued to advance, and we mention
important contributions due to \citet{Bridges1997,Bridges2006} in the
1990s and 2000s. In this setting, Hamiltonian PDEs satisfy a local
\emph{multisymplectic conservation law}, which implies the global
symplectic conservation law as a consequence (i.e., when integrated
over space) and also contains finer-scale information about the local
structure of solutions.

Initial work on multisymplectic methods for Hamiltonian PDEs,
beginning in the late 1990s, tended to employ rectangular grids (e.g.,
applying a symplectic integrator along each coordinate axis) or
low-order finite difference and finite volume methods on unstructured
meshes; we mention the work of Marsden and collaborators
\citep{MaPaSh1998,MaSh1999,MaPeShWe2001,LeMaOrWe2003} as well as Reich
and collaborators
\citep{Reich2000a,Reich2000b,BrRe2001,MoRe2003,FrMoRe2006}. Since
finite element methods would seem to fit more naturally with the
global-function-space approach, their use in this context was mostly
limited to the construction of multisymplectic finite difference
stencils \citep{GuJiLiWu2001,ZhBaLiWu2003,Chen2008}; however, we also
note the more recent work on 1+1-dimensional multisymplectic finite
element methods by \citet{CeJa21}.

In 2020, \citet{McSt2020} developed a general theory of
multisymplectic finite element methods for stationary Hamiltonian PDEs
in the canonical de~Donder--Weyl form. This work, based on the
\emph{hybridization} framework of \citet{CoGoLa2009}, showed that
numerous standard finite element methods---including conforming,
nonconforming, and hybridizable discontinuous Galerkin (HDG)
methods---satisfy a local multisymplectic conservation law involving
the numerical traces and fluxes arising in the hybrid formulation. In
2024, \citet{McSt2024} extended this to time-dependent de~Donder--Weyl
systems, constructing multisymplectic methods by applying hybrid
semidiscretization in space followed by symplectic integration in
time.

While this work established the Hamiltonian structure-preserving
properties of a wide range of high-order methods on unstructured
meshes, its purview was limited to Hamiltonian PDEs in the
de~Donder--Weyl form, which includes the scalar wave equation but
excludes---for instance---Maxwell's equations, the vector wave
equation, and the Hodge wave equation for differential
$k$-forms. Expanding the approach to these additional systems requires
a more general notion of canonical Hamiltonian PDEs, due to
\citet{Bridges2006}, involving the exterior calculus of differential
forms.

In a recent paper \citep{StZa2025}, the present authors developed a
theory of multisymplectic methods for stationary Hamiltonian PDEs in
the canonical form of \citet{Bridges2006}, extending the work of
\citep{McSt2020} for stationary de~Donder--Weyl systems. These
structure-preserving methods are based on \emph{finite element
  exterior calculus} (FEEC) \citep{ArFaWi2006,ArFaWi2010,Arnold2018}
and the FEEC hybridization framework of \citet*{AwFaGuSt2023},
including conforming, nonconforming, and HDG methods.

In this paper, we complete this program by developing multisymplectic
FEEC methods for time-dependent Hamiltonian PDEs in the more general
form of \citet{Bridges2006}, mirroring how \citep{McSt2024} did so for
time-dependent de~Donder--Weyl systems. Specifically, for an
$ ( n + 1 ) $-dimensional system of Hamiltonian PDEs, we first
semidiscretize in space using the multisymplectic FEEC methods of
\citep{StZa2025}, and subsequently discretize in time using a
symplectic integrator. The resulting structure-preserving methods are
shown to satisfy a discrete local multisymplectic conservation law in
space and time.

\subsection{Outline of paper and contributions}

The paper is organized as follows:

\begin{itemize}
\item \Cref{sec:hamiltonian} briefly recalls the canonical formalism
  of \citet{Bridges2006} for stationary Hamiltonian PDEs, as in
  \citep{StZa2025}, before extending to the time-dependent systems
  considered throughout this paper. We show that the canonical
  equations in spacetime have the novel form
  \begin{align*}
    \dot{q} + \mathrm{D} p &= \frac{ \partial H }{ \partial p }, \\
    -\dot{p} + \mathrm{D} q &= \frac{ \partial H }{ \partial q } ,
  \end{align*}
  where $ \mathrm{D} $ is the \emph{Hodge--Dirac operator}, and
  characterize the multisymplectic conservation law for this
  system. Several examples are introduced (and considered throughout
  the paper), particularly the semilinear Hodge wave equation.

\item \Cref{sec:semidiscrete} studies semidiscretization of these
  systems of PDEs by hybridized FEEC methods. The methods of
  \citep{StZa2025} that are multisymplectic for stationary systems are
  shown to yield a semidiscrete multisymplectic conservation law when
  applied to time-dependent systems. Specific methods for the
  semilinear Hodge wave equation are considered in depth; one class of
  HDG methods, constructed as above, is shown to be multisymplectic,
  while another not constructed in this way is shown to be dissipative
  and non-multisymplectic.

\item \Cref{sec:global_hamiltonian} connects the local Hamiltonian
  framework of the previous section to the global Hamiltonian approach
  considered in other work. Under hypotheses satisfied by all the
  methods we consider, we show that the multisymplectic
  semidiscretization methods yield a Hamiltonian system of ODEs, with
  respect to some global ``discrete Hamiltonian.''

\item \Cref{sec:time_integrators} discusses the time-discretization of
  the semidiscretized ODEs via symplectic Runge--Kutta and partitioned
  Runge--Kutta methods. We prove that the resulting fully-discrete
  system satisfies a discrete multisymplectic conservation law in
  space and time. As an example, we discuss the application of the
  St\"ormer/Verlet ``leapfrog'' method to the semilinear Hodge wave
  equation, which requires only a linear variational solver---even in
  the presence of nonlinear source terms.

\item Finally, \cref{sec:numerical_examples} illustrates the foregoing
  methods and theory with numerical examples for the semilinear Hodge
  wave equation.
  
\end{itemize} 

\subsection{Acknowledgments} Ari Stern acknowledges the support of the
National Science Foundation (DMS-2208551) and the Simons Foundation
(SFI-MPS-TSM-00014348).

\section{Canonical time-dependent Hamiltonian PDEs}
\label{sec:hamiltonian}

\subsection{Background: the stationary case}
\label{sec:background}

Before developing the time-dependent case of canonical Hamiltonian
PDEs, we briefly recall the stationary case from \citet[Section
2]{StZa2025}. Here, and throughout the paper, we fix a spatial domain
$ \Omega \subset \mathbb{R}^n $ equipped with the Euclidean metric.

\subsubsection{Exterior algebra}
Let $ \operatorname{Alt} ^k \mathbb{R}^n $ denote the space of
alternating $k$-linear forms
$ \mathbb{R}^n \times \cdots \times \mathbb{R}^n \rightarrow
\mathbb{R} $, and let
$ \operatorname{Alt} \mathbb{R}^n \coloneqq \bigoplus _{ k = 0 } ^n
\operatorname{Alt} ^k \mathbb{R}^n $. The \emph{wedge product} (or
\emph{exterior product})
$ \wedge \colon \operatorname{Alt} ^k \mathbb{R}^n \times
\operatorname{Alt} ^\ell \mathbb{R}^n \rightarrow \operatorname{Alt}
^{ k + \ell } \mathbb{R}^n $ gives
$ ( \operatorname{Alt} \mathbb{R}^n , \wedge ) $ the structure of an
associative algebra, called the \emph{exterior algebra} on
$\mathbb{R}^n$. The Euclidean inner product on $\mathbb{R}^n$ induces
an inner product $ ( \cdot , \cdot ) $ on
$ \operatorname{Alt} \mathbb{R}^n $, and we denote the Euclidean
volume form (i.e., the determinant) by
$ \mathrm{vol} \in \operatorname{Alt} ^n \mathbb{R}^n $. The
\emph{Hodge star} operator
$ \star \colon \operatorname{Alt} ^k \mathbb{R}^n \rightarrow
\operatorname{Alt} ^{ n - k } \mathbb{R}^n $ is defined by the
condition
\begin{equation*}
  v \wedge \star w = ( v, w ) \,\mathrm{vol}, \qquad v, w \in \operatorname{Alt} ^k \mathbb{R}^n ,
\end{equation*}
which implies that $ \star $ is an isometric automorphism on
$ \operatorname{Alt} \mathbb{R}^n $.

\subsubsection{Exterior calculus}
Next, denote by $ \Lambda ^k (\Omega) $ the space of smooth
differential $k$-forms on $\Omega$ and let
$ \Lambda (\Omega) \coloneqq \bigoplus _{ k = 0 } ^n \Lambda ^k
(\Omega) $. These consist of smooth maps
$ \Omega \rightarrow \operatorname{Alt} ^k \mathbb{R}^n $ and
$ \Omega \rightarrow \operatorname{Alt} \mathbb{R}^n $,
respectively. The wedge product and Hodge star extend to
$ \Lambda (\Omega) $ by applying them pointwise at each
$ x \in \Omega $. The \emph{exterior differential}
$ \mathrm{d} ^k \colon \Lambda ^k (\Omega) \rightarrow \Lambda ^{ k +
  1 } (\Omega) $ and \emph{codifferential}
$ \delta ^k \coloneqq ( - 1 ) ^k \star ^{-1} \mathrm{d} ^{ n - k }
\star \colon \Lambda ^k (\Omega) \rightarrow \Lambda ^{ k -1 }
(\Omega) $ extend to operators
$ \mathrm{d} \coloneqq \bigoplus _{ k = 0 } ^n \mathrm{d} ^k $ and
$ \delta \coloneqq \bigoplus _{ k = 0 } ^n \delta ^k $ on
$ \Lambda (\Omega) $, where $ \mathrm{d} $ is $ ( +1) $-graded and
$ \delta $ is $ ( -1) $-graded. The exterior differential and
codifferential satisfy the important identity
\begin{equation}
  \label{eq:d_delta_identity}
  \mathrm{d} ( \tau \wedge \star v ) = \mathrm{d} \tau \wedge \star v - \tau \wedge \star \delta v , \qquad \tau \in \Lambda ^{ k -1 } (\Omega) ,\ v \in \Lambda ^k (\Omega) .
\end{equation}
Finally, we define the \emph{Hodge--Dirac operator}
$ \mathrm{D} \coloneqq \mathrm{d} + \delta $ on $ \Lambda (\Omega)
$. Since $ \mathrm{d} \mathrm{d} = 0 $ and $ \delta \delta = 0 $, the
square of the Hodge--Dirac operator is
$ \mathrm{D} ^2 = \mathrm{d} \delta + \delta \mathrm{d} $,
which is the \emph{Hodge--Laplace operator}.

\subsubsection{Canonical Hamiltonian PDEs} A canonical Hamiltonian
system of PDEs on $\Omega$ has the form
\begin{equation}
  \label{eq:hamiltonian_z_space}
  \mathrm{D} z = \frac{ \partial H }{ \partial z } ,
\end{equation}
where $ z \in \Lambda (\Omega) $ is unknown and
$ H \colon \Omega \times \operatorname{Alt} \mathbb{R}^n \rightarrow
\mathbb{R} $ is a given function called the \emph{Hamiltonian} of the
system. We can write this in terms of the individual components
$ z ^k \in \Lambda ^k (\Omega) $ as
\begin{equation*}
  \begin{bmatrix}
    & \delta ^0 \\
    \mathrm{d} ^0 & & \ddots \\
    & \ddots & & \delta ^n \\
    & & \mathrm{d} ^{ n -1 } 
  \end{bmatrix}
  \begin{bmatrix}
    z ^0 \\
    z ^1 \\
    \vdots \\
    z ^n 
  \end{bmatrix} =
  \begin{bmatrix}
    \partial H / \partial z ^0 \\
    \partial H / \partial z ^1 \\
    \vdots \\
    \partial H / \partial z ^n 
  \end{bmatrix},
\end{equation*}
where the matrix on the left-hand side corresponds to $ \mathrm{D}
$. This formalism is due to \citet{Bridges2006}.

\begin{example}
  If $ u \in \Lambda ^k (\Omega) $ is a solution to the semilinear
  Hodge--Laplace problem
  \begin{equation*}
    \mathrm{D} ^2 u = \frac{ \partial F }{ \partial u } ,
  \end{equation*}
  for some given potential function
  $ F \colon \Omega \times \operatorname{Alt} ^k \mathbb{R}^n
  \rightarrow \mathbb{R} $, then we can write this in the first-order
  form
  \begin{align*}
    \delta u &= \sigma ,\\
    \mathrm{d} \sigma + \delta \rho &= \frac{ \partial F }{ \partial u } ,\\
    \mathrm{d} u &= \rho .
  \end{align*}
  This says that $ z = \sigma \oplus u \oplus \rho
  $ is a solution to \eqref{eq:hamiltonian_z_space} with
  $ H ( x, z ) = \frac{1}{2} \lvert \sigma \rvert ^2 + F ( x, u ) +
  \frac{1}{2} \lvert \rho \rvert ^2 $.
\end{example}

\subsubsection{The multisymplectic conservation law} The
\emph{canonical multisymplectic $2$-form} on
$ \operatorname{Alt} \mathbb{R}^n $ is denoted
$ \omega \colon \operatorname{Alt} \mathbb{R}^n \times
\operatorname{Alt} \mathbb{R}^n \rightarrow \operatorname{Alt} ^{ n -1
} \mathbb{R}^n $ and defined by
\begin{equation}
  \label{eq:ms_form}
  \omega ( w _1 , w _2 ) \coloneqq \sum _{ k = 1 } ^n ( w _1 ^{ k -1 } \wedge \star w _2 ^k - w _2 ^{ k -1 } \wedge \star w _1 ^k ) .
\end{equation}
For $ w _1, w _2 \in \Lambda (\Omega) $, the multisymplectic form is
related to the Hodge--Dirac operator by the identity
\begin{equation*}
  \mathrm{d} \omega ( w _1, w _2 ) =   \bigl( ( \mathrm{D} w _1 , w _2 ) - ( w _1 , \mathrm{D} w _2 ) \bigr) \mathrm{vol}
\end{equation*}
(\citet[Proposition 2.5]{Bridges2006}, \citet[Equation 8]{StZa2025}),
which can be written equivalently as
\begin{equation}
  \label{eq:div_omega_D}
  \operatorname{div} \omega ( w _1 , w _2 ) = ( \mathrm{D} w _1 , w _2 ) - ( w _1 , \mathrm{D} w _2 ) .
\end{equation}
In particular, suppose $ w _1 , w _2 $ are \emph{first variations} of
a solution $z$ to \eqref{eq:hamiltonian_z_space}, meaning that each is a
solution to the linearized equation
\begin{equation*}
  \mathrm{D} w _i = \frac{ \partial ^2 H }{ \partial z ^2 } w _i ,
\end{equation*}
for $ i = 1 , 2 $. Then the identity above, together with the symmetry
of the Hessian, implies
\begin{equation*}
  \operatorname{div} \omega ( w _1, w _2 ) = 0 ,
\end{equation*}
which is called the \emph{multisymplectic conservation law}.

\subsubsection{Boundary traces and the integral form of the
  multisymplectic conservation law} Let $ K \Subset \Omega $ be a
subdomain with boundary $ \partial K $, which we assume (for now) to
be smooth. Let $ \widehat{ \star } $ be the Hodge star on
$ \partial K $, with respect to the orientation induced by $K$ and the
metric induced by the Euclidean inner product on $\mathbb{R}^n$. The
\emph{tangential} and \emph{normal traces} of
$ w \in \Lambda (\Omega) $ on $ \partial K $ are defined by
\begin{equation*}
  w ^{\mathrm{tan}} \coloneqq \operatorname{tr} w \in \Lambda ( \partial K ) , \qquad w ^{\mathrm{nor}} \coloneqq \widehat{ \star } ^{-1} \operatorname{tr} \star w \in \Lambda ( \partial K ) ,
\end{equation*}
where $ \operatorname{tr} $ denotes pullback of differential forms by
the boundary inclusion $ \partial K \hookrightarrow \Omega $. Next,
let $ ( \cdot , \cdot ) _K $ be the $ L ^2 $ inner product on
$ \Lambda (K) $, defined by
$ ( w _1 , w _2 ) _K \coloneqq \sum _{ k = 0 } ^n \int _K w _1 ^k
\wedge \star w _2 ^k $, and similarly let
$ \langle \cdot , \cdot \rangle _{ \partial K } $ be the $ L ^2 $
inner product on $ \Lambda ( \partial K ) $ arising from the boundary
Hodge star $ \widehat{ \star } $. Using Stokes's theorem and the
identity \eqref{eq:d_delta_identity}, we obtain
\begin{equation}
  \label{eq:ibp}
  \langle w _1 ^{\mathrm{tan}} , w _2 ^{\mathrm{nor}} \rangle _{ \partial K } = ( \mathrm{d} w _1 , w _2 ) _K - ( w _1 , \delta w _2 ) _K ,
\end{equation}
which is the integration-by-parts formula for differential forms on
$K$.

Finally, define the antisymmetric bilinear form
$ [ \cdot , \cdot ] _{ \partial K } $ on $ \Lambda ( \partial K ) $ by
\begin{align*}
  [ w _1 , w _2 ] _{ \partial K }
  &\coloneqq \langle w _1 ^{\mathrm{tan}} , w _2 ^{\mathrm{nor}} \rangle _{ \partial K } - \langle w _2 ^{\mathrm{tan}} , w _1 ^{\mathrm{nor}} \rangle _{ \partial K } \\
  &= ( \mathrm{D} w _1 , w _2 ) _K - ( w _1 , \mathrm{D} w _2 ) _K \\
  &= \int _{ \partial K } \operatorname{tr} \omega ( w _1 , w _2 ) .
\end{align*}
It follows that the multisymplectic conservation law is equivalent to
the statement that
\begin{equation*}
  [ w _1 , w _2 ] _{ \partial K } = 0 ,
\end{equation*}
for all $ K \Subset \Omega $, when
$ w _1 , w _2 \in \Lambda (\Omega) $ are first variations of a
solution to \eqref{eq:hamiltonian_z_space}, cf.~\citep[Proposition
2.18]{StZa2025}. We call this the \emph{integral form of the
  multisymplectic conservation law}. In the special case of the
Hodge--Laplace problem, this expresses the symmetry of the
Dirichlet-to-Neumann operator mapping tangential boundary values to
normal boundary values (\citet[Equation 3.6]{BeSh2008}, \citet[Example
2.20]{StZa2025}).

\subsection{Canonical Hamiltonian PDEs in spacetime} To extend the
framework summarized in the previous section to time-dependent
problems, we replace $\Omega$ by $ I \times \Omega $, where $I$ is a
time interval, and equip $ I \times \Omega $ with the Minkowski metric
$ - \mathrm{d} t \otimes \mathrm{d} t + \mathrm{d} x ^1 \otimes
\mathrm{d} x ^1 + \cdots + \mathrm{d} x ^n \otimes \mathrm{d} x ^n
$. This induces an $ L ^2 $ pseudo-inner product on
$ \Lambda ( I \times \Omega ) $, which we denote by
$ ( \cdot , \cdot ) _{ I \times \Omega } $. Let
$ \overline{ \mathrm{d} } $, $ \overline{ \delta } $, and
$ \overline{ \mathrm{D} } $ denote the spacetime exterior
differential, codifferential, and Hodge--Dirac operators on
$ \Lambda ( I \times \Omega ) $. We continue to use $ \mathrm{d} $,
$ \delta $, and $ \mathrm{D} $ to denote those operators on
$ \Lambda (\Omega) $; in the spacetime setting, these are interpreted
as partial differential operators with respect to space.

Now, a canonical Hamiltonian system in spacetime has the form
\begin{equation}
  \label{eq:hamiltonian_z_spacetime}
  \overline{ \mathrm{D} } z = \frac{ \partial H }{ \partial z },
\end{equation}
where $ z \in \Lambda ( I \times \Omega ) $. To interpret this as a
time-evolution problem on $\Omega$, we write
$ z = q - \mathrm{d} t \wedge p $, where
$ q, p \colon I \rightarrow \Lambda (\Omega) $. The next result
expresses the spacetime operators $ \overline{ \mathrm{d} } $,
$ \overline{ \delta } $, and $ \overline{ \mathrm{D} } $ in terms of
the components $q$ and $p$, using ``dot'' notation for time
differentiation.

\begin{proposition}
  Let
  $ z = q - \mathrm{d} t \wedge p \in \Lambda ( I \times \Omega ) $,
  where $ q, p \colon I \rightarrow \Lambda (\Omega) $. Then:
  \begin{subequations}
    \begin{alignat}{2}
      \overline{ \mathrm{d} } z &= \mathrm{d} q &&+ \mathrm{d} t \wedge ( \dot{q} + \mathrm{d} p ) , \label{eq:spacetime_d} \\
      \overline{ \delta } z &= ( - \dot{p} + \delta q ) &&+ \mathrm{d} t \wedge \delta p , \label{eq:spacetime_del} \\
      \overline{ \mathrm{D} } z &= ( - \dot{p} + \mathrm{D} q ) &&+ \mathrm{d} t \wedge ( \dot{q} + \mathrm{D} p ) \label{eq:spacetime_D}.
    \end{alignat}
  \end{subequations}
\end{proposition}

\begin{proof}
  First, by the definition of the exterior differential and the
  Leibniz rule, we have
  \begin{align*}
    \overline{ \mathrm{d} } z
    &= \overline{ \mathrm{d} } q + \mathrm{d}t \wedge \overline{ \mathrm{d} } p \\
    &= ( \mathrm{d} t \wedge \dot{q} + \mathrm{d} q ) + \mathrm{d} t \wedge ( \mathrm{d} t \wedge \dot{p} + \mathrm{d} p ) \\
    &= \mathrm{d} q + \mathrm{d} t \wedge ( \dot{q} + \mathrm{d} p ) ,
  \end{align*}
  where the last step uses $ \mathrm{d} t \wedge \mathrm{d} t = 0 $ to
  cancel the term involving $ \dot{p} $. This gives
  \eqref{eq:spacetime_d}. Next, we recall that the exterior
  differential and codifferential are related by
  \begin{equation*}
    ( \overline{ \delta } z , \zeta ) _{ I \times \Omega } = ( z, \overline{ \mathrm{d} } \zeta ) _{ I \times \Omega } ,
  \end{equation*}
  for all $ \zeta $ smooth and compactly supported in
  $ I \times \Omega $. Writing
  $ \zeta = \phi - \mathrm{d} t \wedge \psi $ and using
  \eqref{eq:spacetime_d},
  \begin{align*}
    ( z , \overline{ \mathrm{d} } \zeta ) _{ I \times \Omega }
    &= \bigl( ( q - \mathrm{d} t \wedge p ) , \mathrm{d} \phi + \mathrm{d} t \wedge ( \dot{ \phi } + \mathrm{d} \psi ) \bigr) _{ I \times \Omega } \\
    &= ( q, \mathrm{d} \phi ) _{ I \times \Omega } + ( p, \dot{ \phi } + \mathrm{d} \psi ) _{ I \times \Omega } \\
    &= ( - \dot{p} + \delta q , \phi ) _{ I \times \Omega } + ( \delta p, \psi ) _{ I \times \Omega } \\
    &= \bigl( ( - \dot{p} + \delta q ) + \mathrm{d} t \wedge \delta p , \zeta \bigr) _{ I \times \Omega } ,
  \end{align*}
  where the third line employs the integration-by-parts identities
  \begin{equation*}
    ( q, \mathrm{d} \phi ) _{ I \times \Omega } = ( \delta q , \phi ) _{ I \times \Omega } , \qquad 
    ( p, \dot{ \phi } ) _{ I \times \Omega } = ( - \dot{p} , \phi ) _{ I \times \Omega } , \qquad ( p , \mathrm{d} \psi ) _{ I \times \Omega } = ( \delta p , \psi ) _{ I \times \Omega } .
  \end{equation*} 
  This gives \eqref{eq:spacetime_del}. Finally, adding
  \eqref{eq:spacetime_d} and \eqref{eq:spacetime_del} immediately gives
  \eqref{eq:spacetime_D}.
\end{proof}

\begin{corollary}
  The canonical Hamiltonian system \eqref{eq:hamiltonian_z_spacetime}
  for $ z = q - \mathrm{d} t \wedge p $ is equivalent to
  \begin{subequations}
    \label{eq:hamiltonian_qp}
    \begin{align}
      \dot{q} + \mathrm{D} p &= \frac{ \partial H }{ \partial p } \label{eq:hamiltonian_qdot} , \\
      -\dot{p} + \mathrm{D} q &= \frac{ \partial H }{ \partial q } \label{eq:hamiltonian_pdot} .
    \end{align}
  \end{subequations}
\end{corollary}

\begin{proof}
  For any $ w = s - \mathrm{d} t \wedge r \in \Lambda ( I \times \Omega ) $, observe that
  \begin{equation*}
    \biggl( \frac{ \partial H }{ \partial z } , w \biggr) _{ I \times \Omega } = \biggl( \frac{ \partial H }{ \partial q } , s \biggr) _\Omega + \biggl( \frac{ \partial H }{ \partial p } , r \biggr) _\Omega = \biggl( \frac{ \partial H }{ \partial q } + \mathrm{d} t \wedge \frac{ \partial H }{ \partial p } , w \biggr) _{ I \times \Omega } .
  \end{equation*}
  Hence,
  $ \partial H / \partial z = \partial H / \partial q + \mathrm{d} t
  \wedge \partial H / \partial p $. Setting this equal to
  $ \overline{ \mathrm{D} } z $ using \eqref{eq:spacetime_D} gives
  \eqref{eq:hamiltonian_qp}.
\end{proof}

\begin{remark}
  We mention two important special cases of
  \eqref{eq:hamiltonian_qp}. First, when $ n = 0 $, we recover the
  canonical Hamiltonian system of ODEs,
  \begin{align*}
    \dot{q} &= \frac{ \partial H }{ \partial p } ,\\
    - \dot{p} &= \frac{ \partial H }{ \partial q } .
  \end{align*}
  Second, for arbitrary $n$, a stationary solution of
  \eqref{eq:hamiltonian_qp} satisfies
  \begin{subequations}
    \label{eq:stationary_qp}
    \begin{align}
      \mathrm{D} p &= \frac{ \partial H }{ \partial p } , \\
      \mathrm{D} q &= \frac{ \partial H }{ \partial q }.
    \end{align}
  \end{subequations}
  If $H$ is \emph{separable}, meaning that it is a sum of functions
  depending only on $q$ and only on $p$, then this decouples into two
  stationary Hamiltonian PDEs on $ \Lambda (\Omega) $. Alternatively,
  we can view \eqref{eq:stationary_qp} (even in the non-separable
  case) as a stationary Hamiltonian system on
  $ \Lambda (\Omega) \otimes \mathbb{R}^2 $, via a straightforward
  generalization to vector-valued differential forms of the framework
  discussed in \cref{sec:background}.
\end{remark}

We next introduce a useful notation that allows us to express
\eqref{eq:hamiltonian_qp} as a single equation and simplifies many of
the calculations to follow. Let
$ \mathbf{\Lambda} (\Omega) \coloneqq \Lambda (\Omega) \otimes
\mathbb{R}^2 $ be the space of $ \mathbb{R}^2 $-valued differential
forms on $\Omega$, and identify
$ z = q - \mathrm{d} t \wedge p \in \Lambda ( I \times \Omega ) $ with
$ \mathbf{z} =
\begin{bsmallmatrix}
  q \\ p 
\end{bsmallmatrix} \colon I \rightarrow \mathbf{\Lambda} (\Omega) $. We equip $ \mathbb{R}^2  $ with the Euclidean inner product and the canonical symplectic structure given by the symplectic matrix $ J =
\begin{bsmallmatrix}
  0 & -1 \\
  1 & 0 
\end{bsmallmatrix} $, and we define $ \mathbf{J} \coloneqq \Lambda (\Omega) \otimes J $, i.e., $ \mathbf{J} \mathbf{z} =
\begin{bsmallmatrix}
  - p \\ q 
\end{bsmallmatrix} $. Finally, let $ \mathbf{D} \coloneqq \mathrm{D} \otimes \mathbb{R}^2 $, i.e., $ \mathbf{D} \mathbf{z} =
\begin{bsmallmatrix}
  \mathrm{D} q \\ \mathrm{D} p 
\end{bsmallmatrix} $, which is the Hodge--Dirac operator on $ \mathbf{\Lambda} (\Omega) $. With this notation, \eqref{eq:hamiltonian_qp} becomes
\begin{equation}
  \label{eq:hamiltonian_JD}
  \mathbf{J} \dot{\mathbf{z}} + \mathbf{D} \mathbf{z} = \frac{ \partial H }{ \partial \mathbf{z} } .
\end{equation}
If we denote
$ \operatorname{\mathbf{Alt}} \mathbb{R}^n \coloneqq
\operatorname{Alt} \mathbb{R}^n \otimes \mathbb{R}^2 $, then we may
view the Hamiltonian as being a function
$ H \colon I \times \Omega \times \operatorname{\mathbf{Alt}}
\mathbb{R}^n \rightarrow \mathbb{R} $.

\begin{remark}
  This can be generalized further by replacing $ \mathbb{R}^2 $ with
  an arbitrary symplectic vector space. The $ n = 0 $ case then
  recovers Hamiltonian mechanics on symplectic vector spaces more
  generally, cf.~\citet[Chapter 2]{MaRa1999}.
\end{remark}

\begin{example}
  When $ H = 0 $, the system \eqref{eq:hamiltonian_qp} becomes
  \begin{align*}
    \dot{q} + \mathrm{D} p &= 0 ,\\
    - \dot{p} + \mathrm{D} q &= 0 .
  \end{align*}
  This can be seen as a generalization of the homogeneous Maxwell's
  equations involving forms of all degrees. Differentiating both
  equations in time and substituting gives
  \begin{equation*}
    \ddot{q} + \mathrm{D} ^2 q = 0 , \qquad \ddot{p} + \mathrm{D} ^2 p = 0 ,
  \end{equation*}
  which says $q$ and $p$ each satisfy the homogeneous \emph{Hodge wave
    equation}.
\end{example}

\begin{example}
  \label{ex:hodge_wave}
  Suppose $ u \colon I \rightarrow \Lambda ^k (\Omega) $ is a solution
  to the $k$-form semilinear Hodge wave equation,
  \begin{equation*}
    \ddot{u} + \mathrm{D} ^2 u = -\frac{ \partial F }{ \partial u } ,
  \end{equation*}
  for some potential
  $ F \colon I \times \Omega \times \operatorname{Alt} ^k \mathbb{R}^n
  \rightarrow \mathbb{R} $. Introducing variables $ p = \dot{u} $,
  $ \sigma = - \delta u $, and $ \rho = -\mathrm{d} u $ implies
  that we have a solution to the first-order system
  \begin{subequations}
    \label{eq:hodge_wave}
    \begin{align}
      \dot{ \sigma } + \delta p &= 0 , \label{eq:hodge_wave_sigmadot}\\
      \dot{u} &= p , \label{eq:hodge_wave_udot} \\
      \dot{ \rho } + \mathrm{d} p &= 0 , \label{eq:hodge_wave_rhodot}\\
      \delta u &= - \sigma , \label{eq:hodge_wave_delu} \\
      - \dot{p} + \mathrm{d} \sigma + \delta \rho &= \frac{ \partial F }{ \partial u } , \label{eq:hodge_wave_pdot} \\
      \mathrm{d} u &= - \rho . \label{eq:hodge_wave_du}
    \end{align}
  \end{subequations}
  Letting $ q = \sigma \oplus u \oplus \rho $, this says that $ \mathbf{z} =
  \begin{bsmallmatrix}
    q \\ p 
  \end{bsmallmatrix}
  $ solves the canonical Hamiltonian system of PDEs with
  $ H ( t, x , \mathbf{z} ) = - \frac{1}{2} \lvert \sigma \rvert ^2 +
  \bigl( \frac{1}{2} \lvert p \rvert ^2 + F ( t, x, u ) \bigr) -
  \frac{1}{2} \lvert \rho \rvert ^2 $.

  Furthermore, if the constraints \eqref{eq:hodge_wave_delu} and
  \eqref{eq:hodge_wave_du} hold at the initial time, then
  \eqref{eq:hodge_wave_sigmadot}--\eqref{eq:hodge_wave_rhodot}
  ensure that these constraints are preserved at all subsequent
  times. Hence, we can eliminate the two constraints and simply evolve
  the remaining four equations
  \begin{align*}
    \dot{ \sigma } + \delta p &= 0 ,\\
    \dot{u} &= p , \\
    \dot{ \rho } + \mathrm{d} p &= 0 ,\\
    - \dot{p} + \mathrm{d} \sigma + \delta \rho &= \frac{ \partial F }{ \partial u } .
  \end{align*}
  As another way to see why it suffices to evolve these components
  only, consider the subspace
  \begin{equation*}
    \mathbf{S} \coloneqq \Biggl\{
    \begin{bmatrix}
      -\delta u \oplus u \oplus -\mathrm{d} u \\ p
    \end{bmatrix} : u , p \in \Lambda ^k (\Omega) \Biggr\} \subset \mathbf{\Lambda} (\Omega) .
  \end{equation*}
  It is straightforward to check that $ \mathbf{z} \in \mathbf{S} $
  implies
  $ \partial H / \partial \mathbf{z} - \mathbf{D} \mathbf{z} \in
  \mathbf{J} \mathbf{S} $ for the Hamiltonian defined above, and
  therefore $\mathbf{S}$ is an invariant subspace of
  \eqref{eq:hamiltonian_JD}.

  Finally, in the linear case
  $ F ( t, x , u ) = \bigl( f ( t, x ) , u \bigr) $ for some
  $ f \colon I \rightarrow \Lambda ^k (\Omega) $, we can simply evolve
  \begin{align*}
    \dot{ \sigma } + \delta p &= 0 ,\\
    - \dot{p} + \mathrm{d} \sigma + \delta \rho &= f ,\\
    \dot{ \rho } + \mathrm{d} p &= 0 ,
  \end{align*}
  after which $u$ may be obtained (if desired) by integrating $p$ over
  time. This recovers the approach in \citet[\S8.5]{Arnold2018}, who
  observes that writing the system above in matrix form,
  \begin{equation*}
    \partial _t 
    \begin{bmatrix}
      \sigma \\
      p \\
      \rho 
    \end{bmatrix}
    =
    \begin{bmatrix}
      0 & -\delta & 0 \\
      \mathrm{d} & 0 & \delta \\
        0 & -\mathrm{d} & 0 
    \end{bmatrix}
    \begin{bmatrix}
      \sigma \\
      p \\
      \rho 
    \end{bmatrix}
    -
    \begin{bmatrix}
      0 \\
      f \\
      0
    \end{bmatrix},
  \end{equation*}
  reveals its symmetric hyperbolic structure. For the scalar wave
  equation when $ k = 0 $ or $ k = n $, this recovers a first-order
  mixed formulation that appears, e.g., in \citet{MoPe2018}.
\end{example}

\subsection{The multisymplectic conservation law for time-dependent systems}

Recall the canonical multisymplectic $2$-form
$ \omega \colon \operatorname{Alt} \mathbb{R}^n \times
\operatorname{Alt} \mathbb{R}^n \rightarrow \operatorname{Alt} ^{ n -1
} \mathbb{R}^n $ from \cref{sec:background}. To extend this to the
$ \mathbb{R}^2 $-valued forms we have just introduced, we let
$ \mathbf{w} _i =
\begin{bsmallmatrix}
  s _i \\
  r _i 
\end{bsmallmatrix} \in \operatorname{\mathbf{Alt}} \mathbb{R}^n  $ for $ i = 1, 2 $, and define
$ \boldsymbol{\omega} \colon \operatorname{\mathbf{Alt}} \mathbb{R}^n
\times \operatorname{\mathbf{Alt}} \mathbb{R}^n \rightarrow
\operatorname{Alt} ^{ n -1 } \mathbb{R}^n $ by
\begin{equation*}
  \boldsymbol{\omega} ( \mathbf{w} _1 , \mathbf{w} _2 ) \coloneqq \omega ( s _1, s _2 ) + \omega ( r _1 , r _2 ) .
\end{equation*}
It follows immediately from \eqref{eq:div_omega_D} and the foregoing
definitions that, for
$ \mathbf{w} _1, \mathbf{w} _2 \in \mathbf{\Lambda} (\Omega) $, we
have
\begin{equation}
  \label{eq:div_omega_D_bold}
  \operatorname{div} \boldsymbol{\omega} ( \mathbf{w} _1 , \mathbf{w} _2 ) = ( \mathbf{D} \mathbf{w} _1 , \mathbf{w} _2 ) - ( \mathbf{w} _1 , \mathbf{D} \mathbf{w} _2 ) .
\end{equation}

\begin{definition}
  Let $ \mathbf{z} \colon I \rightarrow \mathbf{\Lambda} (\Omega) $ be
  a solution to \eqref{eq:hamiltonian_JD}. A \emph{first variation} of
  $\mathbf{z}$ is a solution
  $ \mathbf{w} _i \colon I \rightarrow \mathbf{\Lambda} (\Omega) $ to
  the linearized equation
  \begin{equation}
    \label{eq:variational_JD}
    \mathbf{J} \dot{\mathbf{w}} _i + \mathbf{D} \mathbf{w} _i = \frac{ \partial ^2 H }{ \partial \mathbf{z} ^2 } \mathbf{w} _i ,
  \end{equation}
  called the \emph{variational equation} of \eqref{eq:hamiltonian_JD}
  at $\mathbf{z}$.
\end{definition}

\begin{theorem}
  \label{thm:mscl_differential}
  If $ \mathbf{w} _1 , \mathbf{w} _2 $ are first variations of a
  solution to \eqref{eq:hamiltonian_JD}, then they satisfy
  \begin{equation}
    \label{eq:mscl_differential}
    \partial _t ( \mathbf{J} \mathbf{w} _1 , \mathbf{w} _2 ) + \operatorname{div} \boldsymbol{\omega} ( \mathbf{w} _1 , \mathbf{w} _2 ) = 0 ,
  \end{equation}
  which we call the \emph{multisymplectic conservation law}.  
\end{theorem}

\begin{proof}
  From \eqref{eq:variational_JD}, we have
  \begin{align*}
    ( \mathbf{J} \dot{\mathbf{w}} _1 , \mathbf{w} _2 ) + ( \mathbf{D} \mathbf{w} _1 , \mathbf{w} _2 ) &= \biggl( \frac{ \partial ^2 H }{ \partial \mathbf{z} ^2 } \mathbf{w} _1, \mathbf{w} _2 \biggl), \\
    ( \mathbf{w} _1 , \mathbf{J} \dot{\mathbf{w}} _2 ) + ( \mathbf{w} _1 , \mathbf{D} \mathbf{w} _2 ) &= \biggl( \mathbf{w} _1 , \frac{ \partial ^2 H }{ \partial \mathbf{z} ^2 } \mathbf{w} _2 \biggr) .
  \end{align*}
  The right-hand sides are equal by the symmetry of the Hessian, so
  subtracting gives
  \begin{equation*}
    \bigl[ ( \mathbf{J} \dot{\mathbf{w}} _1 , \mathbf{w} _2 ) - ( \mathbf{w} _1 , \mathbf{J} \dot{\mathbf{w}} _2 ) \bigr] + \bigl[ ( \mathbf{D} \mathbf{w} _1 , \mathbf{w} _2 ) - ( \mathbf{w} _1 , \mathbf{D} \mathbf{w} _2 ) \bigr] = 0 .
  \end{equation*}
  The first term in brackets equals
  $ \partial _t ( \mathbf{J} \mathbf{w} _1 , \mathbf{w} _2 ) $ by the
  Leibniz rule and the antisymmetry of $ \mathbf{J} $, and the second
  term in brackets equals
  $ \operatorname{div} \boldsymbol{\omega} ( \mathbf{w} _1 ,
  \mathbf{w} _2 ) $ by \eqref{eq:div_omega_D_bold}.
\end{proof}

\begin{remark}
  The multisymplectic conservation law can be equivalently obtained in
  the spacetime setting by viewing
  $ w _i = s _i - \mathrm{d} t \wedge r _i $ as a first variation of
  $ z = q - \mathrm{d} t \wedge p $ in
  $ \Lambda ( I \times \Omega ) $. Using \eqref{eq:spacetime_D}, it
  can be seen that \eqref{eq:mscl_differential} is equivalent to
  $ ( \overline{ \mathrm{D} } w _1 , w _2 ) - ( w _1 , \overline{
    \mathrm{D} } w _2 ) = 0 $; the left-hand side may be interpreted
  as a spacetime divergence with respect to the Minkowski metric.
\end{remark}

There is also an integral form of the multisymplectic conservation law
on any spatial subdomain $ K \Subset \Omega $. Assuming (for now) that
$ \partial K $ is smooth, the tangential and normal trace may be
extended to $ \mathbb{R}^2 $-valued forms in the natural way: if $ \mathbf{w} =
\begin{bsmallmatrix}
  s \\ r
\end{bsmallmatrix} \in \mathbf{\Lambda} (\Omega) $, then
\begin{equation*}
  \mathbf{w} ^{\mathrm{tan}} \coloneqq
  \Bigl[ \begin{smallmatrix}
    s ^{\mathrm{tan}} \\
    r ^{\mathrm{tan}} 
  \end{smallmatrix} \Bigr]  \in \mathbf{\Lambda} ( \partial K ) , \qquad \mathbf{w} ^{\mathrm{nor}} \coloneqq
  \Bigl[\begin{smallmatrix}
    s ^{\mathrm{nor}} \\
    r ^{\mathrm{nor}} 
  \end{smallmatrix}\Bigr] \in \mathbf{\Lambda} ( \partial K ) .
\end{equation*}
We may also use the Euclidean inner product on $\mathbb{R}^2$ to
extend the $ L ^2 $ inner products $ ( \cdot , \cdot ) _K $ and
$ \langle \cdot , \cdot \rangle _{ \partial K } $ to
$ \mathbf{\Lambda} (K) $ and $ \mathbf{\Lambda} ( \partial K ) $,
respectively, obtaining the integration-by-parts identity
\begin{equation*}
  \langle \mathbf{w} _1 ^{\mathrm{tan}} , \mathbf{w} _2 ^{\mathrm{nor}} \rangle _{ \partial K } = ( \mathbf{d} \mathbf{w} _1 , \mathbf{w} _2 ) _K - ( \mathbf{w} _1 , \boldsymbol{\delta} \mathbf{w} _2 ) _K .
\end{equation*}
Finally, we extend the antisymmetric bilinear form
$ [ \cdot , \cdot ] _{ \partial K } $ to
$ \mathbf{\Lambda} ( \partial K ) $, defining
\begin{subequations}
  \label{eq:bracket_identities}
  \begin{align}
    [ \mathbf{w} _1 , \mathbf{w} _2 ] _{ \partial K }
    &\coloneqq \langle \mathbf{w} _1 ^{\mathrm{tan}} , \mathbf{w} _2 ^{\mathrm{nor}} \rangle _{ \partial K } - \langle \mathbf{w} _2 ^{\mathrm{tan}} , \mathbf{w} _1 ^{\mathrm{nor}} \rangle _{ \partial K } \label{eq:bracket} \\
    &= ( \mathbf{D} \mathbf{w} _1 , \mathbf{w} _2 ) _K - ( \mathbf{w} _1 , \mathbf{D} \mathbf{w} _2 ) _K \label{eq:bracket_D} \\
    &= \int _{ \partial K } \boldsymbol{\omega} ( \mathbf{w} _1 , \mathbf{w} _2 ) .
  \end{align}
\end{subequations}
Hence, \eqref{eq:mscl_differential} is equivalent to
\begin{equation}
  \label{eq:mscl_integral}
  \frac{\mathrm{d}}{\mathrm{d}t} ( \mathbf{J} \mathbf{w} _1 , \mathbf{w} _2 ) _K + [ \mathbf{w} _1 , \mathbf{w} _2 ] _{ \partial K } = 0 ,
\end{equation}
for all $ K \Subset \Omega $, which we call the \emph{integral form of
  the multisymplectic conservation law}.

\begin{example}
  \label{ex:hodge_wave_integral_mscl}
  Let us return to the semilinear Hodge wave equation, discussed in
  \cref{ex:hodge_wave}, to see how the multisymplectic conservation
  law manifests. If $ \mathbf{z} =
  \begin{bsmallmatrix}
    q \\ p 
  \end{bsmallmatrix} $ with $ q = \sigma \oplus u \oplus \rho $ satisfies \eqref{eq:hodge_wave}, then first variations $ \mathbf{w} _i =
  \begin{bsmallmatrix}
    s _i \\ r _i 
  \end{bsmallmatrix} $ with $ s _i = \tau _i \oplus v _i \oplus \eta _i $ are solutions to the linearized system
  \begin{subequations}
    \label{eq:hodge_wave_var}
    \begin{align}
      \dot{ \tau } _i + \delta r _i &= 0 , \label{eq:hodge_wave_var_taudot}\\
      \dot{v} _i &= r _i , \label{eq:hodge_wave_var_vdot} \\
      \dot{ \eta } _i + \mathrm{d} r _i &= 0 , \label{eq:hodge_wave_var_etadot}\\
      \delta v _i &= - \tau _i , \label{eq:hodge_wave_var_delv} \\
      - \dot{r} _i + \mathrm{d} \tau _i + \delta \eta _i &= \frac{ \partial ^2 F }{ \partial u ^2 } v _i , \label{eq:hodge_wave_var_rdot} \\
      \mathrm{d} v _i &= - \eta _i . \label{eq:hodge_wave_var_dv}
    \end{align}
  \end{subequations}
  In terms of these components, we have
  \begin{align*}
    ( \mathbf{J} \mathbf{w} _1 , \mathbf{w} _2 ) &= ( v _1 , r _2 ) - ( v _2 , r _1 ) ,\\
   \boldsymbol{\omega} ( \mathbf{w} _1 , \mathbf{w} _2 ) &= ( \tau _1 \wedge \star v _2 - \tau _2 \wedge \star v _1 ) + ( v _1 \wedge \star \eta _2 - v _2 \wedge \star \eta _1 ) .
  \end{align*}
  Note that $ r _i $ does not appear on the second line: since it is
  nonvanishing only at degree $k$, we have
  $ \omega ( r _1 , r _2 ) = 0 $ by \eqref{eq:ms_form}. Hence, the
  multisymplectic conservation law \eqref{eq:mscl_differential} can be
  written as
  \begin{equation*}
    \partial _t ( v _1 , r _2 ) + \operatorname{div} ( \tau _1 \wedge \star v _2 + v _1 \wedge \star \eta _2 ) = \partial _t ( v _2 , r _1 ) + \operatorname{div} ( \tau _2 \wedge \star v _1 + v _2 \wedge \star \eta _1 ).
  \end{equation*}
  Likewise, for $ K \Subset \Omega $ we have
  \begin{equation*}
    [ \mathbf{w} _1 , \mathbf{w} _2 ] _{ \partial K } = \Bigl( \langle \tau _1 ^{\mathrm{tan}} , v _2 ^{\mathrm{nor}} \rangle _{ \partial K } - \langle \tau _2 ^{\mathrm{tan}} , v _1 ^{\mathrm{nor}} \rangle _{ \partial K } \Bigr) + \Bigl( \langle v _1 ^{\mathrm{tan}} , \eta _2 ^{\mathrm{nor}} \rangle _{ \partial K } - \langle v _2 ^{\mathrm{tan}} , \eta _1 ^{\mathrm{nor}} \rangle _{ \partial K } \Bigr),
  \end{equation*}
  so the integral form of the multisymplectic conservation law
  \eqref{eq:mscl_integral} can be written as
  \begin{equation*}
    \frac{\mathrm{d}}{\mathrm{d}t} ( v _1 , r _2 ) _K + \langle \tau _1 ^{\mathrm{tan}} , v _2 ^{\mathrm{nor}} \rangle _{ \partial K } + \langle v _1 ^{\mathrm{tan}} , \eta _2 ^{\mathrm{nor}} \rangle _{ \partial K } = \frac{\mathrm{d}}{\mathrm{d}t} ( v _2 , r _1 ) _K + \langle \tau _2 ^{\mathrm{tan}} , v _1 ^{\mathrm{nor}} \rangle _{ \partial K } + \langle v _2 ^{\mathrm{tan}} , \eta _1 ^{\mathrm{nor}} \rangle _{ \partial K } .
  \end{equation*}
  The stationary case recovers the Dirichlet-to-Neumann operator
  symmetry of \citep[Example 2.20]{StZa2025}.
\end{example}

\section{Multisymplectic semidiscretization}
\label{sec:semidiscrete}

\subsection{Hybrid FEEC methods}

We now present a framework for semidiscretizing canonical systems of
PDEs using hybrid FEEC methods, employing essentially the approach of
\citet{StZa2025} for the stationary case. The preliminaries will be
presented fairly quickly, and we refer the reader to \citep{StZa2025}
and references therein for a more detailed account of the spatial
discretization ingredients.

Let $ \Omega \subset \mathbb{R}^n $ be a bounded Lipschitz domain, and
let $ \mathcal{T} _h $ be a partition of $\Omega$ into non-overlapping
Lipschitz subdomains $ K \in \mathcal{T} _h $ (e.g., a simplicial
triangulation). We denote the Sobolev-like spaces
\begin{equation*} 
  H \Lambda (K) \coloneqq \bigl\{ w \in L ^2 \Lambda (K) : \mathrm{d} w \in L ^2 \Lambda (K) \bigr\} , \qquad H ^\ast \Lambda (K) \coloneqq \bigl\{ w \in L ^2 \Lambda (K) : \delta w \in L ^2 \Lambda (K) \bigr\} ,
\end{equation*} 
where $ \mathrm{d} $ and $ \delta $ are taken in the sense of
distributions. It follows that the Hodge--Dirac operator
$ \mathrm{D} $ maps
$ H \Lambda (K) \cap H ^\ast \Lambda (K) \rightarrow L ^2 \Lambda (K)
$. As in the previous section, the ``bold'' spaces
$ \mathbf{L} ^2 \mathbf{\Lambda} (K) $, $ \mathbf{H \Lambda } (K) $, and
$ \mathbf{H ^\ast \Lambda } (K) $ are defined by taking tensor
products with $ \mathbb{R}^2 $.

\citet{Weck2004} showed that it is possible to define a weak
tangential trace of $ w _1 \in H \Lambda (K) $ and weak normal trace
of $ w _2 \in H ^\ast \Lambda (K) $ such that the integration-by-parts
identity \eqref{eq:ibp} continues to hold, where
$ \langle \cdot , \cdot \rangle _{ \partial K } $ is interpreted as a
duality pairing extending the $ L ^2 $ inner product on
$ \partial K $. We can therefore define a weak version of
$ [ \cdot , \cdot ] _{ \partial K } $ by \eqref{eq:bracket} whenever
the arguments possess both tangential and normal traces, e.g., when
both live in $ \mathbf{H \Lambda} (K) \cap \mathbf{H ^\ast \Lambda} (K) $.

Next, we define ``broken'' subspaces of differential forms and traces,
\begin{alignat*}{2}
  W _h &\coloneqq \prod _{ K \in \mathcal{T} _h } W _h (K) , \qquad & W _h (K) &\subset H \Lambda (K) \cap H ^\ast \Lambda (K) ,\\
  \widehat{ W } _h ^{\mathrm{nor}} &\coloneqq \prod _{ K \in \mathcal{T} _h } \widehat{ W } _h ^{\mathrm{nor}} ( \partial K ) , \qquad & \widehat{ W } _h ^{\mathrm{nor}} ( \partial K ) &\subset L ^2 \Lambda ( \partial K  ) ,\\
  \widehat{ W } _h ^{\mathrm{tan}} &\coloneqq \prod _{ K \in \mathcal{T} _h } \widehat{ W } _h ^{\mathrm{tan}} ( \partial K ) , \qquad & \widehat{ W } _h ^{\mathrm{tan}} ( \partial K ) &\subset L ^2 \Lambda ( \partial K  ) ,
\end{alignat*}
with the additional assumption that
$ w _h ^{\mathrm{nor}} , w _h ^{\mathrm{tan}} \in L ^2 \Lambda (
\partial \mathcal{T} _h ) \coloneqq \prod _{ K \in \mathcal{T} _h } L
^2 \Lambda ( \partial K ) $ for all $ w _h \in W _h $. The trace
spaces are generally \emph{double-valued} on the interior skeleton
$ \partial \mathcal{T} _h \setminus \partial \Omega $ and nonvanishing
on the domain boundary $ \partial \Omega $. We also define two
\emph{single-valued} tangential trace spaces
$ \ringhat{ V } _h ^{\mathrm{tan}} \subset \widehat{ V } _h
^{\mathrm{tan}} $ by
\begin{equation*}
  \ringhat{ V } _h ^{\mathrm{tan}}  \coloneqq \bigl\{ \widehat{ w } _h ^{\mathrm{tan}} \in \widehat{ W } _h ^{\mathrm{tan}} : \llbracket \widehat{ w } _h ^{\mathrm{tan}} \rrbracket = 0 \bigr\} , \qquad \widehat{ V } _h ^{\mathrm{tan}}  \coloneqq \bigl\{ \widehat{ w } _h ^{\mathrm{tan}} \in \widehat{ W } _h ^{\mathrm{tan}} : \llbracket \widehat{ w } _h ^{\mathrm{tan}} \rrbracket = 0 \text{ on } \partial \mathcal{T} _h \setminus \partial \Omega \bigr\}.
\end{equation*}
Here, $ \llbracket \widehat{ w } _h ^{\mathrm{tan}} \rrbracket $ is
the tangential jump, which by convention equals
$ \widehat{ w } _h ^{\mathrm{tan}} $ on $ \partial \Omega $. See
\citep[Definition 3.2]{StZa2025} for a detailed discussion of jumps
and averages for both tangential and normal traces. As above, we
define ``bold'' versions of these subspaces by taking tensor products
with $ \mathbb{R}^2 $.

To impose a relation between the normal and tangential traces, we
choose a \emph{local flux function}, which is a bounded linear map
\begin{equation*}
  \mathbf{\Phi}  \coloneqq \prod _{ K \in \mathcal{T} _h } \mathbf{\Phi} _K , \qquad \mathbf{\Phi} _K \colon \mathbf{W} _h (K) \times \widehat{ \mathbf{W} } _h ^{\mathrm{nor}} ( \partial K ) \times \widehat{ \mathbf{W} } _h ^{\mathrm{tan}} ( \partial K ) \rightarrow \mathbf{L}  ^2 \mathbf{\Lambda} ( \partial K ) .
\end{equation*}
We also replace the smooth source term
$ \partial H / \partial \mathbf{z} $ in \eqref{eq:hamiltonian_JD} by a
weaker local source term
\begin{equation*}
  \mathbf{f} \coloneqq \prod _{ K \in \mathcal{T} _h } \mathbf{f} _K , \qquad \mathbf{f} _K \colon I \times \mathbf{W} _h (K) \rightarrow \mathbf{L} ^2 \mathbf{\Lambda} (K) .
\end{equation*}
Assume that $ \mathbf{f} $ is at least $ C ^1 $ in $ \mathbf{z} _h $,
so that we may describe first variations of weak solutions in terms of
the derivative $ \partial \mathbf{f} / \partial \mathbf{z} _h $. The
case where this derivative is symmetric corresponds to the symmetry of
the Hessian in the Hamiltonian case.

We are finally ready to describe the weak form of
\eqref{eq:hamiltonian_JD} on which our methods are based. We seek
$ \mathbf{z} _h \colon I \rightarrow \mathbf{W} _h $ and
$ \widehat{ \mathbf{z} } _h \coloneqq ( \widehat{ \mathbf{z} } _h
^{\mathrm{nor}} , \widehat{ \mathbf{z} } _h ^{\mathrm{tan}} ) \colon I
\rightarrow \widehat{ \mathbf{W} } _h ^{\mathrm{nor}} \times \widehat{
  \mathbf{V} } _h ^{\mathrm{tan}} $ satisfying
\begin{subequations}
  \label{eq:weak}
  \begin{alignat}{2}
    ( \mathbf{J} \dot{\mathbf{z}} _h , \mathbf{w} _h ) _{ \mathcal{T} _h } + ( \mathbf{z} _h , \mathbf{D} \mathbf{w} _h ) _{ \mathcal{T} _h } + [ \widehat{ \mathbf{z} } _h , \mathbf{w} _h ] _{ \partial \mathcal{T} _h } &= \bigl( \mathbf{f} ( t, \mathbf{z} _h ) , \mathbf{w} _h \bigr)  _{ \mathcal{T} _h }  , \quad &\forall \mathbf{w} _h &\in \mathbf{W} _h , \label{eq:weak_w} \\
    \bigl\langle \mathbf{\Phi} ( \mathbf{z} _h , \widehat{ \mathbf{z} } _h ) , \widehat{ \mathbf{w} } _h ^{\mathrm{nor}} \bigr\rangle _{ \partial \mathcal{T} _h } &= 0 , \quad &\forall \widehat{ \mathbf{w} } _h ^{\mathrm{nor}} &\in \widehat{ \mathbf{W} } _h ^{\mathrm{nor}} , \label{eq:weak_wnor}\\
    \langle \widehat{ \mathbf{z} } _h ^{\mathrm{nor}} , \widehat{ \mathbf{w} } _h ^{\mathrm{tan}} \rangle _{ \partial \mathcal{T} _h } &= 0 , \quad & \forall \widehat{ \mathbf{w} } _h ^{\mathrm{tan}} &\in \ringhat{ \mathbf{V} } _h ^{\mathrm{tan}} \label{eq:weak_wtan} .
  \end{alignat}
\end{subequations}
Note that this semidiscretized formulation does not specify initial or
boundary conditions. The multisymplectic conservation law is a
statement about variations within a \emph{family} of solutions, in
which the initial and boundary values may vary as well. Of course, to
find a \emph{particular} solution, we would impose initial and
boundary values in addition to \eqref{eq:weak}.

\begin{remark}
  The tensor product construction of
  $ \mathbf{W} _h = W _h \otimes \mathbb{R}^2 $ ensures that
  $ ( \mathbf{J} \cdot , \cdot ) _{ \mathcal{T} _h } $ is a symplectic
  form on $ \mathbf{W} _h $. In particular, its nondegeneracy implies
  that $ \dot{ \mathbf{z} } _h $ is a well-defined function of
  $ \mathbf{z} _h $ and $ \widehat{ \mathbf{z} } _h $ for each
  $ t \in I $, which allows us to interpret \eqref{eq:weak_w} as an
  equation describing dynamics. This would not necessarily be true if
  we had taken an arbitrary subspace
  $ \mathbf{W} _h (K) \subset \mathbf{H} \mathbf{\Lambda} (K) \cap
  \mathbf{H} ^\ast \mathbf{\Lambda} (K) $.
\end{remark}

A first variation of a solution
$ ( \mathbf{z} _h, \widehat{ \mathbf{z} } _h ) $ to \eqref{eq:weak} is
a solution to the linearized weak problem, consisting of
$ \mathbf{w} _i \colon I \rightarrow \mathbf{W} _h $ and
$ \widehat{ \mathbf{w} } _i \coloneqq ( \widehat{ \mathbf{w} } _i
^{\mathrm{nor}} , \widehat{ \mathbf{w} } _i ^{\mathrm{tan}} ) \colon I
\rightarrow \widehat{ \mathbf{W} } _h ^{\mathrm{nor}} \times \widehat{
  \mathbf{V} } _h ^{\mathrm{tan}} $ satisfying
\begin{subequations}
	\label{eq:weakvar}
  \begin{alignat}{2}
    ( \mathbf{J} \dot{\mathbf{w}} _i , \mathbf{w} _h ) _{ \mathcal{T} _h } + ( \mathbf{w} _i , \mathbf{D} \mathbf{w} _h ) _{ \mathcal{T} _h } + [ \widehat{ \mathbf{w} } _i , \mathbf{w} _h ] _{ \partial \mathcal{T} _h } &= \biggl( \frac{ \partial \mathbf{f} }{ \partial \mathbf{z} _h }  \mathbf{w} _i , \mathbf{w} _h \biggr)  _{ \mathcal{T} _h }  , \quad &\forall \mathbf{w} _h &\in \mathbf{W} _h , \label{eq:weakvar_w}\\
    \bigl\langle \mathbf{\Phi} ( \mathbf{w} _i , \widehat{ \mathbf{w} } _i ) , \widehat{ \mathbf{w} } _h ^{\mathrm{nor}} \bigr\rangle _{ \partial \mathcal{T} _h } &= 0 , \quad &\forall \widehat{ \mathbf{w} } _h ^{\mathrm{nor}} &\in \widehat{ \mathbf{W} } _h ^{\mathrm{nor}} , \label{eq:weakvar_wnor}\\
    \langle \widehat{ \mathbf{w} } _i ^{\mathrm{nor}} , \widehat{ \mathbf{w} } _h ^{\mathrm{tan}} \rangle _{ \partial \mathcal{T} _h } &= 0 , \quad & \forall \widehat{ \mathbf{w} } _h ^{\mathrm{tan}} &\in \ringhat{ \mathbf{V} } _h ^{\mathrm{tan}} \label{eq:weakvar_wtan} .
  \end{alignat}
\end{subequations}
Note that \eqref{eq:weak} and \eqref{eq:weakvar} only differ in the
right-hand sides of \eqref{eq:weak_w} and \eqref{eq:weakvar_w}.

\begin{example}
		\label{ex:AFW}
  As a first example, we extend the \emph{AFW-H method} of
  \citet[Section~4.1]{StZa2025} from stationary to time-dependent
  systems. This is a hybridization of conforming FEEC and is named for
  \citet*{ArFaWi2006,ArFaWi2010}. In the stationary case, it includes
  the hybridized method of \citet*{AwFaGuSt2023} for the
  Hodge--Laplace problem and a similar hybridization of the method of
  \citet{LeSt2016} for the Hodge--Dirac problem.

  Let $ W _h (K) $ be a subcomplex of $ H \Lambda (K) $ for each
  $ K \in \mathcal{T} _h $, e.g., the trimmed piecewise-polynomial
  forms $ W _h ^k (K) = \mathcal{P} _r ^- \Lambda ^k (K) $ for some
  polynomial degree $r$ (cf.~\citep{ArFaWi2006,ArFaWi2010}), and take
  \begin{equation*}
    \widehat{ W } _h ^{\mathrm{nor}} ( \partial K ) = \widehat{ W } _h ^{\mathrm{tan}} ( \partial K ) = W _h ^{\mathrm{tan}} ( \partial K ) \coloneqq \bigl\{ w _h ^{\mathrm{tan}} : w _h \in W _h ( K ) \bigr\} .
  \end{equation*}
  In addition to the broken complex $ W _h $, we get two conforming
  subcomplexes
  $ \mathring{ V } _h \subset V _h \subset H \Lambda (\Omega) $,
  \begin{equation*}
    \mathring{ V } _h \coloneqq \bigl\{ w _h \in W _h : \llbracket w _h ^{\mathrm{tan}} \rrbracket = 0 \bigr\} , \qquad  V _h \coloneqq \bigl\{ w _h \in W _h : \llbracket w _h ^{\mathrm{tan}} \rrbracket = 0 \text{ on } \partial \mathcal{T} _h \setminus \partial \Omega \bigr\} ,
  \end{equation*}
  and the single-valued trace spaces are therefore
  $ \ringhat{ V } _h ^{\mathrm{tan}} = \mathring{ V
  } _h ^{\mathrm{tan}} $ and
  $ \widehat{ V } _h ^{\mathrm{tan}} = V _h
  ^{\mathrm{tan}} $. Finally, taking the local flux function to be 
  \begin{equation*}
    \mathbf{\Phi} ( \mathbf{z} _h , \widehat{ \mathbf{z} } _h ) = \widehat{ \mathbf{z} } _h ^{\mathrm{tan}} - \mathbf{z} _h ^{\mathrm{tan}} ,
  \end{equation*}
  the method \eqref{eq:weak} becomes
  \begin{subequations}
    \label{eq:afw-h}
    \begin{alignat}{2}
      ( \mathbf{J} \dot{\mathbf{z}} _h , \mathbf{w} _h ) _{ \mathcal{T} _h } + ( \mathbf{z} _h , \mathbf{D} \mathbf{w} _h ) _{ \mathcal{T} _h } + [ \widehat{ \mathbf{z} } _h , \mathbf{w} _h ] _{ \partial \mathcal{T} _h } &= \bigl( \mathbf{f} ( t, \mathbf{z} _h ), \mathbf{w} _h \bigr)  _{ \mathcal{T} _h }  , \quad &\forall \mathbf{w} _h &\in \mathbf{W} _h , \label{eq:afw-h_w} \\
      \bigl\langle \widehat{ \mathbf{z} } _h ^{\mathrm{tan}} - \mathbf{z} _h ^{\mathrm{tan}} , \widehat{ \mathbf{w} } _h ^{\mathrm{nor}} \bigr\rangle _{ \partial \mathcal{T} _h } &= 0 , \quad &\forall \widehat{ \mathbf{w} } _h ^{\mathrm{nor}} &\in \widehat{ \mathbf{W} } _h ^{\mathrm{nor}} , \label{eq:afw-h_wnor}\\
      \langle \widehat{ \mathbf{z} } _h ^{\mathrm{nor}} , \widehat{ \mathbf{w} } _h ^{\mathrm{tan}} \rangle _{ \partial \mathcal{T} _h } &= 0 , \quad & \forall \widehat{ \mathbf{w} } _h ^{\mathrm{tan}} &\in \ringhat{ \mathbf{V} } _h ^{\mathrm{tan}} \label{eq:afw-h_wtan} .
    \end{alignat}
  \end{subequations}
  Observe that taking
  $ \widehat{ \mathbf{w} } _h ^{\mathrm{nor}} = \widehat{ \mathbf{z} }
  _h ^{\mathrm{tan}} - \mathbf{z} _h ^{\mathrm{tan}} $ in
  \eqref{eq:afw-h_wnor} implies
  $ \mathbf{z} _h ^{\mathrm{tan}} = \widehat{ \mathbf{z} } _h
  ^{\mathrm{tan}} \in \widehat{ \mathbf{V} } _h ^{\mathrm{tan}} $, and
  therefore $ \mathbf{z} _h \in \mathbf{V} _h $.  By essentially the
  same argument as \citep[Theorem~4.1]{StZa2025}, it follows that
  \eqref{eq:afw-h} is a hybridization of the following conforming
  method: Find $ \mathbf{z} _h \colon I \rightarrow \mathbf{V} _h $
  such that
  \begin{equation}
    ( \mathbf{J} \dot{ \mathbf{z} } _h , \mathbf{w} _h ) _\Omega + ( \mathbf{d} \mathbf{z} _h , \mathbf{w} _h ) _\Omega + ( \mathbf{z} _h , \mathbf{d} \mathbf{w} _h ) _\Omega = \bigl( \mathbf{f} ( t, \mathbf{z} _h ) , \mathbf{w} _h \bigr) _\Omega , \quad \forall \mathbf{w} _h \in \mathring{ \mathbf{V} } _h .
    \label{eq:afw}
  \end{equation}
\end{example}

\begin{example}
	\label{ex:LDG}
  We next extend the LDG-H method of \citep[Section~4.2]{StZa2025}, which is
  a hybridizable discontinuous Galerkin (HDG) method, from stationary
  to time-dependent systems. For this method, one chooses trace spaces
  of the form
  \begin{equation*}
    \widehat{ W } _h ^{\mathrm{nor}} ( \partial K ) = L ^2 \Lambda ( \partial K ) , \qquad \widehat{ W } _h ^{\mathrm{tan}} ( \partial K ) = \prod _{ e \subset \partial K } \widehat{ W } _h ^{\mathrm{tan}} (e) ,
  \end{equation*}
  assuming that
  $ \widehat{ W } _h ^{\mathrm{tan}} ( e ^+ ) = \widehat{ W } _h
  ^{\mathrm{tan}} ( e ^- ) $ at interior facets $e$. One of the
  simplest choices is to take
  \begin{equation*}
    W _h (K) = \mathcal{P} _r \Lambda (K) , \qquad \widehat{ W } _h ^{\mathrm{tan}} (e) = \mathcal{P} _r \Lambda (e) ,
  \end{equation*}
  which gives an ``equal-order'' HDG method. The LDG-H flux function
  has the form
  \begin{equation*}
    \mathbf{\Phi} ( \mathbf{z} _h , \widehat{ \mathbf{z} } _h ) = ( \widehat{ \mathbf{z} } _h ^{\mathrm{nor}} - \mathbf{z} _h ^{\mathrm{nor}} ) + \boldsymbol{\alpha} ( \widehat{ \mathbf{z} } _h ^{\mathrm{tan}} - \mathbf{z} _h ^{\mathrm{tan}} ) ,
  \end{equation*}
  where
  $ \boldsymbol{\alpha} = \prod _{ e \subset \partial \mathcal{T} _h }
  \boldsymbol{\alpha} _e $ is a bounded, symmetric ``penalty''
  operator on
  $ \mathbf{L} ^2 \mathbf{\Lambda} ( \partial \mathcal{T} _h ) $. For
  example, $ \boldsymbol{\alpha} $ might be piecewise constant, where
  $ \boldsymbol{\alpha} _e ^k $ is multiplication by a scalar (or
  symmetric $ 2 \times 2 $ matrix) for each facet
  $ e \subset \partial \mathcal{T} _h $ and form degree
  $ k = 0 , \ldots, n - 1 $. Alternatively, $ \boldsymbol{\alpha} $
  may incorporate projection onto a lower-degree trace space, as in
  the reduced-stabilization techniques of
  \citet{Lehrenfeld2010,LeSc2016} and \citet{Oikawa2015,Oikawa2016};
  see \citep[Section~4.2]{StZa2025} for details. For the remainder of
  the paper, we will consider the LDG-H method with equal-order spaces
  and piecewise-constant penalties. See \citep[Theorem 4.9]{StZa2025}
  for a characterization of other choices that yield
  structure-preserving methods.

  Since
  $ \widehat{ \mathbf{W} } _h ^{\mathrm{nor}} = \mathbf{L} ^2
  \mathbf{\Lambda} ( \partial \mathcal{T} _h ) $, equation
  \eqref{eq:weak_wnor} says that
  $ \widehat{ \mathbf{z} } _h ^{\mathrm{nor}} = \mathbf{z} _h
  ^{\mathrm{nor}} - \boldsymbol{\alpha} ( \widehat{ \mathbf{z} } _h
  ^{\mathrm{tan}} - \mathbf{z} _h ^{\mathrm{tan}} ) $. Substituting
  this into \eqref{eq:weak_w} and \eqref{eq:weak_wtan} and integrating
  by parts gives an equivalent, symmetric formulation of the LDG-H
  method in the remaining variables: Find
  $ ( \mathbf{z} _h , \widehat{ \mathbf{z} } _h ^{\mathrm{tan}} )
  \colon I \rightarrow \mathbf{W} _h \times \widehat{ \mathbf{V} } _h
  ^{\mathrm{tan}} $ satisfying
  \begin{subequations}
    \label{eq:ldg-h}
    \begin{alignat}{2}
      \begin{multlined}[b]
        ( \mathbf{J} \dot{\mathbf{z}} _h , \mathbf{w} _h ) _{ \mathcal{T} _h } + ( \mathbf{z} _h , \boldsymbol{\delta} \mathbf{w} _h ) _{ \mathcal{T} _h } + ( \boldsymbol{\delta} \mathbf{z} _h , \mathbf{w} _h ) _{ \mathcal{T} _h } \\
        + \langle \widehat{ \mathbf{z} } _h ^{\mathrm{tan}} , \mathbf{w} _h ^{\mathrm{nor}} \rangle _{ \partial \mathcal{T} _h } + \bigl\langle \boldsymbol{\alpha} ( \widehat{ \mathbf{z} } _h ^{\mathrm{tan}} - \mathbf{z} _h ^{\mathrm{tan}} ) , \mathbf{w} _h ^{\mathrm{tan}} \bigr\rangle _{ \partial \mathcal{T} _h }
      \end{multlined}
      &= \bigl( \mathbf{f} ( t, \mathbf{z}  _h) , \mathbf{w} _h \bigr) _{ \mathcal{T} _h } , \quad &\forall \mathbf{w} _h &\in \mathbf{W} _h ,\label{eq:ldg-h_w}\\
      \bigl\langle \mathbf{z} _h ^{\mathrm{nor}} - \boldsymbol{\alpha} ( \widehat{ \mathbf{z} } _h ^{\mathrm{tan}} - \mathbf{z} _h ^{\mathrm{tan}} ) , \widehat{ \mathbf{w} } _h ^{\mathrm{tan}} \bigr\rangle _{ \partial \mathcal{T} _h } &= 0 , \quad &\forall \widehat{ \mathbf{w} } _h ^{\mathrm{tan}} &\in \ringhat{\mathbf{V}} _h ^{\mathrm{tan}}. \label{eq:ldg-h_wtan}
    \end{alignat}
  \end{subequations}
\end{example}

\subsection{Weak multisymplecticity}

We now develop a notion of what it means for a weak solution to
satisfy a multisymplectic conservation law locally on
$ K \in \mathcal{T} _h $. Similarly to the approach in
\citep{StZa2025}, this is done by modifying the integral form of the
multisymplectic conservation law \eqref{eq:mscl_integral} so that the
boundary terms involve the numerical traces
$ \widehat{ \mathbf{w} } _i $ rather than the traces of
$ \mathbf{w} _i $ on $ \partial K $.

\begin{definition}
  We say that \eqref{eq:weak} is \emph{(weakly) multisymplectic} if, whenever
  $(\mathbf{z}_h, \widehat{\mathbf{z}}_h)$ satisfies
  \eqref{eq:weak_w}--\eqref{eq:weak_wnor} with
  $\partial \mathbf{f} / \partial \mathbf{z}_h $ being symmetric, and
  $(\mathbf{w}_i, \widehat{\mathbf{w}}_i)$ satisfy
  \eqref{eq:weakvar_w}--\eqref{eq:weakvar_wnor} for $ i = 1, 2 $, we
  have
  \begin{equation}
    \label{eq:mscl_weak}
    \frac{\mathrm{d}}{\mathrm{d}t} ( \mathbf{J} \mathbf{w} _1 , \mathbf{w} _2 ) _K + [ \widehat{ \mathbf{w} } _1 , \widehat{ \mathbf{w} } _2 ] _{ \partial K } = 0 ,
  \end{equation}
  for all $K\in\mathcal{T}_h$.
\end{definition}

The main result of this section extends Lemma 3.9 of \citep{StZa2025}
to the time-dependent case, while also generalizing Theorem 4.6
\citet{McSt2024} for time-dependent de~Donder--Weyl systems.

\begin{theorem}
  \label{thm:local_ms}
  If $ ( \mathbf{z} _h , \widehat{ \mathbf{z} } _h ) $ satisfies
  \eqref{eq:weak_w} with
  $ \partial \mathbf{f} / \partial \mathbf{z} _h $ being symmetric,
  and if $ ( \mathbf{w} _1 , \widehat{ \mathbf{w} } _1 ) $ and
  $ ( \mathbf{w} _2 , \widehat{ \mathbf{w} } _2 ) $ satisfy
  \eqref{eq:weakvar_w}, then
  \begin{equation}
    \frac{\mathrm{d}}{\mathrm{d}t} ( \mathbf{J} \mathbf{w} _1 , \mathbf{w} _2 ) _K + [ \widehat{ \mathbf{w} } _1 , \widehat{ \mathbf{w} } _2 ] _{ \partial K } = [ \widehat{ \mathbf{w} } _1 - \mathbf{w} _1 , \widehat{ \mathbf{w} } _2 - \mathbf{w} _2 ] _{ \partial K } ,
  \end{equation}
  for all $ K \in \mathcal{T} _h $.  Consequently, the
  multisymplecticity condition \eqref{eq:mscl_weak} holds if and only
  if
  \begin{equation}
    \label{eq:mscl_jump}
    [ \widehat{ \mathbf{w} } _1 - \mathbf{w} _1 ,
    \widehat{ \mathbf{w} } _2 - \mathbf{w} _2 ] _{ \partial K } = 0.
  \end{equation}
\end{theorem}

\begin{proof}
  Since $ \mathbf{w} _1 $ satisfies \eqref{eq:weakvar_w}, letting
  $ \mathbf{w} _h $ be the extension by zero of
  $ \mathbf{w} _2 \rvert _K $ gives
  \begin{equation*}
    ( \mathbf{J} \dot{\mathbf{w}} _1 , \mathbf{w} _2 ) _K + ( \mathbf{w} _1 , \mathbf{D} \mathbf{w} _2 ) _K + [ \widehat{ \mathbf{w} } _1 , \mathbf{w} _2 ] _{ \partial K } = \biggl( \frac{ \partial \mathbf{f} }{ \partial \mathbf{z} _h } \mathbf{w} _1 , \mathbf{w} _2 \biggr)  _K ,
  \end{equation*}
  and likewise,
  \begin{equation*}
    ( \mathbf{J} \dot{\mathbf{w}} _2 , \mathbf{w} _1 ) _K + ( \mathbf{w} _2 , \mathbf{D} \mathbf{w} _1 ) _K + [ \widehat{ \mathbf{w} } _2 , \mathbf{w} _1 ] _{ \partial K } = \biggl( \frac{ \partial \mathbf{f} }{ \partial \mathbf{z} _h } \mathbf{w} _2 , \mathbf{w} _1 \biggr)  _K .
  \end{equation*}
  Subtracting these, the right-hand side vanishes by symmetry of
  $ \partial \mathbf{f} / \partial \mathbf{z} _h $, leaving
  \begin{equation*}
    \frac{\mathrm{d}}{\mathrm{d}t} ( \mathbf{J} \mathbf{w} _1 , \mathbf{w} _2 ) _K - [ \mathbf{w} _1 , \mathbf{w} _2 ] _{ \partial K } + [ \widehat{ \mathbf{w} } _1 , \mathbf{w} _2 ] _{ \partial K } + [ \mathbf{w} _1 , \widehat{ \mathbf{w} } _2 ] _{ \partial K } = 0 .
  \end{equation*}
  Here, we have used the Leibniz rule, the antisymmetry of
  $ \mathbf{J} $ and $ [ \cdot , \cdot ] _{ \partial K } $, and
  \eqref{eq:bracket_D}. Finally, adding
  \begin{equation*}
    [ \widehat{ \mathbf{w} } _1 - \mathbf{w} _1 , \widehat{ \mathbf{w} } _2 - \mathbf{w} _2 ] _{ \partial K } = [ \widehat{ \mathbf{w} } _1 , \widehat{ \mathbf{w} } _2 ] _{ \partial K } - [ \widehat{ \mathbf{w} } _1 , \mathbf{w} _2 ] _{ \partial K } - [ \mathbf{w} _1 , \widehat{ \mathbf{w} } _2 ] _{ \partial K } + [ \mathbf{w} _1 , \mathbf{w} _2 ] _{ \partial K } 
  \end{equation*}
  to both sides completes the proof.
\end{proof}

It follows that multisymplecticity is a property of the local flux
function $ \mathbf{\Phi} $, which determines the relationship between
$ \mathbf{w} _i $ and $ \widehat{ \mathbf{w} } _i $.

\begin{definition}
  \label{def:ms_flux}
  The local flux function $ \mathbf{\Phi} $ is \emph{multisymplectic}
  if \eqref{eq:mscl_jump} holds for all
  $ ( \mathbf{w} _1 , \widehat{ \mathbf{w} } _1 ) $ and
  $ ( \mathbf{w} _2 , \widehat{ \mathbf{w} } _2 ) $ satisfying
  \eqref{eq:weakvar_wnor}.
\end{definition}

\begin{example}
  For the AFW-H method, \eqref{eq:weakvar_wnor} reads
  \begin{equation*}
    \bigl\langle \widehat{ \mathbf{w} } _i ^{\mathrm{tan}} - \mathbf{w} _i ^{\mathrm{tan}} , \widehat{ \mathbf{w} } _h ^{\mathrm{nor}} \bigr\rangle _{ \partial \mathcal{T} _h } = 0 , \quad \forall \widehat{ \mathbf{w} } _h ^{\mathrm{nor}} \in \widehat{ \mathbf{W} } _h ^{\mathrm{nor}}.
  \end{equation*}
  Taking
  $ \widehat{ \mathbf{w} } _h ^{\mathrm{nor}} = \widehat{ \mathbf{w} }
  _i ^{\mathrm{tan}} - \mathbf{w} _i ^{\mathrm{tan}} $ implies
  $ \widehat{ \mathbf{w} } _i ^{\mathrm{tan}} - \mathbf{w} _i
  ^{\mathrm{tan}} = 0 $, which immediately gives
  \eqref{eq:mscl_jump}. Hence, the AFW-H method is multisymplectic.
\end{example}

\begin{example}
  \label{ex:ldg_ms}
  For the LDG-H method, \eqref{eq:weakvar_wnor} gives
  $ \widehat{ \mathbf{w} } _i ^{\mathrm{nor}} = \mathbf{w} _i
  ^{\mathrm{nor}} - \boldsymbol{\alpha} ( \widehat{ \mathbf{w} } _i
  ^{\mathrm{tan}} - \mathbf{w} _i ^{\mathrm{tan}} ) $. Therefore,
  \begin{equation*}
    [ \widehat{ \mathbf{w} } _1 - \mathbf{w} _1 , \widehat{ \mathbf{w} } _2 - \mathbf{w} _2 ] _{ \partial K }
    = \bigl\langle \boldsymbol{\alpha} ( \widehat{ \mathbf{w} } _1
    ^{\mathrm{tan}} - \mathbf{w} _1 ^{\mathrm{tan}} ) , \widehat{
      \mathbf{w} } _2 ^{\mathrm{tan}} - \mathbf{w} _2 ^{\mathrm{tan}}
    \bigr\rangle _{ \partial K } - \bigl\langle \boldsymbol{\alpha} (
    \widehat{ \mathbf{w} } _2 ^{\mathrm{tan}} - \mathbf{w} _2
    ^{\mathrm{tan}} ) , \widehat{ \mathbf{w} } _1 ^{\mathrm{tan}} -
    \mathbf{w} _1 ^{\mathrm{tan}} \bigr\rangle _{ \partial K } ,
  \end{equation*} 
  which vanishes since $ \boldsymbol{\alpha} $ is symmetric. Hence,
  \eqref{eq:mscl_jump} holds, and the LDG-H method is multisymplectic.
\end{example}

The definition of multisymplectic flux is essentially identical to
that for non-time-dependent systems \citep[Definition 3.11]{StZa2025},
except for being on the ``bold'' spaces of $ \mathbb{R}^2 $-valued
forms. Consequently, \emph{every multisymplectic method in
  \citep{StZa2025} for non-time-dependent systems yields a
  multisymplectic semidiscretization method for time-dependent
  systems}. We now formalize this statement as follows.

\begin{proposition}
  \label{prop:diagonal_flux}
  A local flux function
  $ \mathbf{\Phi} ( \mathbf{z} _h , \widehat{ \mathbf{z} } _h ) =
  \begin{bmatrix}
    \Phi _q ( q _h , \widehat{ q } _h ) \\
    \Phi _p ( p _h , \widehat{ p } _h ) 
  \end{bmatrix} $ 
  is multisymplectic in the sense of \cref{def:ms_flux} if and only if
  $\Phi _q $ and $ \Phi _p $ are multisymplectic in the sense of \citep[Definition
  3.11]{StZa2025}.
\end{proposition}

\begin{proof}
  Writing $ \mathbf{w} _i = \bigl[
  \begin{smallmatrix}
    s _i \\ r _i 
  \end{smallmatrix} \bigr] $, $ \widehat{ \mathbf{w} } _i =
  \Bigl[
  \begin{smallmatrix}
    \widehat{ s } _i \\ \widehat{ r } _i 
  \end{smallmatrix}
  \Bigr] $, and $ \widehat{ \mathbf{w} } _h ^{\mathrm{nor}} =   \Bigl[
  \begin{smallmatrix}
    \widehat{ s } _h \\ \widehat{ r } _h 
  \end{smallmatrix}
  \Bigr] $, the condition \eqref{eq:weakvar_wnor} is equivalent to
  \begin{alignat*}{2}
    \bigl\langle \Phi _q ( s _i , \widehat{ s } _i ) , \widehat{ s } _h ^{\mathrm{nor}} \bigr\rangle _{ \partial \mathcal{T} _h  } &= 0 , \quad & \forall \widehat{ s } _h ^{\mathrm{nor}} &\in \widehat{ W } _h ^{\mathrm{nor}} ,\\
    \bigl\langle \Phi _p ( r _i , \widehat{ r } _i ) , \widehat{ r } _h ^{\mathrm{nor}} \bigr\rangle _{ \partial \mathcal{T} _h  } &= 0 , \quad & \forall \widehat{ r } _h ^{\mathrm{nor}} &\in \widehat{ W } _h ^{\mathrm{nor}} .
  \end{alignat*}
  Hence, \eqref{eq:mscl_jump} holding for all such
  $ ( \mathbf{w} _i , \widehat{ \mathbf{w} } _i ) $ is equivalent to
  \begin{align*}
    [ \widehat{ s } _1 - s _1 , \widehat{ s } _2 - s _2 ] _{ \partial K } &= 0 ,\\
    [ \widehat{ r } _1 - r _1 , \widehat{ r } _2 - r _2 ] _{ \partial K } &= 0 ,
  \end{align*}
  for all such $ ( s _i , \widehat{ s } _i ) $ and
  $ ( r _i , \widehat{ r } _i ) $, which is precisely
  multisymplecticity of $\Phi _q$ and $ \Phi _p $.
\end{proof}

\begin{corollary}
  A local flux function $ \mathbf{\Phi} = \Phi \otimes \mathbb{R}^2 $
  is multisymplectic if and only if $\Phi$ is.
\end{corollary}

\begin{proof}
  Apply \cref{prop:diagonal_flux} with $ \Phi _q = \Phi _p = \Phi $.
\end{proof}

\begin{remark}
  Multisymplecticity of AFW-H is a special case of this result: its
  local flux function is $ \mathbf{\Phi} = \Phi \otimes \mathbb{R}^2 $
  with
  $ \Phi ( z _h , \widehat{ z } _h ) = \widehat{ z } _h
  ^{\mathrm{tan}} - z _h ^{\mathrm{tan}} $, which is multisymplectic
  by \citep[Theorem~4.3]{StZa2025}.

  For LDG-H, if the penalty operator has the form
  $ \boldsymbol{\alpha} =
  \begin{bsmallmatrix}
    \alpha _q \\ & \alpha _p 
  \end{bsmallmatrix} 
  $, where $\alpha _q $ and $ \alpha _p $ are symmetric operators on
  $ L ^2 \Lambda ( \partial \mathcal{T} _h ) $, then this corresponds
  to
  $ \Phi _q ( q _h , \widehat{ q } _h ) = (\widehat{ q } _h
  ^{\mathrm{nor}} - q _h ^{\mathrm{nor}}) + \alpha _q ( \widehat{ q }
  _h ^{\mathrm{tan}} - q _h ^{\mathrm{tan}} ) $ and
  $ \Phi _p ( p _h , \widehat{ p } _h ) = (\widehat{ p } _h
  ^{\mathrm{nor}} - p _h ^{\mathrm{nor}}) + \alpha _p ( \widehat{ p }
  _h ^{\mathrm{tan}} - p _h ^{\mathrm{tan}} ) $, which are
  multisymplectic by \citep[Theorem~4.9]{StZa2025}. However,
  \cref{ex:ldg_ms} shows that LDG-H is multisymplectic more generally,
  even when $ \boldsymbol{\alpha} $ is not block-diagonal.
\end{remark}

\subsection{Strong multisymplecticity}
\label{sec:strong_ms}

Under some additional hypotheses, it is possible to extend
\eqref{eq:mscl_weak} from a single element $ K \in \mathcal{T} _h $ to
an arbitrary collection of elements
$\mathcal{K} \subset \mathcal{T}_h$, which cover a region with
boundary ${\partial (\overline{\bigcup\mathcal{K}})}$. This stronger
notion of multisymplecticity is characterized as follows.

\begin{definition}
  We say that \eqref{eq:weak} is \emph{strongly multisymplectic} if,
  whenever $(\mathbf{z}_h, \widehat{\mathbf{z}}_h)$ is a solution with
  $\partial \mathbf{f} / \partial \mathbf{z}_h $ being symmetric, and
  $(\mathbf{w}_i, \widehat{\mathbf{w}}_i)$ with $i = 1,2$ are first
  variations, i.e. solutions of \eqref{eq:weakvar}, we have
  \begin{equation}
    \label{eq:mscl_strong}
    \frac{\mathrm{d}}{\mathrm{d}t} ( \mathbf{J} \mathbf{w} _1 , \mathbf{w} _2 ) _{\mathcal{K}}+ [ \widehat{ \mathbf{w} } _1 , \widehat{ \mathbf{w} } _2 ] _{\partial (\overline{\bigcup\mathcal{K}})} = 0 ,
  \end{equation}
  for any collection of elements $\mathcal{K} \subset \mathcal{T}_h$.
\end{definition}

\begin{definition}
  We say that the local flux $\mathbf{\Phi}$ is \emph{strongly
    conservative} if \eqref{eq:weak_wnor}--\eqref{eq:weak_wtan} imply
  that $ \widehat{ \mathbf{z} } _h ^{\mathrm{nor}} $ is single-valued,
  in the sense that
  $ \widehat{\mathbf{z}}_h^{\mathrm{nor}}\rvert_{e^+} +
  \widehat{\mathbf{z}}_h^{\mathrm{nor}}\rvert_{e^-} = 0 $ on every
  interior facet $ e = \partial K ^+ \cap \partial K ^- $.  (In the
  notation of \citep[Definition 3.2]{StZa2025}, this says that the
  normal jump
  $ \llbracket \widehat{\mathbf{z}}_h^{\mathrm{nor}} \rrbracket $
  vanishes.)
\end{definition}

The following result simultaneously generalizes Theorem 3.17 in
\citet{StZa2025} and Theorem 4.7 in \citet{McSt2024}.

\begin{theorem}
  If the local flux function $ \mathbf{\Phi} $ is strongly
  conservative and multisymplectic, then \eqref{eq:weak} is strongly
  multisymplectic.
\end{theorem}

\begin{proof}
  Following the proof of Theorem 3.17 in \citet{StZa2025}, we have
  \begin{equation}
    \label{eq:cons_id}    
    [\widehat{\mathbf{w}}_1, \widehat{\mathbf{w}}_2]_{\partial (\overline{\bigcup\mathcal{K}})} = [\widehat{\mathbf{w}}_1, \widehat{\mathbf{w}}_2]_{\partial \mathcal{K}} \coloneqq \sum_{K\in\mathcal{K}}[\widehat{\mathbf{w}}_1, \widehat{\mathbf{w}}_2]_{\partial K},
  \end{equation}
  since
  $ \llbracket \widehat{ \mathbf{w} } _i ^{\mathrm{nor}} \rrbracket =
  0 $ causes the interior-facet contributions of
  $ \partial \mathcal{K} $ to cancel. Therefore, summing
  \eqref{eq:mscl_weak} over all $K \in \mathcal{K}$ and applying
  \eqref{eq:cons_id} gives \eqref{eq:mscl_strong}, as claimed.
\end{proof}

Finally, we remark that strong conservativity of a flux in the form of
\cref{prop:diagonal_flux} is equivalent to strong conservativity of
$\Phi _q$ and $ \Phi _p $. Thus, strongly multisymplectic methods for
stationary problems immediately give strongly multsymplectic
semidiscretization methods for time-dependent problems.  In
particular, the AFW-H method is multisymplectic but \emph{not strongly
  multisymplectic} except in dimension $ n = 1 $
\citep[Theorem~4.3]{StZa2025}. On the other hand, the LDG-H method is
strongly multisymplectic under some mild assumptions on the spaces and
penalties \citep[Theorem~4.9]{StZa2025}, including the equal-order
method with piecewise-constant penalties
\citep[Corollary~4.10]{StZa2025}.

\subsection{Methods for the semilinear Hodge wave equation}
\label{sec:hodge_wave_methods}

Let us now apply the framework of this section to the $k$-form
semilinear Hodge wave equation introduced in \cref{ex:hodge_wave}. We
seek solutions of the form
\begin{equation*}
  \mathbf{z} _h =
  \begin{bmatrix}
    \sigma _h \oplus u _h \oplus \rho _h \\
    p _h 
  \end{bmatrix}, \qquad \widehat{ \mathbf{z} } _h =
  \begin{bmatrix}
    \widehat{ \sigma } _h \oplus \widehat{ u } _h \oplus \widehat{ \rho } _h \\
    \widehat{ p } _h
  \end{bmatrix},
\end{equation*}
with the right-hand side having the form
\begin{equation*}
  \mathbf{f} ( t, \mathbf{z} _h ) =
  \begin{bmatrix}
    - \sigma _h \oplus f ( t, u _h ) \oplus - \rho _h \\
    p _h 
  \end{bmatrix}, \qquad f \colon I \times W _h ^k \rightarrow L ^2 \Lambda ^k (\Omega) .
\end{equation*}
With this setup, \eqref{eq:weak_w} corresponds to the following
semidiscretization of \eqref{eq:hodge_wave}:
\begin{subequations}
  \label{eq:weak_wave}
  \begin{alignat}{2}
    ( \dot{ \sigma } _h , \tau _h ) _{ \mathcal{T} _h } + ( p _h , \mathrm{d} \tau _h ) _{ \mathcal{T} _h } - \langle \widehat{ p } _h ^{\mathrm{nor}} , \tau _h ^{\mathrm{tan}} \rangle _{ \partial \mathcal{T} _h } &= 0, \quad &\forall \tau _h &\in W _h ^{ k -1 } , \label{eq:weak_wave_sigmadot}\\
    ( \dot{ u } _h , r _h ) _{ \mathcal{T} _h } &= ( p _h , r _h ) _{ \mathcal{T} _h } , \quad &\forall r _h &\in W _h ^k , \label{eq:weak_wave_udot}\\
    ( \dot{ \rho } _h , \eta _h ) _{ \mathcal{T} _h } + ( p _h , \delta \eta _h ) _{ \mathcal{T} _h } + \langle \widehat{ p } _h ^{\mathrm{tan}} , \eta _h ^{\mathrm{nor}} \rangle _{ \partial \mathcal{T} _h } &= 0, \quad &\forall \eta _h &\in W _h ^{ k + 1 } , \label{eq:weak_wave_rhodot} \\
    ( u _h , \mathrm{d} \tau _h ) _{ \mathcal{T} _h } - \langle \widehat{ u } _h ^{\mathrm{nor}} , \tau _h ^{\mathrm{tan}} \rangle _{ \partial \mathcal{T} _h } &= - ( \sigma _h , \tau _h ) _{ \mathcal{T} _h } , \quad &\forall \tau _h &\in W _h ^{ k -1 } , \label{eq:weak_wave_thetadot}\\
    \begin{multlined}[b]
      - ( \dot{ p } _h , v _h ) _{ \mathcal{T} _h } + ( \sigma _h , \delta v _h ) _{ \mathcal{T} _h } + ( \rho _h , \mathrm{d} v _h ) _{ \mathcal{T} _h } \\
      + \langle \widehat{ \sigma } _h ^{\mathrm{tan}} , v _h ^{\mathrm{nor}} \rangle _{ \partial \mathcal{T} _h } - \langle \widehat{ \rho } _h ^{\mathrm{nor}} , v _h ^{\mathrm{tan}} \rangle _{ \partial \mathcal{T} _h }
    \end{multlined} &= \bigl( f ( t, u _h ) , v _h \bigr)  _{ \mathcal{T} _h } , \quad &\forall v _h &\in W _h ^k , \label{eq:weak_wave_pdot}\\
    ( u _h , \delta \eta _h ) _{ \mathcal{T} _h } + \langle \widehat{ u } _h ^{\mathrm{tan}} , \eta _h ^{\mathrm{nor}} \rangle _{ \partial \mathcal{T} _h } &= - ( \rho _h , \eta _h ) _{ \mathcal{T} _h } , \quad &\forall \eta _h &\in W _h ^{ k + 1 } . \label{eq:weak_wave_xidot}
  \end{alignat}
\end{subequations}
In addition to $ \dot{ u } _h = p _h $, which holds by
\eqref{eq:weak_wave_udot}, we also assume
$ \dot{ \widehat{ u } } _h = \widehat{ p } _h $ so that
\eqref{eq:weak_wave_sigmadot} and \eqref{eq:weak_wave_rhodot}
automatically preserve the constraints \eqref{eq:weak_wave_thetadot}
and \eqref{eq:weak_wave_xidot}, respectively.

To ensure that we can ignore form degrees other than those appearing in
\eqref{eq:weak_wave}---analogous to restricting to the invariant
subspace $ \mathbf{S} $ in \cref{ex:hodge_wave}---we fix the following
components of $\mathbf{\Phi}$:
\begin{align*}
  \Phi _q ^j ( \mathbf{z} _h , \widehat{ \mathbf{z} } _h )
  &= \begin{cases}
    ( \widehat{ q } _h ^{\mathrm{tan}} ) ^j - ( q _h ^j ) ^{\mathrm{tan}} , & j \leq  k - 2 , \\
    ( \widehat{ q } _h ^{\mathrm{nor}} ) ^j - ( q _h ^{ j + 1 } ) ^{\mathrm{nor}} , & j \geq k + 1 ,
  \end{cases}\\
  \Phi _p ^j ( \mathbf{z} _h , \widehat{ \mathbf{z} } _h ) 
  &= \begin{cases}
    ( \widehat{ p } _h ^{\mathrm{tan}} ) ^j - ( p _h ^j ) ^{\mathrm{tan}} , & j \leq k -1 , \\
    ( \widehat{ p } _h ^{\mathrm{nor}} ) ^j - ( p _h ^{ j + 1 } ) ^{\mathrm{nor}} , & j \geq k .
  \end{cases}
\end{align*}
Hence, the only remaining flux components to specify are
$ \Phi _q ^{ k -1 } $ and $ \Phi _q ^k $. This form of
$ \mathbf{\Phi} $ also ensures that, for the methods below, the
multisymplectic flux condition \eqref{eq:mscl_jump} simplifies to
\begin{multline}
  \langle \widehat{ \tau } _1 ^{\mathrm{tan}} - \tau _1 ^{\mathrm{tan}} , \widehat{ v } _2 ^{\mathrm{nor}} - v _2 ^{\mathrm{nor}} \rangle _{ \partial K } + \langle \widehat{ v } _1 ^{\mathrm{tan}} - v _1 ^{\mathrm{tan}} , \widehat{ \eta } _2 ^{\mathrm{nor}} - \eta _2 ^{\mathrm{nor}} \rangle _{ \partial K } \\
  = \langle \widehat{ \tau } _2 ^{\mathrm{tan}} - \tau _2 ^{\mathrm{tan}} , \widehat{ v } _1 ^{\mathrm{nor}} - v _1 ^{\mathrm{nor}} \rangle _{ \partial K } + \langle \widehat{ v } _2 ^{\mathrm{tan}} - v _2 ^{\mathrm{tan}} , \widehat{ \eta } _1 ^{\mathrm{nor}} - \eta _1 ^{\mathrm{nor}} \rangle _{ \partial K }, \label{eq:mscl_jump_hodge}
\end{multline}
since all other terms of
$ [ \widehat{ \mathbf{w} } _1 - \mathbf{w} _1 , \widehat{ \mathbf{w} }
  _2 - \mathbf{w} _2 ] _{ \partial K } $ vanish. The semidiscrete
multisymplectic conservation law \eqref{eq:mscl_weak} becomes
\begin{equation*}
  \frac{\mathrm{d}}{\mathrm{d}t} ( v _1 , r _2 ) _K + \langle \widehat{ \tau } _1 ^{\mathrm{tan}} , \widehat{ v } _2 ^{\mathrm{nor}} \rangle _{ \partial K } + \langle \widehat{ v } _1 ^{\mathrm{tan}} , \widehat{ \eta } _2 ^{\mathrm{nor}} \rangle _{ \partial K } = \frac{\mathrm{d}}{\mathrm{d}t} ( v _2 , r _1 ) _K + \langle \widehat{ \tau } _2 ^{\mathrm{tan}} , \widehat{ v } _1 ^{\mathrm{nor}} \rangle _{ \partial K } + \langle \widehat{ v } _2 ^{\mathrm{tan}} , \widehat{ \eta } _1 ^{\mathrm{nor}} \rangle _{ \partial K } ,
\end{equation*}
which is essentially that of \cref{ex:hodge_wave_integral_mscl} with
hats on the trace variables.

\begin{remark}
  \label{rmk:sigmanor_rhotan}
  Just as $ \dot{ p } _h ^{ k \pm 1 } = 0 $ corresponds to the
  constraints \eqref{eq:weak_wave_thetadot} and
  \eqref{eq:weak_wave_xidot}, the condition
  $ \dot{ p } _h ^{ k \pm 2 } = 0 $ corresponds to a pair of
  constraints
  \begin{subequations}
    \label{eq:sigmanor_rhotan}
    \begin{alignat}{2}
      ( \sigma _h  , \mathrm{d} w _h ) _{ \mathcal{T} _h } - \langle \widehat{ \sigma } _h ^{\mathrm{nor}} , w _h ^{\mathrm{tan}} \rangle _{ \partial \mathcal{T} _h } &= 0 , \quad &\forall w _h &\in W _h ^{ k - 2 } , \label{eq:sigmanor} \\
      ( \rho _h  , \delta w _h ) _{ \mathcal{T} _h } + \langle \widehat{ \rho } _h ^{\mathrm{tan}} , w _h ^{\mathrm{nor}} \rangle _{ \partial \mathcal{T} _h } &= 0 , \quad &\forall w _h &\in W _h ^{ k + 2 } . \label{eq:rhotan} 
    \end{alignat}
  \end{subequations}
  By \eqref{eq:weak_wave_thetadot} with $ \tau _h = \mathrm{d} w _h $
  and \eqref{eq:weak_wave_xidot} with $ \eta _h = \delta w _h $, we
  see that \eqref{eq:sigmanor_rhotan} is equivalent to
  \begin{subequations}
    \label{eq:sigmanor_rhotan_boundary}
    \begin{alignat}{2}
      \langle \widehat{ \sigma } _h ^{\mathrm{nor}} , w _h ^{\mathrm{tan}} \rangle _{ \partial \mathcal{T} _h } &= \langle  \widehat{ u } _h ^{\mathrm{nor}}  , \mathrm{d} w _h ^{\mathrm{tan}} \rangle _{ \partial \mathcal{T} _h } , \quad &\forall w _h &\in W _h ^{ k - 2 } , \label{eq:sigmanor_unor} \\
      \langle \widehat{ \rho } _h ^{\mathrm{tan}} , w _h ^{\mathrm{nor}} \rangle _{ \partial \mathcal{T} _h } &= \langle \widehat{ u } _h ^{\mathrm{tan}} , \delta w _h ^{\mathrm{nor}} \rangle _{ \partial \mathcal{T} _h }  , \quad &\forall w _h &\in W _h ^{ k + 2 } \label{eq:rhotan_utan} .
    \end{alignat}
  \end{subequations}
  For the methods below, we will show that there exist well-defined
  $ \widehat{ \sigma } _h ^{\mathrm{nor}} $ and
  $ \widehat{ \rho } _h ^{\mathrm{tan}} $ satisfying these
  constraints---but they need not be computed in practice, since they
  do not appear in \eqref{eq:weak_wave}.
\end{remark}

\subsubsection{Two implementations of the AFW-H method}
  For the AFW-H method, we take the fluxes
  \begin{align*}
    \Phi _q ^{ k -1 } ( \mathbf{z} _h , \widehat{ \mathbf{z} } _h ) &= \widehat{ \sigma } _h ^{\mathrm{tan}} - \sigma _h ^{\mathrm{tan}} ,\\
    \Phi _q ^k ( \mathbf{z} _h , \widehat{ \mathbf{z} } _h ) &= \widehat{ u } _h ^{\mathrm{tan}} - u _h ^{\mathrm{tan}} .
  \end{align*}
  This is clearly multisymplectic: variations satisfy
  $ \widehat{ \tau } _i ^{\mathrm{tan}} = \tau _i ^{\mathrm{tan}} $
  and $ \widehat{ v } _i ^{\mathrm{tan}} = v _i ^{\mathrm{tan}} $ for
  $ i = 1, 2 $, and hence all the terms of \eqref{eq:mscl_jump_hodge}
  vanish. While we only have weak multisymplecticity in general,
  strong multisymplecticity holds in certain special cases, namely
  $ n = 1 $ and $ k = n $, cf.~\citep[Section 4.1]{StZa2025}.
  
  As seen above, taking
  $ \dot{ \widehat{ u } } _h = \widehat{ p } _h $ implies that
  \eqref{eq:weak_wave_sigmadot} and \eqref{eq:weak_wave_rhodot} are
  the time derivatives of \eqref{eq:weak_wave_thetadot} and
  \eqref{eq:weak_wave_xidot}, respectively. By choosing which of these
  pairs of equations to eliminate, we will obtain two implementations
  of AFW-H with equivalent solutions.

  First, suppose we eliminate \eqref{eq:weak_wave_sigmadot} and
  \eqref{eq:weak_wave_rhodot}. This yields the dynamical equations
  \begin{subequations}
    \begin{alignat}{2}
      ( \dot{ u } _h , r _h ) _{ \mathcal{T} _h } &= ( p _h , r _h ) _{ \mathcal{T} _h } , \quad &\forall r _h &\in W _h ^k , \label{eq:afw-h_wave_r} \\
    \begin{multlined}[b]
      - ( \dot{ p } _h , v _h ) _{ \mathcal{T} _h } + ( \sigma _h , \delta v _h ) _{ \mathcal{T} _h } + ( \rho _h , \mathrm{d} v _h ) _{ \mathcal{T} _h } \\
      + \langle \widehat{ \sigma } _h ^{\mathrm{tan}} , v _h ^{\mathrm{nor}} \rangle _{ \partial \mathcal{T} _h } - \langle \widehat{ \rho } _h ^{\mathrm{nor}} , v _h ^{\mathrm{tan}} \rangle _{ \partial \mathcal{T} _h }
    \end{multlined} &= \bigl( f ( t, u _h ) , v _h \bigr)  _{ \mathcal{T} _h } , \quad &\forall v _h &\in W _h ^k , \label{eq:afw-h_wave_v} \\
      \intertext{together with the constraints}
      ( u _h , \mathrm{d} \tau _h ) _{ \mathcal{T} _h } - \langle \widehat{ u } _h ^{\mathrm{nor}} , \tau _h ^{\mathrm{tan}} \rangle _{ \partial \mathcal{T} _h } &= - ( \sigma _h , \tau _h ) _{ \mathcal{T} _h } , \quad &\forall \tau _h &\in W _h ^{ k -1 } , \label{eq:afw-h_wave_tau} \\
      ( u _h , \delta \eta _h ) _{ \mathcal{T} _h } + \langle \widehat{ u } _h ^{\mathrm{tan}} , \eta _h ^{\mathrm{nor}} \rangle _{ \partial \mathcal{T} _h } &= - ( \rho _h , \eta _h ) _{ \mathcal{T} _h } , \quad &\forall \eta _h &\in W _h ^{ k + 1 } , \label{eq:afw-h_wave_eta} \\
      \intertext{the flux conditions}
      \langle \widehat{ \sigma } _h ^{\mathrm{tan}} - \sigma _h ^{\mathrm{tan}} , \widehat{ v } _h ^{\mathrm{nor}} \rangle _{ \partial \mathcal{T} _h } &= 0 , \quad &\forall \widehat{ v } _h ^{\mathrm{nor}} &\in \widehat{ W } _h ^{k-1, \mathrm{nor}}, \label{eq:afw-h_wave_vnor} \\
      \langle \widehat{ u } _h ^{\mathrm{tan}} - u _h ^{\mathrm{tan}} , \widehat{ \eta } _h ^{\mathrm{nor}} \rangle _{ \partial \mathcal{T} _h } &= 0 , \quad &\forall \widehat{ \eta } _h ^{\mathrm{nor}} &\in \widehat{ W } _h ^{k, \mathrm{nor}}, \label{eq:afw-h_wave_etanor} \\
      \intertext{and the conservativity conditions}
      \langle \widehat{ u } _h ^{\mathrm{nor}} , \widehat{ \tau } _h ^{\mathrm{tan}} \rangle _{ \partial \mathcal{T} _h } &= 0 , \quad &\forall \widehat{ \tau } _h ^{\mathrm{tan}} &\in \ringhat{V} _h ^{k-1, \mathrm{tan}} , \label{eq:afw-h_wave_tautan} \\
      \langle \widehat{ \rho } _h ^{\mathrm{nor}} , \widehat{ v } _h ^{\mathrm{tan}} \rangle _{ \partial \mathcal{T} _h } &= 0 , \quad &\forall \widehat{ v } _h ^{\mathrm{tan}} &\in \ringhat{V} _h ^{k, \mathrm{tan}} . \label{eq:afw-h_wave_vtan}
    \end{alignat}
  \end{subequations}
  This is a hybridization of the conforming AFW method with dynamical equations
  \begin{subequations}
    \begin{alignat}{2}
      ( \dot{ u } _h , r _h ) _\Omega &= ( p _h , r _h ) _\Omega , \quad &\forall r _h &\in  \mathring{ V } _h ^k , \label{eq:afw_wave_r} \\
      - ( \dot{ p } _h , v _h ) _\Omega + ( \mathrm{d} \sigma _h , v _h ) _\Omega + ( \rho _h , \mathrm{d} v _h ) _\Omega &= \bigl( f ( t, u _h ) , v _h \bigr)  _\Omega , \quad &\forall v _h &\in \mathring{ V }  _h ^k , \label{eq:afw_wave_v} \\
      \intertext{and constraints}
      ( u _h , \mathrm{d} \tau _h ) _\Omega &= - ( \sigma _h , \tau _h ) _\Omega , \quad &\forall \tau _h &\in \mathring{ V } _h ^{ k -1 } , \label{eq:afw_wave_tau} \\
      ( \mathrm{d} u _h , \eta _h ) _\Omega &= - ( \rho _h , \eta _h ) _\Omega , \quad &\forall \eta _h &\in \mathring{ V } _h ^{ k + 1 } . \label{eq:afw_wave_eta} 
    \end{alignat}
  \end{subequations}
  For $ k = 0 $, this coincides with \citet[Equation 3]{SaVa24}.

  Alternatively, suppose we eliminate the constraints
  \eqref{eq:weak_wave_thetadot} and \eqref{eq:weak_wave_xidot},
  assuming that they hold at the initial time. The resulting method
  has the dynamical equations
  \begin{subequations}
    \begin{alignat}{2}
      ( \dot{ \sigma } _h , \tau _h ) _{ \mathcal{T} _h } + ( p _h , \mathrm{d} \tau _h ) _{ \mathcal{T} _h } - \langle \widehat{ p } _h ^{\mathrm{nor}} , \tau _h ^{\mathrm{tan}} \rangle _{ \partial \mathcal{T} _h } &= 0, \quad &\forall \tau _h &\in W _h ^{ k -1 } , \label{eq:afw2-h_wave_tau} \\
      ( \dot{ u } _h , r _h ) _{ \mathcal{T} _h } &= ( p _h , r _h ) _{ \mathcal{T} _h } , \quad &\forall r _h &\in W _h ^k , \label{eq:afw2-h_wave_r} \\
      ( \dot{ \rho } _h , \eta _h ) _{ \mathcal{T} _h } + ( p _h , \delta \eta _h ) _{ \mathcal{T} _h } + \langle \widehat{ p } _h ^{\mathrm{tan}} , \eta _h ^{\mathrm{nor}} \rangle _{ \partial \mathcal{T} _h } &= 0, \quad &\forall \eta _h &\in W _h ^{ k + 1 } , \label{eq:afw2-h_wave_eta} \\
      \begin{multlined}[b]
        - ( \dot{ p } _h , v _h ) _{ \mathcal{T} _h } + ( \sigma _h , \delta v _h ) _{ \mathcal{T} _h } + ( \rho _h , \mathrm{d} v _h ) _{ \mathcal{T} _h } \\
        + \langle \widehat{ \sigma } _h ^{\mathrm{tan}} , v _h ^{\mathrm{nor}} \rangle _{ \partial \mathcal{T} _h } - \langle \widehat{ \rho } _h ^{\mathrm{nor}} , v _h ^{\mathrm{tan}} \rangle _{ \partial \mathcal{T} _h }
      \end{multlined} &= \bigl( f ( t, u _h ) , v _h \bigr) _{ \mathcal{T} _h } , \quad &\forall v _h &\in W _h ^k , \label{eq:afw2-h_wave_v} \\
      \intertext{the flux conditions}
      \langle \widehat{ \sigma } _h ^{\mathrm{tan}} - \sigma _h ^{\mathrm{tan}} , \widehat{ v } _h ^{\mathrm{nor}} \rangle _{ \partial \mathcal{T} _h } &= 0 , \quad &\forall \widehat{ v } _h ^{\mathrm{nor}} &\in \widehat{ W } _h ^{k-1, \mathrm{nor}}, \label{eq:afw2-h_wave_vnor} \\
      \langle \widehat{ p } _h ^{\mathrm{tan}} - p _h ^{\mathrm{tan}} , \widehat{ \eta } _h ^{\mathrm{nor}} \rangle _{ \partial \mathcal{T} _h } &= 0 , \quad &\forall \widehat{ \eta } _h ^{\mathrm{nor}} &\in \widehat{ W } _h ^{k, \mathrm{nor}}, \label{eq:afw2-h_wave_etanor} \\
      \intertext{and the conservativity conditions}
      \langle \widehat{ p } _h ^{\mathrm{nor}} , \widehat{ \tau } _h ^{\mathrm{tan}} \rangle _{ \partial \mathcal{T} _h } &= 0 , \quad &\forall \widehat{ \tau } _h ^{\mathrm{tan}} &\in \ringhat{V} _h ^{k-1, \mathrm{tan}} , \label{eq:afw2-h_wave_tautan} \\
      \langle \widehat{ \rho } _h ^{\mathrm{nor}} , \widehat{ v } _h ^{\mathrm{tan}} \rangle _{ \partial \mathcal{T} _h } &= 0 , \quad &\forall \widehat{ v } _h ^{\mathrm{tan}} &\in \ringhat{V} _h ^{k, \mathrm{tan}} . \label{eq:afw2-h_wave_vtan}
    \end{alignat}
  \end{subequations}
  This is a hybridization of the conforming AFW method
  \begin{subequations}
    \begin{alignat}{2}
      ( \dot{ \sigma } _h , \tau _h ) _\Omega + ( p _h , \mathrm{d} \tau _h ) _\Omega &= 0, \quad &\forall \tau _h &\in \mathring{ V } _h ^{ k -1 } , \label{eq:afw2_wave_tau} \\
      ( \dot{ u } _h , r _h ) _\Omega &= ( p _h , r _h ) _\Omega , \quad &\forall r _h &\in  \mathring{ V } _h ^k ,  \label{eq:afw2_wave_r} \\
      ( \dot{ \rho } _h , \eta _h ) _\Omega + ( \mathrm{d} p _h , \eta _h ) _\Omega &= 0, \quad &\forall \eta _h &\in \mathring{ V } _h ^{ k + 1 } ,  \label{eq:afw2_wave_eta} \\
      - ( \dot{ p } _h , v _h ) _\Omega + ( \mathrm{d} \sigma _h , v _h ) _\Omega + ( \rho _h , \mathrm{d} v _h ) _\Omega &= \bigl( f ( t, u _h ) , v _h \bigr) _\Omega , \quad &\forall v _h &\in \mathring{ V } _h ^k ,  \label{eq:afw2_wave_v} 
    \end{alignat}
  \end{subequations}
  containing only dynamical equations and no constraints.  For
  $ k = 0 $, this coincides with \citet[Equation 5]{SaVa24}.

  In the linear case where $f = f (t) $, this second formulation
  allows us to eliminate the variable $ u _h $; if desired, it may be
  recovered by integrating $ p _h $ over time. Modulo notation and
  sign conventions, this is a hybridization of the conforming method
  for the linear Hodge wave equation given in \citet[Equation
  8.6]{Arnold2018}; see also \citet[Equation
  4.7]{Quenneville-Belair2015}.

  We conclude by showing that the unused trace variables
  $ \widehat{ \sigma } _h ^{\mathrm{nor}} $ and
  $ \widehat{ \rho } _h ^{\mathrm{tan}} $ may be determined in order
  to satisfy the constraints discussed in \cref{rmk:sigmanor_rhotan}.

  \begin{proposition}
    \label{prop:afw-h_sigmanor_rhotan}
    Given a solution to the AFW-H method, there exist
    $ \widehat{ \sigma } _h ^{\mathrm{nor}} \in \widehat{ W } _h ^{ k
      - 2 , \mathrm{nor} } $ and
    $ \widehat{ \rho } _h ^{\mathrm{tan}} \in \widehat{ V } _h ^{k+1,
      \mathrm{tan}} $ satisfying \eqref{eq:sigmanor_rhotan_boundary},
    such that $ \widehat{ \sigma } _h ^{\mathrm{nor}} $ satisfies the
    weak conservativity condition
    \begin{equation*}
      \langle \widehat{ \sigma } _h ^{\mathrm{nor}} , \widehat{ w } _h ^{\mathrm{tan}} \rangle _{ \partial \mathcal{T} _h } = 0 , \quad \forall \widehat{ w } _h ^{\mathrm{tan}} \in \ringhat{V} _h ^{k-2, \mathrm{tan}} .
    \end{equation*}
  \end{proposition}

  \begin{proof}
    The right-hand side of \eqref{eq:sigmanor_unor} vanishes whenever
    $ w _h ^{\mathrm{tan}} = 0 $, since this implies
    $ \mathrm{d} w _h ^{\mathrm{tan}} = 0 $, so it is a well-defined
    functional on $ W _h ^{ k - 2 , \mathrm{tan} } $.  Since
    $ \langle \cdot , \cdot \rangle _{ \partial \mathcal{T} _h } $ is
    an inner product on
    $ \widehat{ W } _h ^{ k - 2 , \mathrm{nor} } = W _h ^{ k - 2 ,
      \mathrm{tan} } $, the Riesz representation theorem gives a
    unique $ \widehat{ \sigma } _h ^{\mathrm{nor}} $ satisfying
    \eqref{eq:sigmanor_unor}. Furthermore, since
    $ w _h \in \mathring{ V } _h ^{ k - 2 } $ implies
    $ \mathrm{d} w _h \in \mathring{ V } _h ^{ k - 1 } $, equation
    \eqref{eq:afw-h_wave_tautan} with
    $ \widehat{ \tau } _h ^{\mathrm{tan}} = \mathrm{d} w _h
    ^{\mathrm{tan}} $ implies conservativity of
    $ \widehat{ \sigma } _h ^{\mathrm{nor}} $.

    Next, \eqref{eq:afw-h_wave_eta} and \eqref{eq:afw-h_wave_etanor}
    imply that $ \rho _h = - \mathrm{d} u _h \in V _h ^{ k + 1 }
    $. Hence, for any $ w _h \in W _h ^{ k + 2 } $, we have
    \begin{equation*}
      \langle \rho _h ^{\mathrm{tan}} , w _h ^{\mathrm{nor}} \rangle _{ \partial \mathcal{T} _h } = ( \mathrm{d} \rho _h , w _h ) _{ \mathcal{T} _h } - ( \rho _h , \delta w _h ) _{ \mathcal{T} _h } = \langle \widehat{ u } _h ^{\mathrm{tan}} , \delta w _h ^{\mathrm{nor}} \rangle _{ \partial \mathcal{T} _h } ,
    \end{equation*}
    since $ \mathrm{d} \rho _h = - \mathrm{d} \mathrm{d} u _h = 0 $.
    Thus,
    $ \widehat{ \rho } _h ^{\mathrm{tan}} = \rho _h ^{\mathrm{tan}}
    \in \widehat{ V } _h ^{ k + 1 , \mathrm{tan} } $ satisfies
    \eqref{eq:rhotan_utan}, which completes the proof.
  \end{proof}

  \subsubsection{A multisymplectic LDG-H method}
  \label{sec:ldg-h_wave}
  
  Next, we consider an LDG-H method given by the fluxes
  \begin{subequations}
  \begin{align}
    \Phi _q ^{ k -1 } ( \mathbf{z} _h , \widehat{ \mathbf{z} } _h ) &= ( \widehat{ u } _h ^{\mathrm{nor}} - u _h ^{\mathrm{nor}} ) + \alpha ^{ k -1 } ( \widehat{ \sigma } _h ^{\mathrm{tan}} - \sigma _h ^{\mathrm{tan}} ) , \label{eq:ldg-h_wave_flux_qkm}\\
    \Phi _q ^k ( \mathbf{z} _h , \widehat{ \mathbf{z} } _h ) &= ( \widehat{ \rho } _h ^{\mathrm{nor}} - \rho _h ^{\mathrm{nor}} ) + \alpha ^k ( \widehat{ u } _h ^{\mathrm{tan}} - u _h ^{\mathrm{tan}} ) ,\label{eq:ldg-h_wave_flux_qk}
  \end{align}
  \end{subequations}
  where $ \alpha ^{ k -1 } $ and $ \alpha ^k $ are symmetric operators
  on $ L ^2 \Lambda ^{ k -1 } ( \partial \mathcal{T} _h ) $ and
  $ L ^2 \Lambda ^k ( \partial \mathcal{T} _h ) $, respectively. By
  the symmetry of these operators, variations satisfy
  \begin{multline*}
    \langle \widehat{ \tau } _1 ^{\mathrm{tan}} - \tau _1 ^{\mathrm{tan}} , \widehat{ v } _2 ^{\mathrm{nor}} - v _2 ^{\mathrm{nor}} \rangle _{ \partial K } - \langle \widehat{ \tau } _2 ^{\mathrm{tan}} - \tau _2 ^{\mathrm{tan}} , \widehat{ v } _1 ^{\mathrm{nor}} - v _1 ^{\mathrm{nor}} \rangle _{ \partial K } \\
    = \bigl\langle \alpha ^{ k -1 } ( \widehat{ \tau } _1 ^{\mathrm{tan}} - \tau _1 ^{\mathrm{tan}} ) , \widehat{ \tau } _2 ^{\mathrm{tan}} - \tau _2 ^{\mathrm{tan}} \bigr\rangle _{ \partial K } - \bigl\langle \alpha ^{ k -1 } ( \widehat{ \tau } _2 ^{\mathrm{tan}} - \tau _2 ^{\mathrm{tan}} ) , \widehat{ \tau } _1 ^{\mathrm{tan}} - \tau _1 ^{\mathrm{tan}} \bigr\rangle _{ \partial K } = 0 ,
  \end{multline*}
  and similarly,
  \begin{multline*}
    \langle \widehat{ v } _1 ^{\mathrm{tan}} - v _1 ^{\mathrm{tan}} , \widehat{ \eta } _2 ^{\mathrm{nor}} - \eta _2 ^{\mathrm{nor}} \rangle _{ \partial K } - \langle \widehat{ v } _2 ^{\mathrm{tan}} - v _2 ^{\mathrm{tan}} , \widehat{ \eta } _1 ^{\mathrm{nor}} - \eta _1 ^{\mathrm{nor}} \rangle _{ \partial K } \\
    = \bigl\langle \alpha ^k ( \widehat{ v } _1 ^{\mathrm{tan}} - v _1 ^{\mathrm{tan}} ) , \widehat{ v } _2 ^{\mathrm{tan}} - v _2 ^{\mathrm{tan}} \bigr\rangle _{ \partial K } - \bigl\langle \alpha ^k ( \widehat{ v } _2 ^{\mathrm{tan}} - v _2 ^{\mathrm{tan}} ) , \widehat{ v } _1 ^{\mathrm{tan}} - v _1 ^{\mathrm{tan}} \bigr\rangle _{ \partial K } = 0 ,
  \end{multline*}
  so the multisymplecticity condition \eqref{eq:mscl_jump_hodge}
  holds.
  
  As before, we take $ \dot{ \widehat{ u } } _h = \widehat{ p } _h $
  in order to eliminate \eqref{eq:weak_wave_sigmadot} and
  \eqref{eq:weak_wave_rhodot}; we also eliminate the normal trace
  variables and integrate by parts as in \eqref{eq:ldg-h}. This yields
  the dynamical equations
  \begin{subequations}
    \label{eq:ldg-h_wave}
    \begin{alignat}{2}
      ( \dot{u} _h , r _h ) _{ \mathcal{T} _h } &= ( p _h , r _h ) _{ \mathcal{T} _h } , \quad &\forall r _h &\in W _h ^k , \label{eq:ldg-h_wave_r} \\
      \begin{multlined}[b]
        - ( \dot{p} _h , v _h ) _{ \mathcal{T} _h } + ( \sigma _h , \delta v _h ) _{ \mathcal{T} _h } + ( \delta \rho _h , v _h ) _{ \mathcal{T} _h } \\
        + \langle \widehat{ \sigma } _h ^{\mathrm{tan}} , v _h ^{\mathrm{nor}} \rangle _{ \partial \mathcal{T} _h } + \bigl\langle \alpha ^k ( \widehat{ u } _h ^{\mathrm{tan}} - u _h ^{\mathrm{tan}} ) , v _h ^{\mathrm{tan}} \bigr\rangle _{ \partial \mathcal{T} _h }
      \end{multlined} &= \bigl( f ( t, u _h ) , v _h \bigr) _{ \mathcal{T} _h } , \quad &\forall v _h &\in W _h ^k , \label{eq:ldg-h_wave_v} \\
      \intertext{together with the constraints}
      ( \delta u _h , \tau _h ) _{ \mathcal{T} _h } + \bigl\langle \alpha ^{ k -1 } ( \widehat{ \sigma } _h ^{\mathrm{tan}} - \sigma _h ^{\mathrm{tan}} ) , \tau _h ^{\mathrm{tan}} \bigr\rangle _{ \partial \mathcal{T} _h } &= - ( \sigma _h , \tau _h ) _{ \mathcal{T} _h } , \quad &\forall \tau _h &\in W _h ^{ k -1 } , \label{eq:ldg-h_wave_tau} \\
      ( u _h , \delta \eta _h ) _{ \mathcal{T} _h } + \langle \widehat{ u } _h ^{\mathrm{tan}} , \eta _h ^{\mathrm{nor}} \rangle _{ \partial \mathcal{T} _h } &= - ( \rho _h , \eta _h )  _{ \mathcal{T} _h } , \quad &\forall \eta _h &\in W _h ^{ k + 1 } , \label{eq:ldg-h_wave_eta} \\
      \intertext{and the conservativity conditions}
      \bigl\langle u _h ^{\mathrm{nor}} - \alpha ^{ k -1 } ( \widehat{ \sigma } _h ^{\mathrm{tan}} - \sigma _h ^{\mathrm{tan}} ) , \widehat{ \tau } _h ^{\mathrm{tan}} \bigr\rangle _{ \partial \mathcal{T} _h } &= 0 , \quad &\forall \widehat{ \tau } _h ^{\mathrm{tan}} &\in \ringhat{V} _h ^{k-1, \mathrm{tan}}, \label{eq:ldg-h_wave_tautan} \\
      \bigl\langle \rho _h ^{\mathrm{nor}} - \alpha ^k ( \widehat{ u } _h ^{\mathrm{tan}} - u _h ^{\mathrm{tan}} ) , \widehat{ v } _h ^{\mathrm{tan}} \bigr\rangle _{ \partial \mathcal{T} _h } &= 0 , \quad &\forall \widehat{ v } _h ^{\mathrm{tan}} &\in \ringhat{V} _h ^{k, \mathrm{tan}} . \label{eq:ldg-h_wave_vtan}
    \end{alignat}
  \end{subequations}
  The case $ k = 0 $ recovers the Hamiltonian LDG-H method for the
  semilinear scalar wave equation of \citet[Equation 4]{SaVa24}, which
  in the linear case is that of \citet[Equation 10]{SaCiNgPe17}. The
  foregoing results of this section show that this method is strongly
  multisymplectic, which strengthens the results of
  \citep{SaCiNgPe17,SaVa24} showing that they are symplectic. This can
  also be deduced from the multisymplecticity results of
  \citet[Section 4.5]{McSt2024} for scalar problems.

  The following result gives a sufficient condition on the penalty
  operators for \eqref{eq:ldg-h_wave_tau}--\eqref{eq:ldg-h_wave_vtan}
  to be solved uniquely in terms of $ u _h $, and thus for the
  dynamics to be well-defined.

  \begin{theorem}
    \label{thm:ldg-h_wave_definite}
    If $ \alpha ^{ k -1 } $ is negative-definite and $\alpha ^k $ is
    positive-definite, then for all $ u _h \in W _h ^k $, there exist
    unique $ \sigma _h \in W _h ^{ k -1 } $,
    $ \rho _h \in W _h ^{ k + 1 } $,
    $ \widehat{ \sigma } _h ^{\mathrm{tan}} \rvert _{ \partial
      \mathcal{T} _h \setminus \partial \Omega } \in \ringhat{V} _h
    ^{k-1, \mathrm{tan}} $, and
    $ \widehat{ u } _h ^{\mathrm{tan}} \rvert _{ \partial \mathcal{T}
      _h \setminus \partial \Omega } \in \ringhat{V} _h ^{k,
      \mathrm{tan}} $ satisfying
    \eqref{eq:ldg-h_wave_tau}--\eqref{eq:ldg-h_wave_vtan}. Hence,
    given $f$ and boundary conditions
    $ \widehat{ \sigma } _h ^{\mathrm{tan}} \rvert _{ \partial \Omega
    } $ and
    $ \widehat{ u } _h ^{\mathrm{tan}} \rvert _{ \partial \Omega } $,
    the remaining equations
    \eqref{eq:ldg-h_wave_r}--\eqref{eq:ldg-h_wave_v} give well-defined
    dynamics for $ ( u _h , p _h ) \in \mathbf{W} _h ^k $.
  \end{theorem}

  \begin{proof}
    Adding \eqref{eq:ldg-h_wave_tau} and \eqref{eq:ldg-h_wave_tautan}
    and rearranging gives
    \begin{equation*}
      - ( \sigma _h , \tau _h ) _{ \mathcal{T} _h } + \bigl\langle \alpha ^{ k -1 } ( \widehat{ \sigma } _h ^{\mathrm{tan}} - \sigma _h ^{\mathrm{tan}} ) , \widehat{ \tau } _h ^{\mathrm{tan}} - \tau _h ^{\mathrm{tan}} \bigr\rangle _{ \partial \mathcal{T} _h } = ( \delta u _h , \tau _h ) _{ \mathcal{T} _h } + \langle u _h ^{\mathrm{nor}} , \widehat{ \tau } _h ^{\mathrm{tan}} \rangle _{ \partial \mathcal{T} _h } .
    \end{equation*}
    If $ \alpha ^{ k -1 } $ is negative-definite, then so is the
    bilinear form on the left-hand side. Hence, we can solve uniquely
    for $ \sigma _h $ and
    $ \widehat{ \sigma } _h ^{\mathrm{tan}} \rvert _{ \partial
      \mathcal{T} _h \setminus \partial \Omega } $ in terms of
    $ u _h $. Next, subtracting \eqref{eq:ldg-h_wave_eta} from
    \eqref{eq:ldg-h_wave_vtan} and rearranging,
    \begin{equation*}
      ( \rho _h , \eta _h ) _{ \mathcal{T} _h } + \langle \alpha ^k \widehat{ u } _h ^{\mathrm{tan}} , \widehat{ v } _h ^{\mathrm{tan}} \rangle _{ \partial \mathcal{T} _h } + \langle \widehat{ u } _h ^{\mathrm{tan}} , \eta _h ^{\mathrm{nor}} \rangle _{ \partial \mathcal{T} _h } - \langle \rho _h ^{\mathrm{nor}} , \widehat{ v } _h ^{\mathrm{tan}} \rangle _{ \partial \mathcal{T} _h } = - ( u _h , \delta \eta _h ) _{ \mathcal{T} _h } + \langle \alpha ^k u _h ^{\mathrm{tan}} , \widehat{ v } _h ^{\mathrm{tan}} \rangle _{ \partial \mathcal{T} _h } .
    \end{equation*}
    If $ \alpha ^k $ is positive-definite, then so is the bilinear
    form on the left-hand side. (Observe that the last two
    left-hand-side terms cancel when $ \eta _h = \rho _h $ and
    $ \widehat{ v } _h ^{\mathrm{tan}} = \widehat{ u } _h
    ^{\mathrm{tan}} $.) Hence, we can solve uniquely for $ \rho _h $
    and $ \widehat{ u } _h ^{\mathrm{tan}} $ in terms of $u _h $,
    which completes the proof.
  \end{proof}

  \begin{remark}
    This proof is easily adapted to natural boundary conditions for
    $ \widehat{ u } _h ^{\mathrm{nor}} $ and
    $ \widehat{ \rho } _h ^{\mathrm{nor}} $ on $ \partial \Omega $. In
    that case, we would replace \eqref{eq:ldg-h_wave_tautan} and
    \eqref{eq:ldg-h_wave_vtan} by
    \begin{alignat}{2}
      \bigl\langle u _h ^{\mathrm{nor}} - \alpha ^{ k -1 } ( \widehat{ \sigma } _h ^{\mathrm{tan}} - \sigma _h ^{\mathrm{tan}} ) , \widehat{ \tau } _h ^{\mathrm{tan}} \bigr\rangle _{ \partial \mathcal{T} _h } &= \langle \widehat{ u } _h ^{\mathrm{nor}} , \widehat{ \tau } _h ^{\mathrm{tan}} \rangle _{ \partial \Omega } , \quad &\forall \widehat{ \tau } _h ^{\mathrm{tan}} &\in \widehat{V} _h ^{k-1, \mathrm{tan}}, \tag{\ref*{eq:ldg-h_wave_tautan}$^\prime$} \\
      \bigl\langle \rho _h ^{\mathrm{nor}} - \alpha ^k ( \widehat{ u } _h ^{\mathrm{tan}} - u _h ^{\mathrm{tan}} ) , \widehat{ v } _h ^{\mathrm{tan}} \bigr\rangle _{ \partial \mathcal{T} _h } &= \langle \widehat{ \rho } _h ^{\mathrm{nor}} , \widehat{ v } _h ^{\mathrm{tan}} \rangle _{ \partial \Omega } , \quad &\forall \widehat{ v } _h ^{\mathrm{tan}} &\in \widehat{V} _h ^{k, \mathrm{tan}} . \tag{\ref*{eq:ldg-h_wave_vtan}$^\prime$}
    \end{alignat}
    We would then solve for
    $ \widehat{ \sigma } _h ^{\mathrm{tan}} \in \widehat{ V } _h
    ^{k-1, \mathrm{tan}} $ and
    $ \widehat{ u } _h ^{\mathrm{tan}} \in \widehat{ V } _h ^{k,
      \mathrm{tan}} $ on all of $ \partial \mathcal{T} _h $ rather
    than just $ \partial \mathcal{T} _h \setminus \partial \Omega $.
  \end{remark}

  \begin{proposition}
    \label{prop:ldg-h_sigmanor_rhotan}
    Given a solution to the LDG-H method, there exist
    $ \widehat{ \sigma } _h ^{\mathrm{nor}} \in \widehat{ W } _h ^{ k
      - 2 , \mathrm{nor} } $ and
    $ \widehat{ \rho } _h ^{\mathrm{tan}} \in \widehat{ V } _h ^{k+1,
      \mathrm{tan}} $ satisfying \eqref{eq:sigmanor_rhotan_boundary},
    such that $ \widehat{ \sigma } _h ^{\mathrm{nor}} $ satisfies the
    strong conservativity condition
    $ \llbracket \widehat{ \sigma } _h ^{\mathrm{nor}} \rrbracket = 0
    $.
  \end{proposition}

  \begin{proof}
    First, by the definition of tangential jump, any
    $ \widehat{ \sigma } _h ^{\mathrm{nor}} \in \llbracket W _h ^{ k -
      2 , \mathrm{tan} } \rrbracket \coloneqq \bigl\{ \llbracket w _h
    ^{\mathrm{tan}} \rrbracket : w _h \in W _h ^{ k - 2 } \bigr\} $
    will satisfy the strong conservativity condition. We claim that
    there exists a unique such
    $ \widehat{ \sigma } _h ^{\mathrm{nor}} $ satisfying
    \eqref{eq:sigmanor_unor}. Since
    $ \llbracket \widehat{ \sigma } _h ^{\mathrm{nor}} \rrbracket = 0
    $ and
    $ \llbracket \widehat{ u } _h ^{\mathrm{nor}} \rrbracket = 0 $,
    \citep[Proposition 3.4]{StZa2025} implies that
    \eqref{eq:sigmanor_unor} becomes
    \begin{equation*}
      \bigl\langle \widehat{ \sigma } _h ^{\mathrm{nor}} , \llbracket w _h ^{\mathrm{tan}} \rrbracket \bigr\rangle _{ \partial \mathcal{T} _h } = \bigl\langle \widehat{ u } _h ^{\mathrm{nor}} , \llbracket \mathrm{d} w _h ^{\mathrm{tan}} \rrbracket \bigr\rangle _{ \partial \mathcal{T} _h } , \quad \forall w _h \in W _h ^{ k - 2 } .
    \end{equation*}
    Now, the right-hand side vanishes whenever
    $ \llbracket w _h ^{\mathrm{tan}} \rrbracket = 0 $, since this
    implies
    $ \llbracket \mathrm{d} w _h ^{\mathrm{tan}} \rrbracket = 0 $, so
    it is a well-defined functional on
    $ \llbracket W _h ^{ k - 2, \mathrm{tan} } \rrbracket $. Hence,
    existence and uniqueness of
    $ \widehat{ \sigma } _h ^{\mathrm{nor}} \in \llbracket W _h ^{ k -
      2, \mathrm{tan} } \rrbracket $ follows from the Riesz
    representation theorem.

    By a similar argument, when
    $ \widehat{ \rho } _h ^{\mathrm{tan}} \rvert _{ \partial
      \mathcal{T} _h \setminus \partial \Omega } \in \llbracket W _h
    ^{k+2, \mathrm{nor}} \rrbracket $, the condition
    \eqref{eq:rhotan_utan} becomes
    \begin{equation*}
      \bigl\langle \widehat{ \rho } _h ^{\mathrm{tan}} , \llbracket w _h ^{\mathrm{nor}} \rrbracket \bigr\rangle _{ \partial \mathcal{T} _h \setminus \partial \Omega } + \langle \widehat{ \rho } _h ^{\mathrm{tan}} , w _h ^{\mathrm{nor}} \rangle _{ \partial \Omega } = \bigl\langle \widehat{ u } _h ^{\mathrm{tan}} , \llbracket \delta w _h ^{\mathrm{nor}} \rrbracket \bigr\rangle _{ \partial \mathcal{T} _h \setminus \partial \Omega } + \langle \widehat{ u } _h ^{\mathrm{tan}} , \delta w _h ^{\mathrm{nor}} \rangle _{ \partial \Omega }, \quad \forall w _h \in W _h ^{ k + 2 } .
    \end{equation*}
    (Here, the boundary terms must be handled separately, since
    $ \llbracket w _h ^{\mathrm{nor}} \rrbracket $ is defined to
    vanish on $ \partial \Omega $, cf.~\citep[Definition
    3.2]{StZa2025}.) The right-hand side vanishes whenever
    $ \llbracket w _h ^{\mathrm{nor}} \rrbracket = 0 $, since this
    implies $ \llbracket \delta w _h ^{\mathrm{nor}} \rrbracket = 0 $,
    so it is a well-defined functional on the elements of
    $ \widehat{ V } _h ^{ k + 1 , \mathrm{tan}} $ extending
    $ \llbracket W _h ^{ k + 2 , \mathrm{nor} } \rrbracket $. Hence,
    the Riesz representation theorem gives a unique such
    $ \widehat{ \rho } _h ^{\mathrm{tan}} $, which completes the
    proof.
  \end{proof}
  
  Unlike with the AFW-H method, however, the form of the flux prevents
  us from simply eliminating the constraints to get a formulation
  involving $ \widehat{ p } _h ^{\mathrm{tan}} $ rather than
  $ \widehat{ u } _h ^{\mathrm{tan}} $. Indeed, even if we eliminate
  the constraints, $ \widehat{ u } _h ^{\mathrm{tan}} $ still appears
  on the left-hand side of \eqref{eq:ldg-h_wave_vtan}. Obtaining a
  formulation in $ \widehat{ p } _h ^{\mathrm{tan}} $ alone requires
  adopting a different flux that, as we shall see next, fails to be
  multisymplectic.
  
  \subsubsection{A non-multisymplectic LDG-H method}
  Let us take the fluxes
  \begin{align*}
    \Phi _q ^{ k -1 } ( \mathbf{z} _h , \widehat{ \mathbf{z} } _h ) &= ( \widehat{ p } _h ^{\mathrm{nor}} - p _h ^{\mathrm{nor}} ) + \alpha ^{ k -1 } ( \widehat{ \sigma } _h ^{\mathrm{tan}} - \sigma _h ^{\mathrm{tan}} ) ,\\
    \Phi _q ^k ( \mathbf{z} _h , \widehat{ \mathbf{z} } _h ) &= ( \widehat{ \rho } _h ^{\mathrm{nor}} - \rho _h ^{\mathrm{nor}} ) + \alpha ^k ( \widehat{ p } _h ^{\mathrm{tan}} - p _h ^{\mathrm{tan}} ) ,
  \end{align*}
  where again $ \alpha ^{ k -1 } $ and $ \alpha ^k $ are symmetric
  operators on $ L ^2 \Lambda ^{ k -1 } ( \partial \mathcal{T} _h ) $
  and $ L ^2 \Lambda ^k ( \partial \mathcal{T} _h ) $, respectively.

  Suppose we eliminate the constraints \eqref{eq:weak_wave_thetadot}
  and \eqref{eq:weak_wave_xidot}, assuming that they hold at the
  initial time, as well as eliminating the normal trace variables and
  integrating by parts as in \eqref{eq:ldg-h}. The resulting method
  has the dynamical equations
  \begin{subequations}
    \label{eq:ldg-h_wave_mixed}
    \begin{alignat}{2}
      ( \dot{ \sigma } _h , \tau _h ) _{ \mathcal{T} _h } + ( \delta p _h , \tau _h ) _{ \mathcal{T} _h } + \bigl\langle \alpha ^{ k -1 } ( \widehat{ \sigma } _h ^{\mathrm{tan}} - \sigma _h ^{\mathrm{tan}} ) , \tau _h ^{\mathrm{tan}} \bigr\rangle _{ \partial \mathcal{T} _h } &= 0, \quad &\forall \tau _h &\in W _h ^{ k -1 } , \label{eq:ldg-h_wave_mixed_tau} \\
      ( \dot{ u } _h , r _h ) _{ \mathcal{T} _h } &= ( p _h , r _h ) _{ \mathcal{T} _h } , \quad &\forall r _h &\in W _h ^k , \label{eq:ldg-h_wave_mixed_r} \\
      ( \dot{ \rho } _h , \eta _h ) _{ \mathcal{T} _h } + ( p _h , \delta \eta _h ) _{ \mathcal{T} _h } + \langle \widehat{ p } _h ^{\mathrm{tan}} , \eta _h ^{\mathrm{nor}} \rangle _{ \partial \mathcal{T} _h } &= 0, \quad &\forall \eta _h &\in W _h ^{ k + 1 } , \label{eq:ldg-h_wave_mixed_eta} \\
    \begin{multlined}[b]
      - ( \dot{ p } _h , v _h ) _{ \mathcal{T} _h } + ( \sigma _h , \delta v _h ) _{ \mathcal{T} _h } + ( \delta \rho _h , v _h ) _{ \mathcal{T} _h } \\
      + \langle \widehat{ \sigma } _h ^{\mathrm{tan}} , v _h ^{\mathrm{nor}} \rangle _{ \partial \mathcal{T} _h } + \bigl\langle \alpha ^k ( \widehat{ p } _h ^{\mathrm{tan}} - p _h ^{\mathrm{tan}} ) , v _h ^{\mathrm{tan}} \bigr\rangle _{ \partial \mathcal{T} _h }
    \end{multlined} &= \bigl( f ( t, u _h ) , v _h \bigr)  _{ \mathcal{T} _h } ,\quad &\forall v _h &\in W _h ^k , \label{eq:ldg-h_wave_mixed_v}\\
      \intertext{together with the conservativity conditions}
      \bigl\langle p _h ^{\mathrm{nor}} - \alpha ^{ k -1 } ( \widehat{ \sigma } _h ^{\mathrm{tan}} - \sigma _h ^{\mathrm{tan}} ) , \widehat{ \tau } _h ^{\mathrm{tan}} \bigr\rangle _{ \partial \mathcal{T} _h } &= 0 , \quad &\forall \widehat{ \tau } _h ^{\mathrm{tan}} &\in \ringhat{V} _h ^{k-1, \mathrm{tan}}, \label{eq:ldg-h_wave_mixed_tautan}\\
      \bigl\langle \rho _h ^{\mathrm{nor}} - \alpha ^k ( \widehat{ p } _h ^{\mathrm{tan}} - p _h ^{\mathrm{tan}} ) , \widehat{ v } _h ^{\mathrm{tan}} \bigr\rangle _{ \partial \mathcal{T} _h } &= 0 , \quad &\forall \widehat{ v } _h ^{\mathrm{tan}} &\in \ringhat{V} _h ^{k, \mathrm{tan}} . \label{eq:ldg-h_wave_mixed_vtan}
    \end{alignat}
  \end{subequations}
  For $ k = 0 $, this coincides with \citet[Equation 6]{SaVa24}. In
  the linear case where $ f = f (t) $, this formulation allows us to
  eliminate the variable $ u _h $ and the equation
  \eqref{eq:ldg-h_wave_mixed_r}, evolving only the remaining
  variables; for $ k = 0 $, this recovers \citet[Equation
  5]{NgPeCo2011_wave}. Again, as with the second formulation of AFW-H,
  we may recover $ u _h $, if desired, by integrating $ p _h $ over
  time.

  For this formulation, we may provide arbitrary initial conditions
  for $ u _h $, $ p _h $, and the single-valued traces
  $ \widehat{ u } _h ^{\mathrm{nor}} $ and
  $ \widehat{ u } _h ^{\mathrm{tan}} $; solve the constraint equations
  \eqref{eq:weak_wave_thetadot} and \eqref{eq:weak_wave_xidot} to
  obtain initial values for $ \sigma _h $ and $ \rho _h $,
  respectively; and then evolve forward according to
  \eqref{eq:ldg-h_wave_mixed}. Notice that, if $ \alpha ^{ k -1 } $
  and $ \alpha ^k $ are nondegenerate on
  $ \ringhat{V} _h ^{k-1, \mathrm{tan}} $ and
  $ \ringhat{V} _h ^{k, \mathrm{tan}} $, then we can solve
  \eqref{eq:ldg-h_wave_mixed_tautan} and
  \eqref{eq:ldg-h_wave_mixed_vtan} for
  $ \widehat{ \sigma } _h ^{\mathrm{tan}} \rvert _{ \partial
    \mathcal{T} _h \setminus \partial \Omega } $ and
  $ \widehat{ p } _h ^{\mathrm{tan}} \rvert _{ \partial \mathcal{T} _h
    \setminus \partial \Omega } $ in terms of the remaining variables.
  While this is sufficient for well-defined dynamics, the stronger
  definiteness hypotheses of \cref{thm:ldg-h_wave_definite} imply
  \emph{stable} dynamics, as follows.

  \begin{lemma}
    \label{lem:ldg-h_wave_mixed_energy}
    Consider \eqref{eq:ldg-h_wave_mixed} for the case of the
    homogeneous linear Hodge wave equation $ f = 0 $ with homogeneous
    Dirichlet boundary conditions
    $ \widehat{ \sigma } _h ^{\mathrm{tan}} \in \ringhat{V} _h ^{k-1,
      \mathrm{tan}} $ and
    $ \widehat{ p } _h ^{\mathrm{tan}} \in \ringhat{V} _h ^{k,
      \mathrm{tan}} $. Then
    \begin{multline}
      \label{eq:ldg-h_wave_mixed_energy}
      \frac{\mathrm{d}}{\mathrm{d}t} \frac{1}{2} \bigl( \lVert \sigma _h \rVert _{ \mathcal{T} _h } ^2 + \lVert p _h \rVert _{ \mathcal{T} _h } ^2 + \lVert \rho _h \rVert _{ \mathcal{T} _h } ^2 \bigr) \\
      = \bigl\langle \alpha ^{ k -1 } ( \widehat{ \sigma } _h ^{\mathrm{tan}} - \sigma _h ^{\mathrm{tan}} ), \widehat{ \sigma } _h ^{\mathrm{tan}} - \sigma _h ^{\mathrm{tan}} \bigr\rangle _{ \partial \mathcal{T} _h } - \bigl\langle \alpha ^k ( \widehat{ p } _h ^{\mathrm{tan}} - p _h ^{\mathrm{tan}} ) , \widehat{ p } _h ^{\mathrm{tan}} - p _h ^{\mathrm{tan}} \bigr\rangle _{ \partial \mathcal{T} _h } .
    \end{multline}
  \end{lemma}

  \begin{proof}
    By \eqref{eq:ldg-h_wave_mixed_tau} with $ \tau _h = \sigma _h $,
    \eqref{eq:ldg-h_wave_mixed_v} with $ v _h = p _h $, and
    \eqref{eq:ldg-h_wave_mixed_eta} with $ \eta _h = \rho _h $, we
    have
    \begin{multline*}
      \frac{\mathrm{d}}{\mathrm{d}t} \frac{1}{2} \bigl( \lVert \sigma _h \rVert _{ \mathcal{T} _h } ^2 + \lVert p _h \rVert _{ \mathcal{T} _h } ^2 + \lVert \rho _h \rVert _{ \mathcal{T} _h } ^2 \bigr) \\
      = \langle \widehat{ \sigma } _h ^{\mathrm{tan}} , p _h ^{\mathrm{nor}} \rangle _{ \partial \mathcal{T} _h } - \langle \widehat{ p } _h ^{\mathrm{tan}} , \rho _h ^{\mathrm{nor}} \rangle _{ \partial \mathcal{T} _h } - \bigl\langle \alpha ^{ k -1 } ( \widehat{ \sigma } _h ^{\mathrm{tan}} - \sigma _h ^{\mathrm{tan}} ), \sigma _h ^{\mathrm{tan}} \bigr\rangle _{ \partial \mathcal{T} _h } + \bigl\langle \alpha ^k ( \widehat{ p } _h ^{\mathrm{tan}} - p _h ^{\mathrm{tan}} ) , p _h ^{\mathrm{tan}} \bigr\rangle _{ \partial \mathcal{T} _h } .
    \end{multline*}
    Applying \eqref{eq:ldg-h_wave_mixed_tautan} with
    $ \widehat{ \tau } _h ^{\mathrm{tan}} = \widehat{ \sigma } _h
    ^{\mathrm{tan}} $ and \eqref{eq:ldg-h_wave_mixed_vtan} with
    $ \widehat{ v } _h ^{\mathrm{tan}} = \widehat{ p } _h
    ^{\mathrm{tan}} $ gives \eqref{eq:ldg-h_wave_mixed_energy}.
  \end{proof}

  \begin{corollary}
    If $ \alpha ^{ k -1 } $ is negative-definite and $ \alpha ^k $ is
    positive-definite, then under the hypotheses of
    \cref{lem:ldg-h_wave_mixed_energy}, we have
    \begin{equation*}
      \frac{\mathrm{d}}{\mathrm{d}t} \frac{1}{2} \bigl( \lVert \sigma _h \rVert _{ \mathcal{T} _h } ^2 + \lVert p _h \rVert _{ \mathcal{T} _h } ^2 + \lVert \rho _h \rVert _{ \mathcal{T} _h } ^2 \bigr) \leq 0 ,
    \end{equation*}
    with equality if and only if
    $ \widehat{ \sigma } _h ^{\mathrm{tan}} = \sigma _h
    ^{\mathrm{tan}} $ (i.e.,
    $ \widehat{ p } _h ^{\mathrm{nor}} = p _h ^{\mathrm{nor}} $) and
    $ \widehat{ p } _h ^{\mathrm{tan}} = p _h ^{\mathrm{tan}} $.
  \end{corollary}

  Finally, we consider the multisymplecticity of this method. Clearly,
  the flux is not multisymplectic, since it does not restrict any of
  the variables appearing in \eqref{eq:mscl_jump_hodge}: in
  particular, since $ r _i $ and $ \widehat{ r } _i $ appear in the
  flux but not in \eqref{eq:mscl_jump_hodge}, we may choose any
  $ \tau _i , \widehat{ \tau } _i ^{\mathrm{tan}} $, $ v _i $,
  $ \widehat{ v } _i $, $ \eta _i $, and
  $ \widehat{ \eta } _i ^{\mathrm{nor}} $ such that
  \eqref{eq:mscl_jump_hodge} fails. However, multisymplecticity of the
  flux (as in \cref{def:ms_flux}) is a sufficient but not necessary
  condition for multisymplecticity of the method. To prove that the
  method as a whole is non-multisymplectic, we must show that there
  exist variations of \eqref{eq:ldg-h_wave_mixed} such that
  \eqref{eq:mscl_jump_hodge} fails.

  \begin{theorem}
    If $ \alpha ^{ k -1 } $ is negative-definite and $ \alpha ^k $ is
    positive-definite, then the method \eqref{eq:ldg-h_wave_mixed} is
    not multisymplectic.
  \end{theorem}

  \begin{proof}
    Since we are free to choose arbitrary initial conditions for
    $ u _h $, $ p _h $, and $ \widehat{ u } _h $, it follows that we
    are free to do so for the corresponding variation components
    $ v _i $, $ r _i $, and $ \widehat{ v } _i $. We show that these
    initial conditions may be chosen such that
    \begin{multline*}
      \langle \widehat{ \tau } _1 ^{\mathrm{tan}} - \tau _1 ^{\mathrm{tan}} , \widehat{ v } _2 ^{\mathrm{nor}} - v _2 ^{\mathrm{nor}} \rangle _{ \partial \mathcal{T} _h } + \langle \widehat{ v } _1 ^{\mathrm{tan}} - v _1 ^{\mathrm{tan}} , \widehat{ \eta } _2 ^{\mathrm{nor}} - \eta _2 ^{\mathrm{nor}} \rangle _{ \partial \mathcal{T} _h } \\
      - \langle \widehat{ \tau } _2 ^{\mathrm{tan}} - \tau _2 ^{\mathrm{tan}} , \widehat{ v } _1 ^{\mathrm{nor}} - v _1 ^{\mathrm{nor}} \rangle _{ \partial \mathcal{T} _h } - \langle \widehat{ v } _2 ^{\mathrm{tan}} - v _2 ^{\mathrm{tan}} , \widehat{ \eta } _1 ^{\mathrm{nor}} - \eta _1 ^{\mathrm{nor}} \rangle _{ \partial \mathcal{T} _h }
    \end{multline*}
    is nonvanishing at the initial time, which implies that
    \eqref{eq:mscl_jump_hodge} fails to hold for some
    $ K \in \mathcal{T} _h $.

    First, initialize $ v _1 = 0 $ and $ \widehat{ v } _1 = 0 $. The
    constraints \eqref{eq:weak_wave_thetadot} and
    \eqref{eq:weak_wave_xidot} imply $ \tau _1 = 0 $ and
    $ \eta _1 = 0 $, respectively, so the expression above simplifies
    to
    \begin{equation*}
      \langle \widehat{ \tau } _1 ^{\mathrm{tan}} , \widehat{ v } _2 ^{\mathrm{nor}} - v _2 ^{\mathrm{nor}} \rangle _{ \partial \mathcal{T} _h } - \langle \widehat{ v } _2 ^{\mathrm{tan}} - v _2 ^{\mathrm{tan}} , \widehat{ \eta } _1 ^{\mathrm{nor}} \rangle _{ \partial \mathcal{T} _h } .
    \end{equation*}
    Next, given any initial condition for $ r _1 $ (which we have yet
    to specify), take $ v _2 = - r _1 $ and
    $ \widehat{ v } _2 = - \widehat{ r } _1 $. Substituting these
    above and applying the flux definitions for
    $ \widehat{ r } _1 ^{\mathrm{nor}} $ and
    $ \widehat{ \eta } _1 ^{\mathrm{nor}} $ gives
    \begin{equation*}
      \langle \alpha ^{ k -1 } \widehat{ \tau } _1 ^{\mathrm{tan}} , \widehat{ \tau } _1 ^{\mathrm{tan}} \rangle _{ \partial \mathcal{T} _h } - \bigl\langle \alpha ^k ( \widehat{ r } _1 ^{\mathrm{tan}} - r _1 ^{\mathrm{tan}} )  , \widehat{ r } _1 ^{\mathrm{tan}} - r _1 ^{\mathrm{tan}} \bigr\rangle _{ \partial \mathcal{T} _h } \leq 0 ,
    \end{equation*}
    with equality if and only if
    $ \widehat{ \tau } _1 ^{\mathrm{tan}} = 0 $ (i.e.,
    $ \widehat{ r } _1 ^{\mathrm{nor}} = r _1 ^{\mathrm{nor}} $) and
    $ \widehat{ r } _1 ^{\mathrm{tan}} = r _1 ^{\mathrm{tan}}
    $. Finally, since $ \widehat{ r } _1 ^{\mathrm{nor}} $ and
    $ \widehat{ r } _1 ^{\mathrm{tan}} $ are single-valued, equality
    holds only if $ r _1 ^{\mathrm{nor}} $ and
    $ r _1 ^{\mathrm{tan}} $ are also single valued, i.e., $ r _1 $ is
    continuous. Therefore, choosing any discontinuous initial
    condition for $ r _1 $ causes the expression above to be strictly
    negative at the initial time, and thus \eqref{eq:mscl_jump_hodge}
    fails to hold.
  \end{proof}

  \begin{remark}
    Existence of $ \widehat{ \sigma } _h ^{\mathrm{nor}} $ and
    $ \widehat{ \rho } _h ^{\mathrm{tan}} $ satisfying
    \eqref{eq:sigmanor_rhotan_boundary} is proved exactly as in
    \cref{prop:ldg-h_sigmanor_rhotan}, so we do not repeat the proof
    here.
  \end{remark}

\section{Global Hamiltonian structure preservation}
\label{sec:global_hamiltonian}

Time-dependent Hamiltonian PDEs are often viewed as ordinary
Hamiltonian dynamical systems evolving on some infinite-dimensional
function space. (See, for instance, \citet[Chapter 3]{MaRa1999} and
references therein.) From this viewpoint, structure-preserving
semidiscretization methods aim to approximate these
infinite-dimensional Hamiltonian systems by finite-dimensional
Hamiltonian systems (e.g., on a finite-dimensional subspace for
conforming Galerkin methods). This alternative approach gives a
\emph{global} symplectic conservation law on all of $\Omega$, but not
necessarily the finer \emph{local} structure of the multisymplectic
approach developed in the preceding sections.

In this section, we relate the two approaches. First, we give an
infinite-dimensional Hamiltonian description for the canonical systems
of \cref{sec:hamiltonian}. Next, we describe the multisymplectic
semidiscretization methods of \cref{sec:semidiscrete} as
finite-dimensional Hamiltonian systems. In both the infinite- and
finite-dimensional cases, the global symplectic conservation law is
seen to be a special case of the integral form of the multisymplectic
conservation law. This establishes that the multisymplectic approach
gives finer information about Hamiltonian structure preservation than
the global symplectic approach.

\subsection{The smooth setting}
\label{sec:global_hamiltonian_smooth}

Let
$ H \colon I \times \Omega \times \operatorname{\mathbf{Alt}}
\mathbb{R}^n \rightarrow \mathbb{R} $, as in
\cref{sec:hamiltonian}. Define the \emph{global Hamiltonian}
$ \mathcal{H} \colon I \times \mathbf{\Lambda} (\Omega) \rightarrow
\mathbb{R} $ to be the functional
\begin{equation*}
  \mathcal{H} ( t, \mathbf{z} ) \coloneqq \int _\Omega \Bigl[ H ( t, x, \mathbf{z} ) - \tfrac{1}{2} ( \mathbf{D} \mathbf{z} , \mathbf{z} ) \Bigr] \mathrm{vol} .
\end{equation*}
Letting $ \mathring{\mathbf{\Lambda}} (\Omega) $ denote the subspace
of forms having compact support in $\Omega$, the functional derivative
of $\mathcal{H}$ along
$ \mathbf{w} \in \mathring{ \mathbf{\Lambda} } (\Omega) $ is
\begin{equation*}
  \frac{ \partial \mathcal{H}  }{ \partial \mathbf{z} } \mathbf{w} \coloneqq \frac{\mathrm{d}}{\mathrm{d}\epsilon} \mathcal{H} ( t, x, \mathbf{z} + \epsilon \mathbf{w} ) \biggr\rvert _{ \epsilon = 0 } = \biggl( \frac{ \partial H }{ \partial \mathbf{z} } - \mathbf{D} \mathbf{z} , \mathbf{w} \biggr) _\Omega .
\end{equation*}
(In the literature, functional derivatives are often denoted
$ \delta \mathcal{H} / \delta \mathbf{z} $, but we have chosen the
notation above to avoid confusion with the codifferential $\delta$.)
Hence, \eqref{eq:hamiltonian_JD} is equivalent to
\begin{equation*}
  ( \mathbf{J} \dot{\mathbf{z}} , \mathbf{w} ) _\Omega = \frac{ \partial \mathcal{H} }{ \partial \mathbf{z} } \mathbf{w} , \quad \forall \mathbf{w} \in \mathring{\mathbf{\Lambda}} (\Omega) ,
\end{equation*}
which is the \emph{weak form} of Hamilton's equations on
$ \mathbf{\Lambda} (\Omega) $, cf.~\citet[p.~106]{MaRa1999}.

\begin{remark}
  If $\Omega$ is not compact, then the \emph{Hamiltonian density}
  $ H ( t, x, \mathbf{z} ) - \frac{1}{2} ( \mathbf{D} \mathbf{z} ,
  \mathbf{z} ) $ might not be integrable for all
  $ \mathbf{z} \in \mathbf{\Lambda} (\Omega) $, causing
  $ \mathcal{H} ( t, \mathbf{z} ) $ to be undefined. However, this is
  only a minor technical obstacle: we can still make sense of the
  functional derivative along
  $ \mathbf{w} \in \mathring{ \mathbf{\Lambda} } (\Omega) $ by
  restricting the integrals above to
  $ \operatorname{supp} \mathbf{w} $, which is compact.
\end{remark}

A compactly supported first variation
$ \mathbf{w} _i \colon I \rightarrow \mathring{\mathbf{\Lambda}}
(\Omega) $ of a solution to Hamilton's equations in weak form
satisfies
\begin{equation*}
  ( \mathbf{J} \dot{\mathbf{\mathbf{w}}} _i , \mathbf{w} ) _\Omega = \frac{ \partial ^2 \mathcal{H} }{ \partial \mathbf{z} ^2 } ( \mathbf{w} _i , \mathbf{w} ) , \quad \forall \mathbf{w} \in \mathring{\mathbf{\Lambda}} (\Omega) ,
\end{equation*}
which is equivalent to the variational equation
\eqref{eq:variational_JD}. It follows that if
$ \mathbf{w} _1 , \mathbf{w} _2 $ are a pair of compactly supported
first variations, then we have the \emph{global symplectic
  conservation law}
\begin{equation*}
  \frac{\mathrm{d}}{\mathrm{d}t} ( \mathbf{J} \mathbf{w} _1 , \mathbf{w} _2 ) _\Omega = 0 .
\end{equation*}
Note that this also follows immediately from the integral form of the
multisymplectic conservation law \eqref{eq:mscl_integral}, e.g., by
taking
$ K = \operatorname{supp} \mathbf{w} _1 \cap \operatorname{supp}
\mathbf{w} _2$. Hence, multisymplecticity implies symplecticity.

\begin{example}
  \label{ex:global_hamiltonian_wave}
  
  Recall from \cref{ex:hodge_wave} that the semilinear Hodge wave
  equation has Hamiltonian
  $ H (t, x, \mathbf{z} ) = - \frac{1}{2} \lvert \sigma \rvert ^2 +
  \bigl( \frac{1}{2} \lvert p \rvert ^2 + F ( t, x, u ) \bigr) -
  \frac{1}{2} \lvert \rho \rvert ^2 $. For the global Hamiltonian
  approach, we first write $ \mathbf{z} =
  \begin{bmatrix}
    \sigma \oplus u \oplus \rho \\
    \theta \oplus p \oplus \xi 
  \end{bmatrix} $ and subsequently show that we can set $\theta$ and $\xi$ equal to zero. We then have
  \begin{multline*}
    \mathcal{H} ( t, \mathbf{z} ) = - \frac{1}{2} \lVert \sigma \rVert ^2 _\Omega + \biggl( \frac{1}{2} \lVert p \rVert _\Omega ^2 + \int _\Omega F ( t, x, u ) \,\mathrm{vol} \biggr) - \frac{1}{2} \lVert \rho \rVert _\Omega ^2 \\
    - \frac{1}{2} \Bigl[ ( \delta u , \sigma ) _\Omega + ( \mathrm{d} \sigma + \delta \rho , u ) _\Omega + ( \mathrm{d} u, \rho ) _\Omega + ( \delta p , \theta ) _\Omega + ( \mathrm{d} \theta + \delta \xi , p ) _\Omega + ( \mathrm{d} p, \xi ) _\Omega \Bigr] .
  \end{multline*}
  Hence, Hamilton's equations are
  \begin{subequations}
    \label{eq:hodge_wave_global}
    \begin{alignat}{2}
      \dot{ \sigma } &= \frac{ \partial \mathcal{H} }{ \partial \theta } &&= - \delta p , \label{eq:hodge_wave_global_sigma} \\
      \dot{u} &= \frac{ \partial \mathcal{H} }{ \partial p } &&= p - \mathrm{d} \theta - \delta \xi,\\
      \dot{ \rho } &= \frac{ \partial \mathcal{H} }{ \partial \xi } &&= - \mathrm{d} p , \label{eq:hodge_wave_global_rho} \\
      -\dot{\theta} &= \frac{ \partial \mathcal{H} }{ \partial \sigma } &&= - \sigma - \delta u , \label{eq:hodge_wave_global_theta}\\
      -\dot{p} &= \frac{ \partial \mathcal{H} }{ \partial u } &&= \frac{ \partial F }{ \partial u } - \mathrm{d} \sigma - \delta \rho ,\\
      - \dot{ \xi } &= \frac{ \partial \mathcal{H} }{ \partial \rho } &&= - \rho - \mathrm{d} u \label{eq:hodge_wave_global_xi} .
    \end{alignat}
  \end{subequations}
  This system is immediately seen to be equivalent to
  \eqref{eq:hodge_wave} when $ \theta $ and $ \xi $ vanish. By the
  same argument as in \cref{ex:hodge_wave}, if $\theta$ and $\xi$
  vanish at the initial time with $ \sigma = - \delta u $ and
  $ \rho = - \delta u $, then these conditions remain true for all
  time, and we may eliminate \eqref{eq:hodge_wave_global_theta} and
  \eqref{eq:hodge_wave_global_xi} to obtain a first-order system in
  $\sigma$, $u$, $\rho$, and $p$ alone that remains in the invariant
  subspace $\mathbf{S}$ with $ \mathbf{z} =
  \begin{bmatrix}
    - \delta u \oplus u \oplus - \mathrm{d} u \\
    p 
  \end{bmatrix} $.

  An alternative but equivalent choice of global Hamiltonian is
  \begin{equation*}
    \widetilde{ \mathcal{H} } (t, \mathbf{z} ) = - \frac{1}{2} \lVert \sigma \rVert ^2 _\Omega + \biggl( \frac{1}{2} \lVert p \rVert _\Omega ^2 + \int _\Omega F ( t, x, u ) \,\mathrm{vol} \biggr) - \frac{1}{2} \lVert \rho \rVert _\Omega ^2 
    - ( \delta u , \sigma ) _\Omega - ( \mathrm{d} u, \rho ) _\Omega - ( \delta p , \theta ) _\Omega - ( \mathrm{d} p, \xi ) _\Omega .
  \end{equation*}
  Integration by parts shows that this agrees with $\mathcal{H}$ up to
  boundary terms, so it has the same functional derivatives along
  compactly supported test functions, and hence yields the same
  dynamics \eqref{eq:hodge_wave_global}. Furthermore, restricting
  $ \widetilde{ \mathcal{H} } $ to $ \mathbf{S} $, which is
  parametrized by $ ( u, p ) \in \mathbf{\Lambda} ^k (\Omega) $,
  yields
  \begin{equation*}
    \widetilde{ \mathcal{H} } _{ \mathbf{S} } ( t, u, p ) = \frac{1}{2} \lVert p \rVert _\Omega ^2 + \frac{1}{2} \lVert \mathrm{D} u \rVert _\Omega ^2 + \int _\Omega F ( t, x, u ) \,\mathrm{vol}.
  \end{equation*}
  This can be interpreted as a global Hamiltonian on
  $ \mathbf{\Lambda} ^k (\Omega) $ whose dynamics are
  \begin{alignat*}{2}
    \dot{ u } &= \frac{ \partial \widetilde{ \mathcal{H} } _{ \mathbf{S} }  }{ \partial p } &&= p ,\\
    - \dot{p} &= \frac{ \partial \widetilde{ \mathcal{H} } _{ \mathbf{S} }  }{ \partial u } &&= \mathrm{D} ^2 u + \frac{ \partial F }{ \partial u } ,
  \end{alignat*}
  which is again equivalent to the semilinear Hodge wave equation.
  This generalizes the usual global Hamiltonian formulation of the
  scalar semilinear wave equation, cf.~\citet[\S3.2]{MaRa1999}, where
  $ \widetilde{ \mathcal{H} } _{ \mathbf{S} } $ is interpreted
  as energy. Note that these equations are first-order in time and
  second-order in space, whereas the previous formulation including
  $\sigma$ and $\rho$ is first-order in both time and space.
\end{example}

\subsection{Global Hamiltonian structure of multisymplectic methods}
\label{sec:discrete_hamiltonian}

We now express the multisymplectic semidiscretization methods of
\cref{sec:semidiscrete} as global Hamiltonian systems corresponding to
a discrete Hamiltonian $ \mathcal{H} _h $. For simplicity, we assume
that we have sufficient regularity to write
$ \mathbf{f} ( t, \mathbf{z} _h ) = \partial H / \partial \mathbf{z}
_h $.

To put \eqref{eq:weak} into global Hamiltonian form, we must choose
boundary conditions and eliminate the trace variables and constraints
\eqref{eq:weak_wnor}--\eqref{eq:weak_wtan} so that \eqref{eq:weak_w}
reduces to Hamiltonian dynamics on some symplectic vector space
$ \widetilde{ \mathbf{W} } _h = \widetilde{ W } _h \otimes \mathbb{R}
^2 $, whose symplectic form we denote by
$ \widetilde{ \mathbf{J} } \coloneqq \widetilde{ W } _h \otimes J
$. (In all of our examples, $ \widetilde{ W } _h \subset W _h $ is a
subspace, so $ \widetilde{ \mathbf{J} } $ is simply the restriction of
$ \mathbf{J} $ to the symplectic subspace
$ \widetilde{ \mathbf{W} } _h \subset \mathbf{W} _h $.) The following
assumption formalizes conditions under which we can perform such a
reduction for solutions satisfying homogeneous Dirichlet boundary
conditions
$ \widehat{ \mathbf{z} } _h ^{\mathrm{tan}} \in \ringhat{ \mathbf{V} }
_h ^{\mathrm{tan}} $.

\begin{assumption}
  \label{assumption:theta}
  Suppose
  $ \mathbf{\Theta} \colon I \times \widetilde{ \mathbf{W} } _h
  \rightarrow \mathbf{W} _h \times \widehat{ \mathbf{W} } _h
  ^{\mathrm{nor}} \times \ringhat{ \mathbf{V} } _h ^{\mathrm{tan}} $,
  $ (t, \widetilde{ \mathbf{z} } _h ) \mapsto ( \mathbf{z} _h ,
  \widehat{ \mathbf{z} } _h ) $, satisfies the following conditions
  for all $ t \in I $:
  \begin{assumptionenumerate}
  \item The map $ \widetilde{ \mathbf{z} } _h \mapsto \mathbf{z} _h $
    is constant in $t$ and symplectic, i.e.,
    $ ( \mathbf{J} \mathbf{w} _1 , \mathbf{w} _2 ) _{ \mathcal{T} _h }
    = ( \widetilde{ \mathbf{J} } \widetilde{ \mathbf{w} } _1 ,
    \widetilde{ \mathbf{w} } _2 ) _{ \mathcal{T} _h } $ with
    $ \mathbf{w} _i = \frac{ \partial \mathbf{z} _h }{ \partial
      \widetilde{ \mathbf{z} } _h } \widetilde{ \mathbf{w} } _i $ for
    all
    $ \widetilde{ \mathbf{w} } _1 , \widetilde{ \mathbf{w} } _2 \in
    \widetilde{ \mathbf{W} } _h $.
    \label{assumption:theta_symplectic}

  \item If
    $ \dot{ \widetilde{ \mathbf{z} } } _h \in \widetilde{ \mathbf{W} }
    _h $ is such that \eqref{eq:weak_w} holds with
    $ \dot{ \mathbf{z} } _h = \frac{ \partial \mathbf{z} _h }{
      \partial \widetilde{ \mathbf{z} } _h } \dot{ \widetilde{
        \mathbf{z} } } _h $ and
    $ \mathbf{w} _h = \frac{ \partial \mathbf{z} _h }{ \partial
      \widetilde{ \mathbf{z} } _h } \widetilde{ \mathbf{w} } _h $ for
    all
    $ \widetilde{ \mathbf{w} } _h \in \widetilde{ \mathbf{W} } _h $,
    then \eqref{eq:weak_w} holds for all
    $ \mathbf{w} _h \in \mathbf{W} _h $.
    \label{assumption:theta_wtilde}

  \item Equations \eqref{eq:weak_wnor}--\eqref{eq:weak_wtan} hold for
    all
    $ \widetilde{ \mathbf{z} } _h \in \widetilde{ \mathbf{W} } _h
    $. \label{assumption:theta_constraints}
  \end{assumptionenumerate}
\end{assumption}

\begin{theorem}
  \label{thm:discrete_hamiltonian}
  Suppose $ \mathbf{\Phi} $ is multisymplectic and
  \cref{assumption:theta} holds. Then
  $ ( \mathbf{z} _h , \widehat{ \mathbf{z} } _h ) = \mathbf{\Theta} (
  t, \widetilde{ \mathbf{z} } _h ) $ satisfies \eqref{eq:weak} if and
  only if
  $ \widetilde{ \mathbf{z} } _h \colon I \rightarrow \widetilde{
    \mathbf{W} } _h $ satisfies Hamilton's equations,
  \begin{equation*}
    ( \widetilde{ \mathbf{J} } \dot{ \widetilde{ \mathbf{z} } } _h, \widetilde{ \mathbf{w}} _h ) _{ \mathcal{T} _h } = \frac{ \partial \mathcal{H} _h }{ \partial \widetilde{ \mathbf{z} } _h } \widetilde{ \mathbf{w} } _h , \quad \forall \widetilde{ \mathbf{w} } _h \in \widetilde{ \mathbf{W} } _h ,
  \end{equation*}
  where the discrete Hamiltonian
  $ \mathcal{H} _h \colon I \times \widetilde{ \mathbf{W} } _h
  \rightarrow \mathbb{R} $ is given by
  \begin{equation*}
    \mathcal{H} _h ( t, \widetilde{ \mathbf{z} } _h ) \coloneqq \int _\Omega H ( t, x, \mathbf{z} _h ) \,\mathrm{vol} - \frac{1}{2} \Bigl( ( \mathbf{z} _h , \mathbf{D} \mathbf{z} _h ) _{ \mathcal{T} _h } + [ \widehat{ \mathbf{z} } _h , \mathbf{z} _h ] _{ \partial \mathcal{T} _h } \Bigr) .
  \end{equation*}
\end{theorem}

\begin{proof}
  Let $ \widetilde{ \mathbf{w} } _h \in \widetilde{ \mathbf{W} } _h $
  be arbitrary. First, by the chain rule and
  \cref{assumption:theta_symplectic}, we have
  \begin{equation*}
    ( \widetilde{ \mathbf{J} } \dot{ \widetilde{ \mathbf{z} } } _h , \widetilde{ \mathbf{w} } _h ) _{ \mathcal{T} _h } = ( \mathbf{J} \dot{ \mathbf{z} } _h , \mathbf{w} _h ) _{ \mathcal{T} _h } .
  \end{equation*}
  Next, letting
  $ ( \mathbf{w} _h , \widehat{ \mathbf{w} } _h ) \coloneqq
  \frac{\partial \mathbf{\Theta} }{ \partial \widetilde{ \mathbf{z} }
    _h } \widetilde{ \mathbf{w} } _h $, we calculate
  \begin{align*}
    \frac{ \partial \mathcal{H} _h }{ \partial \widetilde{ \mathbf{z} } _h } \widetilde{ \mathbf{w} } _h &= \biggl( \frac{ \partial H }{ \partial \mathbf{z} _h } , \mathbf{w} _h \biggr) _{ \mathcal{T} _h } - \frac{1}{2} \Bigl( ( \mathbf{D} \mathbf{z} _h , \mathbf{w} _h ) _{ \mathcal{T} _h } + ( \mathbf{z} _h , \mathbf{D} \mathbf{w} _h ) _{ \mathcal{T} _h } + [ \widehat{ \mathbf{z} } _h , \mathbf{w} _h ] _{ \partial \mathcal{T} _h } - [ \mathbf{z} _h , \widehat{ \mathbf{w} } _h ] _{ \partial \mathcal{T} _h } \Bigr) \\
    &= \biggl( \frac{ \partial H }{ \partial \mathbf{z} _h } , \mathbf{w} _h \biggr) _{ \mathcal{T} _h } - ( \mathbf{z} _h , \mathbf{D} \mathbf{w} _h ) _{ \mathcal{T} _h } - [ \widehat{ \mathbf{z} } _h , \mathbf{w} _h ] _{ \partial \mathcal{T} _h } + \frac{1}{2} \Bigl( [ \widehat{ \mathbf{z} } _h , \mathbf{w} _h ] _{ \partial \mathcal{T} _h } + [ \mathbf{z} _h , \widehat{ \mathbf{w} } _h ] _{ \partial \mathcal{T} _h } - [ \mathbf{z} _h , \mathbf{w} _h ] _{ \partial \mathcal{T} _h } \Bigr) \\
    &= \biggl( \frac{ \partial H }{ \partial \mathbf{z} _h } , \mathbf{w} _h \biggr) _{ \mathcal{T} _h } - ( \mathbf{z} _h , \mathbf{D} \mathbf{w} _h ) _{ \mathcal{T} _h } - [ \widehat{ \mathbf{z} } _h , \mathbf{w} _h ] _{ \partial \mathcal{T} _h } + \frac{1}{2} \Bigl( [ \widehat{ \mathbf{z} } _h , \widehat{ \mathbf{w} } _h ] _{ \partial \mathcal{T} _h } - [ \widehat{ \mathbf{z} } _h - \mathbf{z} _h , \widehat{ \mathbf{w} } _h - \mathbf{w} _h ] _{ \partial \mathcal{T} _h } \Bigr) ,
  \end{align*}
  where the second line uses the integration-by-parts-identity
  \eqref{eq:bracket_D}. It suffices to show that the last group of
  terms on the right-hand side vanishes, since then equality of the
  right-hand sides is equivalent to \eqref{eq:weak_w} by
  \cref{assumption:theta_wtilde}, and
  \eqref{eq:weak_wnor}--\eqref{eq:weak_wtan} hold by
  \cref{assumption:theta_constraints}.

  Differentiating \cref{assumption:theta_constraints} implies that
  $ ( \mathbf{w} _h , \widehat{ \mathbf{w} } _h ) $ also satisfies
  \eqref{eq:weak_wnor}--\eqref{eq:weak_wtan}, since these equations
  are linear. Thus, multisymplecticity of $ \mathbf{\Phi} $ implies
  that
  $ [ \widehat{ \mathbf{z} } _h - \mathbf{z} _h , \widehat{ \mathbf{w}
  } _h - \mathbf{w} _h ] _{ \partial \mathcal{T} _h } = 0 $, since
  $ ( \mathbf{z} _h , \widehat{ \mathbf{z} } _h ) $ and
  $ ( \mathbf{w} _h , \widehat{ \mathbf{w} } _h ) $ both satisfy
  \eqref{eq:weak_wnor}, which is identical to \eqref{eq:weakvar_wnor}.
  All that remains is
  \begin{equation*}
    [ \widehat{ \mathbf{z} } _h , \widehat{ \mathbf{w} } _h ] _{ \partial \mathcal{T} _h } = \langle \widehat{ \mathbf{z} } _h ^{\mathrm{tan}} , \widehat{ \mathbf{w} } _h ^{\mathrm{nor}} \rangle _{ \partial \mathcal{T} _h } - \langle \widehat{ \mathbf{w} } _h ^{\mathrm{tan}}, \widehat{ \mathbf{z} } _h ^{\mathrm{nor}} \rangle _{ \partial \mathcal{T} _h } ,
  \end{equation*}
  and both terms vanish by \eqref{eq:weak_wtan} since
  $ \widehat{ \mathbf{z} } _h ^{\mathrm{tan}} , \widehat{ \mathbf{w} }
  _h ^{\mathrm{tan}} \in \ringhat{\mathbf{V} } _h ^{\mathrm{tan}}
  $. Therefore, we have shown that
  \begin{equation*}
    \frac{1}{2} \Bigl( [ \widehat{ \mathbf{z} } _h , \widehat{ \mathbf{w} } _h ] _{ \partial \mathcal{T} _h } - [ \widehat{ \mathbf{z} } _h - \mathbf{z} _h , \widehat{ \mathbf{w} } _h - \mathbf{w} _h ] _{ \partial \mathcal{T} _h } \Bigr) = 0 ,
  \end{equation*}
  as claimed, which completes the proof.
\end{proof}

\begin{remark}
  \label{rmk:hamiltonian_bc}
  This proof can be adapted to other boundary conditions satisfying
  $ [ \widehat{ \mathbf{z} } _h , \widehat{ \mathbf{w} } _h ] _{
    \partial \mathcal{T} _h } = 0 $. This includes homogeneous Neumann
  boundary conditions, which are imposed naturally by requiring that
  \eqref{eq:weak_wtan} hold for test functions in
  $ \widehat{ \mathbf{V} } _h ^{\mathrm{tan}} $, not merely
  $ \ringhat{ \mathbf{V} } _h ^{\mathrm{tan}} $; in that case, we
  would take
  $ \mathbf{\Theta} \colon I \times \widetilde{ \mathbf{W} } _h
  \rightarrow \mathbf{W} _h \times \widehat{ \mathbf{W} } _h
  ^{\mathrm{nor}} \times \widehat{ \mathbf{V} } _h ^{\mathrm{tan}} $.
\end{remark}

As in \cref{sec:global_hamiltonian_smooth}, this Hamiltonian structure
immediately implies a symplectic conservation law. Indeed, under the
hypotheses of \cref{thm:discrete_hamiltonian}, variations
$ \widetilde{ \mathbf{w} } _i \colon I \rightarrow \widetilde{
  \mathbf{W} } _h $ with
$ ( \mathbf{w} _i , \widehat{ \mathbf{w} } _i ) = \frac{ \partial \mathbf{\Theta} }{
  \partial \widetilde{ \mathbf{z} } _h } \widetilde{ \mathbf{w}}  _i $ satisfy
\begin{equation*}
  ( \widetilde{ \mathbf{J} } \dot{\widetilde{ \mathbf{w} }} _i , \widetilde{ \mathbf{w} } _h ) _{ \mathcal{T} _h } = \biggl( \frac{ \partial ^2 \mathcal{H} _h }{ \partial \widetilde{ \mathbf{z} } _h ^2 } \widetilde{ \mathbf{w}}  _i , \widetilde{ \mathbf{w} } _h \biggr) _{ \mathcal{T} _h } , \quad \forall \widetilde{ \mathbf{w} } _h \in \widetilde{ \mathbf{W} } _h .
\end{equation*}
Hence, the antisymmetry of $ \widetilde{ \mathbf{J} } $ and symmetry
of the Hessian yield the symplectic conservation law,
\begin{equation*}
  \frac{\mathrm{d}}{\mathrm{d}t} ( \mathbf{J} \mathbf{w} _1, \mathbf{w} _2 ) _{ \mathcal{T} _h } = \frac{\mathrm{d}}{\mathrm{d}t} ( \widetilde{ \mathbf{J} } \widetilde{ \mathbf{w} } _1, \widetilde{ \mathbf{w} } _2 ) _{ \mathcal{T} _h } = 0 ,
\end{equation*}
for all such pairs of variations, where the first equality is by
\cref{assumption:theta_symplectic}.  However, a stronger conclusion
follows directly from the multisymplectic conservation law,
\emph{without} these extra hypotheses. Summing \eqref{eq:mscl_weak}
over $ K \in \mathcal{T} _h $ gives
\begin{equation*}
  \frac{\mathrm{d}}{\mathrm{d}t} ( \mathbf{J} \mathbf{w} _1, \mathbf{w} _2 ) _{ \mathcal{T} _h } + [ \widehat{ \mathbf{w} } _1 , \widehat{ \mathbf{w} } _2 ] _{ \partial \mathcal{T} _h } = 0 ,
\end{equation*}
for \emph{arbitrary} pairs of variations. In particular, if
$ \widehat{ \mathbf{w} } _1 ^{\mathrm{tan}} , \widehat{ \mathbf{w} }
_2 ^{\mathrm{tan}} \in \ringhat{\mathbf{V}} _h ^{\mathrm{tan}} $, then
\eqref{eq:weakvar_wtan} implies that
$ [ \widehat{ \mathbf{w} } _1 , \widehat{ \mathbf{w} } _2 ] _{
  \partial \mathcal{T} _h } = 0 $, which recovers the symplectic
conservation law. Along similar lines as \cref{rmk:hamiltonian_bc},
this argument extends to other boundary conditions such that
$ [ \widehat{ \mathbf{w} } _1 , \widehat{ \mathbf{w} } _2 ] _{
  \partial \mathcal{T} _h } = 0 $.

\begin{example}
  \label{ex:afw_discrete_hamiltonian}
  For the AFW-H method introduced in \cref{ex:AFW}, with homogeneous
  Dirichlet boundary conditions, we take
  $ \widetilde{ \mathbf{W} } _h \coloneqq \mathring{ \mathbf{V} } _h $
  and define the map $ \mathbf{\Theta} $ as follows. First, we take
  $ \mathbf{z} _h \coloneqq \widetilde{ \mathbf{z} } _h $ and
  $ \widehat{ \mathbf{z} } _h ^{\mathrm{tan}} \coloneqq \mathbf{z} _h
  ^{\mathrm{tan}} \in \ringhat{ \mathbf{V} } _h ^{\mathrm{tan}}
  $. Then, we find
  $ \dot{\mathbf{z}} _h \in \mathring{ \mathbf{V} } _h $ satisfying
  \eqref{eq:afw} and solve for
  $ \widehat{ \mathbf{z} } _h ^{\mathrm{nor}} \in \widehat{ \mathbf{W}
  } _h ^{\mathrm{nor}} $ satisfying \eqref{eq:afw-h_w}. (This
  procedure for recovering the traces from $ \mathbf{z} _h $ is a
  minor modification of the converse direction in the proof of
  \citep[Theorem~4.1]{StZa2025}.) This satisfies
  \cref{assumption:theta} by construction:
  \cref{assumption:theta_symplectic} holds since
  $ \mathring{ \mathbf{V} } _h \subset \mathbf{W} _h $ is a symplectic
  subspace (i.e., $ \mathbf{J} $ is nondegenerate on
  $ \mathring{ \mathbf{V} } _h $), and
  \crefrange{assumption:theta_wtilde}{assumption:theta_constraints} hold by
  the fact that the AFW-H method \eqref{eq:afw-h} is a hybridization
  of the AFW method \eqref{eq:afw}. Now, subtracting
  \eqref{eq:afw-h_w} and \eqref{eq:afw} with
  $ \mathbf{w} _h = \mathbf{z} _h \in \mathring{ \mathbf{V} } _h $
  gives
  \begin{equation*}
    \frac{1}{2} \Bigl( ( \mathbf{z} _h , \mathbf{D} \mathbf{z} _h ) _{ \mathcal{T} _h }  + [ \widehat{ \mathbf{z} } _h , \mathbf{z} _h ] _{ \partial \mathcal{T} _h } \Bigr) = ( \mathbf{d} \mathbf{z} _h , \mathbf{z} _h ) _\Omega ,
  \end{equation*}
  so applying \cref{thm:discrete_hamiltonian}, we conclude that the
  discrete Hamiltonian for AFW(-H) is
  \begin{equation*}
    \mathcal{H} _h ( t, \mathbf{z} _h ) = \int _\Omega H ( t, x, \mathbf{z} _h ) \,\mathrm{vol} - ( \mathbf{d} \mathbf{z} _h , \mathbf{z} _h ) _\Omega .
  \end{equation*}
\end{example}

\begin{example}
  \label{ex:ldg-h_discrete_hamiltonian}
  For the LDG-H method introduced in \cref{ex:LDG}, again with
  homogeneous Dirichlet boundary conditions, we take
  $ \widetilde{ \mathbf{W} } _h \coloneqq \mathbf{W} _h $. Assuming
  that the symmetric bilinear form
  $ \langle \boldsymbol{\alpha} \cdot , \cdot \rangle _{ \partial
    \mathcal{T} _h } $ is nondegenerate on
  $ \ringhat{ \mathbf{V} } _h ^{\mathrm{tan}} $, we define
  $ \mathbf{\Theta} $ by taking
  $ \mathbf{z} _h \coloneqq \widetilde{ \mathbf{z} } _h $, solving
  \eqref{eq:ldg-h_wtan} for
  $ \widehat{ \mathbf{z} } _h ^{\mathrm{tan}} \in \ringhat{ \mathbf{V}
  } _h ^{\mathrm{tan}} $, and letting
  $ \widehat{ \mathbf{z} } _h ^{\mathrm{nor}} \coloneqq \mathbf{z} _h
  ^{\mathrm{nor}} - \boldsymbol{\alpha} ( \widehat{ \mathbf{z} } _h
  ^{\mathrm{tan}} - \mathbf{z} _h ^{\mathrm{tan}} ) $. This clearly
  satisfies \cref{assumption:theta}. It follows that
  \begin{align*}
    \frac{1}{2} \Bigl( ( \mathbf{z} _h , \mathbf{D} \mathbf{z} _h ) _{ \mathcal{T} _h } + [ \widehat{ \mathbf{z} } _h , \mathbf{z} _h ] _{ \partial \mathcal{T} _h }
    \Bigr)
    &= ( \mathbf{z} _h , \boldsymbol{\delta} \mathbf{z} _h ) _{ \mathcal{T} _h } + \frac{1}{2} \langle \widehat{ \mathbf{z} } _h ^{\mathrm{tan}} , \mathbf{z} _h ^{\mathrm{nor}} \rangle _{ \partial \mathcal{T} _h }  - \frac{1}{2} \langle \widehat{ \mathbf{z} } _h ^{\mathrm{nor}} - \mathbf{z} _h ^{\mathrm{nor}} , \mathbf{z} _h ^{\mathrm{tan}} \rangle _{ \partial \mathcal{T} _h } \\
    &= ( \mathbf{z} _h , \boldsymbol{\delta} \mathbf{z} _h ) _{ \mathcal{T} _h } + \frac{1}{2} \bigl\langle \widehat{ \mathbf{z} } _h ^{\mathrm{tan}} , \boldsymbol{\alpha} ( \widehat{ \mathbf{z} } _h ^{\mathrm{tan}} - \mathbf{z} _h ^{\mathrm{tan}} ) \bigr\rangle _{ \partial \mathcal{T} _h } + \frac{1}{2} \bigl\langle \boldsymbol{\alpha} ( \widehat{ \mathbf{z} } _h ^{\mathrm{tan}} - \mathbf{z} _h ^{\mathrm{tan}} ) , \mathbf{z} _h ^{\mathrm{tan}} \bigr\rangle _{ \partial \mathcal{T} _h } \\
    &= ( \mathbf{z} _h , \boldsymbol{\delta} \mathbf{z} _h ) _{ \mathcal{T} _h } + \frac{1}{2} \langle \boldsymbol{\alpha} \widehat{ \mathbf{z} } _h ^{\mathrm{tan}} , \widehat{ \mathbf{z} } _h ^{\mathrm{tan}} \rangle _{ \partial \mathcal{T} _h } - \frac{1}{2} \langle \boldsymbol{\alpha} \mathbf{z} _h ^{\mathrm{tan}} , \mathbf{z} _h ^{\mathrm{tan}} \rangle _{ \partial \mathcal{T} _h } .
  \end{align*}
  Therefore, applying \cref{thm:discrete_hamiltonian}, we conclude
  that the discrete Hamiltonian for LDG-H is
  \begin{equation*}
    \mathcal{H} _h ( t , \mathbf{z} _h ) = \int _\Omega H ( t, x, \mathbf{z} _h ) \,\mathrm{vol} - ( \mathbf{z} _h , \boldsymbol{\delta} \mathbf{z} _h ) _{ \mathcal{T} _h } - \frac{1}{2} \Bigl( \langle \boldsymbol{\alpha} \widehat{ \mathbf{z} } _h ^{\mathrm{tan}} , \widehat{ \mathbf{z} } _h ^{\mathrm{tan}} \rangle _{ \partial \mathcal{T} _h } - \langle \boldsymbol{\alpha} \mathbf{z} _h ^{\mathrm{tan}} , \mathbf{z} _h ^{\mathrm{tan}} \rangle _{ \partial \mathcal{T} _h } \Bigr) .
  \end{equation*}
\end{example}

\subsection{Global structure of methods for the semilinear Hodge wave
  equation}
\label{sec:hodge_wave_hamiltonian}

We now consider the global Hamiltonian structure of the
methods in \cref{sec:hodge_wave_methods}. Recall that the Hamiltonian
for the the semilinear Hodge wave equation is
$ H (t, x, \mathbf{z} ) = - \frac{1}{2} \lvert \sigma \rvert ^2 +
\bigl( \frac{1}{2} \lvert p \rvert ^2 + F ( t, x, u ) \bigr) -
\frac{1}{2} \lvert \rho \rvert ^2 $. Following
\cref{sec:discrete_hamiltonian}, we assume sufficient regularity to
take $ f ( t, u _h ) = \partial F / \partial u _h $.

As above, we impose homogeneous Dirichlet boundary conditions
$ \widehat{ \mathbf{z} } _h ^{\mathrm{tan}} \in \ringhat{ \mathbf{V} }
_h ^{\mathrm{tan}} $, but the arguments may be adapted to homogeneous
Neumann or other boundary conditions as described in
\cref{rmk:hamiltonian_bc}.

\subsubsection{The AFW-H method} Take
$ \widetilde{ \mathbf{W} } _h \coloneqq \mathring{ \mathbf{V} } _h ^k
$, and define $ \mathbf{\Theta} $ as follows:
\begin{itemize}
\item Take $ ( u _h , p _h ) \coloneqq \widetilde{ \mathbf{z} } _h $.

\item Solve for $ \sigma _h \in \mathring{ V } _h ^{ k -1 } $
  satisfying \eqref{eq:afw_wave_tau} and
  $ \rho _h \in \mathring{ V } _h ^{ k + 1 } $ satisfying
  \eqref{eq:afw_wave_eta}.

\item Take the tangential traces
  $ \widehat{ \sigma } _h ^{\mathrm{tan}} \coloneqq \sigma _h
  ^{\mathrm{tan}} $,
  $ \widehat{ u } _h ^{\mathrm{tan}} \coloneqq u _h ^{\mathrm{tan}} $,
  and
  $ \widehat{ p } _h ^{\mathrm{tan}} \coloneqq p _h ^{\mathrm{tan}} $.

\item Solve for $ \dot{ \sigma } _h \in \mathring{ V } _h ^{ k -1 } $
  satisfying \eqref{eq:afw2_wave_tau} and
  $ \dot{ p } _h \in \mathring{ V } _h ^k $ satisfying
  \eqref{eq:afw2_wave_v}.

\item Solve for the normal traces
  $ \widehat{ u } _h ^{\mathrm{nor}} \in \widehat{ W } _h ^{ k -1 ,
    \mathrm{nor} } $ satisfying \eqref{eq:afw-h_wave_tau},
  $ \widehat{ p } _h ^{\mathrm{nor}} \in \widehat{ W } _h ^{ k -1 ,
    \mathrm{nor} } $ satisfying \eqref{eq:afw2-h_wave_tau}, and
  $ \widehat{ \rho } _h ^{\mathrm{nor}} \in \widehat{ W } _h ^{k,
    \mathrm{nor}} $ satisfying \eqref{eq:afw2-h_wave_v}.

\item Take $ \widehat{ \sigma } _h ^{\mathrm{nor}} $ and
  $ \widehat{ \rho } _h ^{\mathrm{tan}} $ as in
  \cref{prop:afw-h_sigmanor_rhotan} (but these need not be computed,
  per \cref{rmk:sigmanor_rhotan}).
\end{itemize}
This satisfies the hypotheses of \cref{thm:discrete_hamiltonian}, and
hence we obtain the following corollary.

\begin{corollary}
  \label{cor:afw-h_discrete_hamiltonian}
  For the semilinear Hodge wave equation, the AFW(-H) method is
  equivalent to Hamilton's equations for
  $ ( u _h , p _h ) \in \mathring{ \mathbf{V} } _h ^k $, where the
  discrete Hamiltonian is
  \begin{equation*}
    \mathcal{H} _h ( t, u _h , p _h ) = \frac{1}{2} \Bigl( \lVert \sigma _h \rVert _\Omega ^2 + \lVert p _h \rVert _\Omega ^2 + \lVert \rho _h \rVert _\Omega ^2 \Bigr) + \int _\Omega F ( t, x, u _h ) \,\mathrm{vol}.
  \end{equation*}
\end{corollary}

\begin{proof}
  This follows directly from \cref{thm:discrete_hamiltonian}, where we calculate
  \begin{equation*}
    \int _\Omega H ( t, x, \mathbf{z} _h )\,\mathrm{vol} = - \frac{1}{2} \lVert \sigma _h \rVert _\Omega ^2 + \frac{1}{2} \lVert p _h \rVert _\Omega ^2 - \frac{1}{2} \lVert \rho _h \rVert _\Omega ^2 + \int _\Omega F ( t, x, u _h ) \,\mathrm{vol}
  \end{equation*}
  and subtract
  \begin{equation*}
    \frac{1}{2} \Bigl( ( \mathbf{z} _h , \mathbf{D} \mathbf{z} _h ) _{ \mathcal{T} _h }  + [ \widehat{ \mathbf{z} } _h , \mathbf{z} _h ] _{ \partial \mathcal{T} _h } \Bigr) = ( \mathbf{d} \mathbf{z} _h , \mathbf{z} _h ) _\Omega = ( \mathrm{d} \sigma _h , u _h ) _\Omega + ( \mathrm{d} u _h , \rho _h ) _\Omega = - \lVert \sigma _h \rVert _\Omega ^2 - \lVert \rho _h \rVert _\Omega ^2
  \end{equation*}
  to obtain $ \mathcal{H} _h $.  The last equality above holds by
  \eqref{eq:afw_wave_tau} with $ \tau _h = \sigma _h $ and
  \eqref{eq:afw_wave_eta} with $ \eta _h = \rho _h $.
\end{proof}

This substantially generalizes previous work on the global Hamiltonian
structure of conforming finite element methods, including results of
\citet[Theorem 4.2]{SaDuCo22} for Maxwell's equations and
\citet[Theorem 4.1]{SaVa24} for the semilinear wave equation.

\begin{remark}
  The map $ \mathbf{\Theta} $ parametrizes the discrete state space by
  $ ( u _h , p _h ) \in \mathring{ \mathbf{V} } _h ^k $, just as the
  invariant subspace $\mathbf{S}$ is parametrized by
  $ ( u , p ) \in \mathbf{\Lambda} ^k (\Omega) $ in the smooth
  case. The discrete Hamiltonian $ \mathcal{H} _h $ may thus be seen
  as a discrete version of the global Hamiltonian
  $ \widetilde{ \mathcal{H} } _{ \mathbf{S} } $ from
  \cref{ex:global_hamiltonian_wave}.
\end{remark}

\subsubsection{The multisymplectic LDG-H method}
As in \cref{sec:ldg-h_wave}, assume that $ \alpha ^{ k -1 } $ is
negative-definite and $ \alpha ^k $ is positive-definite. We then take
$ \widetilde{ \mathbf{W} } _h \coloneqq \mathbf{W} _h ^k $ and define
$ \mathbf{\Theta} $ as follows:
\begin{itemize}
\item Take $ ( u _h , p _h ) \coloneqq \widetilde{ \mathbf{z} } _h $.

\item Solve for $ \sigma _h \in W _h ^{ k -1 } $,
  $ \rho _h \in W _h ^{ k + 1 } $,
  $ \widehat{ \sigma } _h ^{\mathrm{tan}}\in \ringhat{V} _h ^{k-1,
    \mathrm{tan}} $, and
  $ \widehat{ u } _h ^{\mathrm{tan}} \in \ringhat{V} _h ^{k,
    \mathrm{tan}} $ satisfying
  \eqref{eq:ldg-h_wave_tau}--\eqref{eq:ldg-h_wave_vtan}. The solution
  exists uniquely by \cref{thm:ldg-h_wave_definite}.

\item Solve for $ \dot{ \sigma } _h \in W _h ^{ k -1 } $,
  $ \dot{ \rho } _h \in W _h ^{ k + 1 } $,
  $ \dot{ \widehat{ \sigma }} _h ^{\mathrm{tan}}\in \ringhat{V} _h ^{k-1,
    \mathrm{tan}} $, and
  $ \widehat{ p } _h ^{\mathrm{tan}} \in \ringhat{V} _h ^{k,
    \mathrm{tan}} $ satisfying
  \begin{alignat*}{2}
    ( \delta p _h , \tau _h ) _{ \mathcal{T} _h } + \bigl\langle \alpha ^{ k -1 } ( \dot{ \widehat{ \sigma } } _h ^{\mathrm{tan}} - \dot{ \sigma } _h ^{\mathrm{tan}} ) , \tau _h ^{\mathrm{tan}} \bigr\rangle _{ \partial \mathcal{T} _h } &= - ( \dot{ \sigma } _h , \tau _h ) _{ \mathcal{T} _h } , \quad &\forall \tau _h &\in W _h ^{ k -1 } , \\
    ( p _h , \delta \eta _h ) _{ \mathcal{T} _h } + \langle \widehat{ p } _h ^{\mathrm{tan}} , \eta _h ^{\mathrm{nor}} \rangle _{ \partial \mathcal{T} _h } &= - ( \dot{ \rho } _h , \eta _h )  _{ \mathcal{T} _h } , \quad &\forall \eta _h &\in W _h ^{ k + 1 } , \\
    \bigl\langle p _h ^{\mathrm{nor}} - \alpha ^{ k -1 } ( \dot{ \widehat{ \sigma }} _h ^{\mathrm{tan}} - \dot{ \sigma } _h ^{\mathrm{tan}} ) , \widehat{ \tau } _h ^{\mathrm{tan}} \bigr\rangle _{ \partial \mathcal{T} _h } &= 0 , \quad &\forall \widehat{ \tau } _h ^{\mathrm{tan}} &\in \ringhat{V} _h ^{k-1, \mathrm{tan}}, \\
    \bigl\langle \dot{ \rho } _h ^{\mathrm{nor}} - \alpha ^k ( \widehat{ p } _h ^{\mathrm{tan}} - p _h ^{\mathrm{tan}} ) , \widehat{ v } _h ^{\mathrm{tan}} \bigr\rangle _{ \partial \mathcal{T} _h } &= 0 , \quad &\forall \widehat{ v } _h ^{\mathrm{tan}} &\in \ringhat{V} _h ^{k, \mathrm{tan}} ,
  \end{alignat*}
  which is obtained by differentiating
  \eqref{eq:ldg-h_wave_tau}--\eqref{eq:ldg-h_wave_vtan} with respect
  to time and substituting $ \dot{ u } _h = p _h $ and
  $ \dot{ \widehat{ u } } ^{\mathrm{tan}} _h = \widehat{ p }
  ^{\mathrm{tan}} _h $. The solution exists uniquely by the same
  argument as \cref{thm:ldg-h_wave_definite}.

\item Take the normal traces
  \begin{align*}
    \widehat{ u } _h ^{\mathrm{nor}}
    &= u _h ^{\mathrm{nor}} - \alpha
      ^{ k -1 } ( \widehat{ \sigma } _h ^{\mathrm{tan}} - \sigma _h
      ^{\mathrm{tan}} ),\\
    \widehat{ p } _h ^{\mathrm{nor}}
    &= p _h ^{\mathrm{nor}} - \alpha
      ^{ k -1 } ( \dot{ \widehat{ \sigma } } _h ^{\mathrm{tan}} - \dot{
      \sigma } _h ^{\mathrm{tan}} ) ,\\    
    \widehat{ \rho } _h ^{\mathrm{nor}}
    &= \rho _h ^{\mathrm{nor}} -
      \alpha ^k ( \widehat{ u } _h ^{\mathrm{tan}} - u _h ^{\mathrm{tan}}
      ).
  \end{align*}
\item Take $ \widehat{ \sigma } _h ^{\mathrm{nor}} $ and
  $ \widehat{ \rho } _h ^{\mathrm{tan}} $ as in
  \cref{prop:ldg-h_sigmanor_rhotan} (but these need not be computed,
  per \cref{rmk:sigmanor_rhotan}).
\end{itemize} 
This satisfies the hypotheses of \cref{thm:discrete_hamiltonian}, so
we obtain the following.

\begin{theorem}
  \label{thm:ldg-h_discrete_hamiltonian}
  For the semilinear Hodge wave equation, if $ \alpha ^{ k -1 } $ is
  negative-definite and $ \alpha ^k $ is positive-definite, then the
  multisymplectic LDG-H method is equivalent to Hamilton's equations
  for $ ( u _h , p _h ) \in \mathbf{W} _h ^k $, where the discrete
  Hamiltonian is
  \begin{multline*}
    \mathcal{H} _h ( t, u _h , p _h ) = \frac{1}{2} \Bigl( \lVert \sigma _h \rVert _\Omega ^2 + \lVert p _h \rVert _\Omega ^2 + \lVert \rho _h \rVert _\Omega ^2 \Bigr) + \int _\Omega F ( t, x, u _h ) \,\mathrm{vol} \\
    - \frac{1}{2} \bigl\langle \alpha ^{ k -1 } ( \widehat{ \sigma } _h ^{\mathrm{tan}} - \sigma _h ^{\mathrm{tan}} ), \widehat{ \sigma } _h ^{\mathrm{tan}} - \sigma _h ^{\mathrm{tan}} \bigr\rangle _{ \partial \mathcal{T} _h } + \frac{1}{2} \bigl\langle \alpha ^k ( \widehat{ u } _h ^{\mathrm{tan}} - u _h ^{\mathrm{tan}} ) , \widehat{ u } _h ^{\mathrm{tan}} - u _h ^{\mathrm{tan}} \bigr\rangle _{ \partial \mathcal{T} _h } .
  \end{multline*}
\end{theorem}

\begin{proof}
  Similarly to the proof of \cref{cor:afw-h_discrete_hamiltonian}, we
  apply \cref{thm:discrete_hamiltonian} by calculating
  \begin{equation*}
    \int _\Omega H ( t, x, \mathbf{z} _h )\,\mathrm{vol} = - \frac{1}{2} \lVert \sigma _h \rVert _\Omega ^2 + \frac{1}{2} \lVert p _h \rVert _\Omega ^2 - \frac{1}{2} \lVert \rho _h \rVert _\Omega ^2 + \int _\Omega F ( t, x, u _h ) \,\mathrm{vol}
  \end{equation*}
  and subtracting
  \begin{align*}
    \frac{1}{2} \Bigl( ( \mathbf{z} _h , \mathbf{D} \mathbf{z} _h ) _{ \mathcal{T} _h }  + [ \widehat{ \mathbf{z} } _h , \mathbf{z} _h ] _{ \partial \mathcal{T} _h } \Bigr)
    &= ( \mathbf{z} _h , \boldsymbol{\delta} \mathbf{z} _h ) _{ \mathcal{T} _h } + \frac{1}{2} \langle \widehat{ \mathbf{z} } _h ^{\mathrm{tan}} , \mathbf{z} _h ^{\mathrm{nor}} \rangle _{ \partial \mathcal{T} _h }  - \frac{1}{2} \langle \widehat{ \mathbf{z} } _h ^{\mathrm{nor}} - \mathbf{z} _h ^{\mathrm{nor}} , \mathbf{z} _h ^{\mathrm{tan}} \rangle _{ \partial \mathcal{T} _h } \\
    &= ( \sigma _h , \delta u _h ) _{ \mathcal{T} _h } + \frac{1}{2} \langle \widehat{ \sigma } _h ^{\mathrm{tan}} , u _h ^{\mathrm{nor}} \rangle _{ \partial \mathcal{T} _h } - \frac{1}{2} \langle \widehat{ u } _h ^{\mathrm{nor}} - u _h ^{\mathrm{nor}} , \sigma _h ^{\mathrm{tan}} \rangle _{ \partial \mathcal{T} _h } \\
    &\qquad + ( u _h , \delta \rho _h ) _{ \mathcal{T} _h } + \frac{1}{2} \langle \widehat{ u } _h ^{\mathrm{tan}} , \rho _h ^{\mathrm{nor}} \rangle _{ \partial \mathcal{T} _h } - \frac{1}{2} \langle \widehat{ \rho } _h ^{\mathrm{nor}} - \rho _h ^{\mathrm{nor}} , u _h ^{\mathrm{tan}} \rangle _{ \partial \mathcal{T} _h } .
  \end{align*}
  We now evaluate the two lines of this last expression
  separately. First, by \eqref{eq:ldg-h_wave_tautan} with
  $ \widehat{ \tau } _h ^{\mathrm{tan}} = \widehat{ \sigma } _h
  ^{\mathrm{tan}} $ and the definition of
  $ \widehat{ u } _h ^{\mathrm{nor}} $, we have
  \begin{multline*}
    ( \sigma _h , \delta u _h ) _{ \mathcal{T} _h } + \frac{1}{2} \langle \widehat{ \sigma } _h ^{\mathrm{tan}} , u _h ^{\mathrm{nor}} \rangle _{ \partial \mathcal{T} _h } - \frac{1}{2} \langle \widehat{ u } _h ^{\mathrm{nor}} - u _h ^{\mathrm{nor}} , \sigma _h ^{\mathrm{tan}} \rangle _{ \partial \mathcal{T} _h }\\
    \begin{aligned}
      &= ( \sigma _h , \delta u _h ) _{ \mathcal{T} _h } + \frac{1}{2} \bigl\langle \widehat{ \sigma } _h ^{\mathrm{tan}} , \alpha ^{ k -1 } ( \widehat{ \sigma } _h ^{\mathrm{tan}} - \sigma _h ^{\mathrm{tan}} ) \bigr\rangle _{ \partial \mathcal{T} _h } + \frac{1}{2} \bigl\langle \alpha ^{ k -1 } ( \widehat{ \sigma } _h ^{\mathrm{tan}} - \sigma _h ^{\mathrm{tan}} ) , \sigma _h ^{\mathrm{tan}} \bigr\rangle _{ \partial \mathcal{T} _h } \\
      &= - \lVert \sigma _h \rVert ^2 _{ \mathcal{T} _h } + \frac{1}{2} \bigl\langle \widehat{ \sigma } _h ^{\mathrm{tan}} , \alpha ^{ k -1 } ( \widehat{ \sigma } _h ^{\mathrm{tan}} - \sigma _h ^{\mathrm{tan}} ) \bigr\rangle _{ \partial \mathcal{T} _h } - \frac{1}{2} \bigl\langle \alpha ^{ k -1 } ( \widehat{ \sigma } _h ^{\mathrm{tan}} - \sigma _h ^{\mathrm{tan}} ) , \sigma _h ^{\mathrm{tan}} \bigr\rangle _{ \partial \mathcal{T} _h } \\
      &= - \lVert \sigma _h \rVert ^2 _{ \mathcal{T} _h } + \frac{1}{2} \bigl\langle \alpha ^{ k -1 } ( \widehat{ \sigma } _h ^{\mathrm{tan}} - \sigma _h ^{\mathrm{tan}} ), \widehat{ \sigma } _h ^{\mathrm{tan}} - \sigma _h ^{\mathrm{tan}} \bigr\rangle _{ \partial \mathcal{T} _h } ,
    \end{aligned}
  \end{multline*}
  where the second equality uses \eqref{eq:ldg-h_wave_tau} with
  $ \tau _h = \sigma _h $, and the last line collects terms. Next, by
  \eqref{eq:ldg-h_wave_eta} with $ \eta _h = \rho _h $ and the
  definition of $ \widehat{ \rho } _h ^{\mathrm{nor}} $, we have
  \begin{multline*}
    ( u _h , \delta \rho _h ) _{ \mathcal{T} _h } + \frac{1}{2} \langle \widehat{ u } _h ^{\mathrm{tan}} , \rho _h ^{\mathrm{nor}} \rangle _{ \partial \mathcal{T} _h } - \frac{1}{2} \langle \widehat{ \rho } _h ^{\mathrm{nor}} - \rho _h ^{\mathrm{nor}} , u _h ^{\mathrm{tan}} \rangle _{ \partial \mathcal{T} _h }\\
    \begin{aligned}
      &= - \lVert \rho _h \rVert ^2 _{ \mathcal{T} _h } - \frac{1}{2} \langle \widehat{ u } _h ^{\mathrm{tan}} , \rho _h ^{\mathrm{nor}} \rangle _{ \partial \mathcal{T} _h } + \frac{1}{2} \bigl\langle \alpha ^k (\widehat{ u } _h ^{\mathrm{tan}} - u _h ^{\mathrm{tan}}) , u _h ^{\mathrm{tan}} \bigr\rangle _{ \partial \mathcal{T} _h } \\
      &= - \lVert \rho _h \rVert ^2 _{ \mathcal{T} _h } - \frac{1}{2} \bigl\langle \widehat{ u } _h ^{\mathrm{tan}} , \alpha ^k ( \widehat{ u } _h ^{\mathrm{tan}} - u _h ^{\mathrm{tan}} ) \bigr\rangle _{ \partial \mathcal{T} _h } + \frac{1}{2} \bigl\langle \alpha ^k (\widehat{ u } _h ^{\mathrm{tan}} - u _h ^{\mathrm{tan}}) , u _h ^{\mathrm{tan}} \bigr\rangle _{ \partial \mathcal{T} _h } \\
      &= - \lVert \rho _h \rVert ^2 _{ \mathcal{T} _h } - \frac{1}{2} \bigl\langle \alpha ^k (\widehat{ u } _h ^{\mathrm{tan}} - u _h ^{\mathrm{tan}}) , \widehat{ u } _h ^{\mathrm{tan}} - u _h ^{\mathrm{tan}} \bigr\rangle _{ \partial \mathcal{T} _h } ,
    \end{aligned}
  \end{multline*}
  where the second equality uses \eqref{eq:ldg-h_wave_vtan} with
  $ \widehat{ v } _h ^{\mathrm{tan}} = \widehat{ u } _h
  ^{\mathrm{tan}} $, and the last line collects terms. Altogether,
  \begin{multline*}
    \frac{1}{2} \Bigl( ( \mathbf{z} _h , \mathbf{D} \mathbf{z} _h ) _{ \mathcal{T} _h }  + [ \widehat{ \mathbf{z} } _h , \mathbf{z} _h ] _{ \partial \mathcal{T} _h } \Bigr) = - \lVert \sigma _h \rVert ^2 _{ \mathcal{T} _h } - \lVert \rho _h \rVert ^2 _{ \mathcal{T} _h } \\
    + \frac{1}{2} \bigl\langle \alpha ^{ k -1 } ( \widehat{ \sigma } _h ^{\mathrm{tan}} - \sigma _h ^{\mathrm{tan}} ), \widehat{ \sigma } _h ^{\mathrm{tan}} - \sigma _h ^{\mathrm{tan}} \bigr\rangle _{ \partial \mathcal{T} _h } - \frac{1}{2} \bigl\langle \alpha ^k (\widehat{ u } _h ^{\mathrm{tan}} - u _h ^{\mathrm{tan}}) , \widehat{ u } _h ^{\mathrm{tan}} - u _h ^{\mathrm{tan}} \bigr\rangle _{ \partial \mathcal{T} _h },
  \end{multline*}
  which yields the claimed expression for the discrete Hamiltonian.
\end{proof}

Again, this substantially generalizes the work of S\'anchez and
collaborators on the Hamiltonian structure of LDG-H methods for linear
\citep[Theorem 1]{SaCiNgPe17} and semilinear \citep[Theorem
4.1]{SaVa24} scalar wave equations, as well as Maxwell's equations
\citep[Theorem 4.2]{SaDuCo22}.

\section{Structure-preserving time integration of semidiscretized systems}
\label{sec:time_integrators}

In this section, we discuss the application of numerical integrators
to the finite-dimensional dynamical systems resulting from the
semidiscretization methods in \cref{sec:semidiscrete}. First,
following similar approach to \cref{sec:global_hamiltonian}, we
express \eqref{eq:weak} as a system of ODEs, rather than a system
containing both dynamical equations and (linear) algebraic
constraints. Next, we discuss the application of numerical integrators
to this system of ODEs, focusing particularly on symplectic
Runge--Kutta and partitioned Runge--Kutta methods. Finally, we use the
theory of \emph{functional equivariance} from \citet{McSt2024} to show
that, when a multisymplectic semidiscretization method is combined
with a symplectic integrator, we obtain a fully discrete (in both
space and time) multisymplectic conservation law for Hamiltonian
systems.

As in \cref{sec:global_hamiltonian}, we impose homogeneous Dirichlet
boundary conditions
$ \widehat{ \mathbf{z} } _h ^{\mathrm{tan}} \in \ringhat{ \mathbf{V} }
_h ^{\mathrm{tan}} $, but the arguments may be adapted to homogeneous
Neumann or other boundary conditions as described in
\cref{rmk:hamiltonian_bc}.

\subsection{Semidiscretized dynamics as systems of ODEs}

In \cref{sec:global_hamiltonian}, we used \cref{assumption:theta} to
express \eqref{eq:weak} as a Hamiltonian system of ODEs on a
symplectic vector space
$ \widetilde{ \mathbf{W} } _h = \widetilde{ W } _h \otimes
\mathbb{R}^2 $ in the case where
$ \mathbf{f} ( t, \mathbf{z} _h ) = \partial H / \partial \mathbf{z}
_h $ (\cref{thm:discrete_hamiltonian}). We begin by generalizing this
to arbitrary $\mathbf{f}$, where the resulting system of ODEs is not
necessarily Hamiltonian unless $\mathbf{f}$ is. Note that we do not
yet need the assumption that $ \mathbf{\Phi} $ is multisymplectic.

\begin{lemma}
  \label{lem:ftilde}
  Suppose \cref{assumption:theta} holds. Then
  $ ( \mathbf{z} _h , \widehat{ \mathbf{z} } _h ) = \mathbf{\Theta} (
  t, \widetilde{ \mathbf{z} } _h ) $ satisfies \eqref{eq:weak} if and
  only if
  \begin{equation}
    \label{eq:ztilde_ODE}
    \widetilde{ \mathbf{J} } \dot{ \widetilde{ \mathbf{z} } } _h = \widetilde{ \mathbf{f} } ( t, \widetilde{ \mathbf{z} } _h ) ,
  \end{equation}
  where
  $ \widetilde{ \mathbf{f} } \colon I \times \widetilde{ \mathbf{W} }
  _h \rightarrow \widetilde{ \mathbf{W} } _h $ is defined by
  \begin{equation}
    \label{eq:ftilde}
    \bigl( \widetilde{ \mathbf{f} } ( t, \widetilde{ \mathbf{z} } _h ), \widetilde{ \mathbf{w} } _h \bigr) _{ \mathcal{T} _h } = \bigl( \mathbf{f} ( t, \mathbf{z} _h ) , \mathbf{w} _h \bigr) _{ \mathcal{T} _h } - ( \mathbf{z} _h , \mathbf{D} \mathbf{w} _h ) _{ \mathcal{T} _h } - [ \widehat{ \mathbf{z} } _h , \mathbf{w} _h ] _{ \partial \mathcal{T} _h } , \quad \forall \widetilde{ \mathbf{w} } _h \in \widetilde{ \mathbf{W} } _h ,
  \end{equation}
  with
  $ \mathbf{w} _h = \frac{ \partial \mathbf{z} _h }{ \partial
    \widetilde{ \mathbf{z} } _h } \widetilde{ \mathbf{w} } _h $.
\end{lemma}

\begin{proof}
  As in the proof of \cref{thm:discrete_hamiltonian}, the chain rule
  and \cref{assumption:theta_symplectic} imply
  \begin{equation*}
    ( \widetilde{ \mathbf{J} } \dot{ \widetilde{ \mathbf{z} } } _h , \widetilde{ \mathbf{w} } _h ) _{ \mathcal{T} _h } = ( \mathbf{J} \dot{ \mathbf{z} } _h , \mathbf{w} _h ) _{ \mathcal{T} _h } .
  \end{equation*}
  Hence, this equals
  $ \bigl( \widetilde{ \mathbf{f} } ( t, \widetilde{ \mathbf{z} } _h
  ), \widetilde{ \mathbf{w} } _h \bigr) _{ \mathcal{T} _h } $ for all
  $ \widetilde{ \mathbf{w} } _h \in \widetilde{ \mathbf{W} } _h $ if
  and only if \eqref{eq:weak_w} holds, by
  \cref{assumption:theta_wtilde}.  Finally,
  \eqref{eq:weak_wnor}--\eqref{eq:weak_wtan} hold by
  \cref{assumption:theta_constraints}.
\end{proof}

Differentiating and applying the chain rule immediately gives us a
similar characterization of the variational equations.

\begin{corollary}
  \label{cor:ftilde_var}
  Under the assumptions of \cref{lem:ftilde},
  $ ( \mathbf{w} _i , \widehat{ \mathbf{w} } _i ) = \frac{ \partial
    \mathbf{\Theta} }{ \partial \widetilde{ \mathbf{z} } _h }
  \widetilde{ \mathbf{w} } _i $ satisfies \eqref{eq:weakvar} if and only if
  \begin{equation*}
    \widetilde{ \mathbf{J} } \dot{ \widetilde{ \mathbf{w} } } _i = \frac{ \partial \widetilde{ \mathbf{f} }}{ \partial \widetilde{ \mathbf{z} } _h } \widetilde{ \mathbf{w} } _i .
  \end{equation*}
\end{corollary}

The next result shows that multisymplecticity of \eqref{eq:weak}
corresponds to symplecticity of the corresponding system of
ODEs. (This is ultimately equivalent to
\cref{thm:discrete_hamiltonian} by an application of the Poincar\'e
lemma, cf.~\citet[Proposition 2.5.3]{MaRa1999}.)

\begin{theorem}
  Under the assumptions of \cref{lem:ftilde}, if $ \mathbf{\Phi} $ is
  multisymplectic and
  $ \frac{ \partial \mathbf{f} }{ \partial \mathbf{z} _h } $ is
  symmetric, then
  $ \frac{ \partial \widetilde{ \mathbf{f} } }{ \partial \widetilde{
      \mathbf{z} } _h } $ is also symmetric.
\end{theorem}

\begin{proof}
  Let
  $ \widetilde{ \mathbf{w} } _1 , \widetilde{ \mathbf{w} } _2 \in
  \widetilde{ \mathbf{W} } _h $ and
  $ ( \mathbf{w} _i , \widehat{ \mathbf{w} } _i ) = \frac{ \partial
    \mathbf{\Theta} }{ \partial \widetilde{ \mathbf{z} } _h }
  \widetilde{ \mathbf{w} } _i $. Differentiating \eqref{eq:ftilde}
  along $ \widetilde{ \mathbf{w} } _1 $ with
  $ \widetilde{ \mathbf{w} } _h = \widetilde{ \mathbf{w} } _2 $ gives
  \begin{equation*}
    \biggl( \frac{ \partial \widetilde{ \mathbf{f} } }{ \partial \widetilde{ \mathbf{z} } _h } \widetilde{ \mathbf{w} } _1 , \widetilde{ \mathbf{w} } _2 \biggr) _{ \mathcal{T} _h } = \biggl( \frac{ \partial \mathbf{f} }{ \partial \mathbf{z} _h } \mathbf{w} _1 , \mathbf{w} _2 \biggr) _{ \mathcal{T} _h } - ( \mathbf{w} _1 , \mathbf{D} \mathbf{w} _2 ) _{ \mathcal{T} _h } - [ \widehat{ \mathbf{w} } _1 , \mathbf{w} _2 ] _{ \partial \mathcal{T} _h } ,
  \end{equation*}
  and similarly,
  \begin{equation*}
    \biggl( \frac{ \partial \widetilde{ \mathbf{f} } }{ \partial \widetilde{ \mathbf{z} } _h } \widetilde{ \mathbf{w} } _2 , \widetilde{ \mathbf{w} } _1 \biggr) _{ \mathcal{T} _h } = \biggl( \frac{ \partial \mathbf{f} }{ \partial \mathbf{z} _h } \mathbf{w} _2 , \mathbf{w} _1 \biggr) _{ \mathcal{T} _h } - ( \mathbf{w} _2 , \mathbf{D} \mathbf{w} _1 ) _{ \mathcal{T} _h } - [ \widehat{ \mathbf{w} } _2 , \mathbf{w} _1 ] _{ \partial \mathcal{T} _h } .
  \end{equation*}
  Subtracting, using symmetry of $ \frac{ \partial \mathbf{f} }{ \partial \mathbf{z} _h } $, and integrating by parts with \eqref{eq:bracket_D} gives
  \begin{align*}
    \biggl( \frac{ \partial \widetilde{ \mathbf{f} } }{ \partial \widetilde{ \mathbf{z} } _h } \widetilde{ \mathbf{w} } _1 , \widetilde{ \mathbf{w} } _2 \biggr) _{ \mathcal{T} _h } - \biggl( \frac{ \partial \widetilde{ \mathbf{f} } }{ \partial \widetilde{ \mathbf{z} } _h } \widetilde{ \mathbf{w} } _2 , \widetilde{ \mathbf{w} } _1 \biggr) _{ \mathcal{T} _h } &= [ \mathbf{w} _1 , \mathbf{w} _2 ] _{ \partial \mathcal{T} _h } - [ \widehat{ \mathbf{w} } _1 , \mathbf{w} _2 ] _{ \partial \mathcal{T} _h } - [ \mathbf{w} _1 , \widehat{ \mathbf{w} } _2 ] _{ \partial \mathcal{T} _h } \\
    &= [ \widehat{ \mathbf{w} } _1 - \mathbf{w} _1 , \widehat{ \mathbf{w} } _2 - \mathbf{w} _2 ] _{ \partial \mathcal{T} _h } - [ \widehat{ \mathbf{w} } _1 , \widehat{ \mathbf{w} } _2 ] _{ \partial \mathcal{T} _h },
  \end{align*}
  which we claim vanishes. Indeed, differentiating
  \cref{assumption:theta_constraints} implies that
  $ ( \mathbf{w} _i , \widehat{ \mathbf{w} } _i ) $ satisfy
  \eqref{eq:weakvar_wnor}--\eqref{eq:weakvar_wtan}. Hence, the first
  right-hand-side term vanishes by \eqref{eq:weakvar_wnor} and
  multisymplecticity of $ \mathbf{\Phi} $, and the second
  right-hand-side term vanishes by \eqref{eq:weakvar_wtan} with
  $ \widehat{ \mathbf{w} } _1, \widehat{ \mathbf{w} } _2 \in \ringhat{
    \mathbf{V} } _h ^{\mathrm{tan}} $.
\end{proof}

\subsection{Numerical integrators for semidiscretized dynamics}

An $s$-stage Runge--Kutta (RK) method for \eqref{eq:ztilde_ODE} with
time-step size $ \Delta t = t ^1 - t ^0 $ can be written in the form
\begin{subequations}
  \label{eq:rk_tilde_dot}
  \begin{align}
    \widetilde{ \mathbf{Z} } _h ^i = \widetilde{ \mathbf{z} } _h ^0 + \Delta t \sum _{ j = 1 } ^s a _{ i j } \smash{\dot{ \widetilde{ \mathbf{Z} } }} _h ^j ,\\
    \widetilde{ \mathbf{z} } _h ^1  = \widetilde{ \mathbf{z} } _h ^0 + \Delta t \sum _{ i = 1 } ^s b _i \smash{\dot{ \widetilde{ \mathbf{Z} } }} _h ^i ,
  \end{align}
\end{subequations}
where
$ \widetilde{ \mathbf{J} } \smash{\dot{ \widetilde{ \mathbf{Z} } }} _h
^i \coloneqq \widetilde{ \mathbf{f} } ( T ^i , \widetilde{ \mathbf{Z}
} _h ^i ) $ and $ T ^i \coloneqq t ^0 + c _i \Delta t $. Applying
$ \widetilde{ \mathbf{J} } $ to both sides gives the equivalent form
\begin{subequations}
  \label{eq:rk_tilde}
  \begin{align}
    \widetilde{ \mathbf{J} } \widetilde{ \mathbf{Z} } _h ^i &= \widetilde{ \mathbf{J}} \widetilde{ \mathbf{z} } _h ^0 + \Delta t \sum _{ j = 1 } ^s a _{ i j } \widetilde{ \mathbf{f} } ( T ^j , \widetilde{ \mathbf{Z} } _h ^j  ) , \label{eq:rk_tilde_stage} \\
    \widetilde{ \mathbf{J} } \widetilde{ \mathbf{z} } _h ^1 &= \widetilde{ \mathbf{J}} \widetilde{ \mathbf{z} } _h ^0 + \Delta t \sum _{ i = 1 } ^s b _i \widetilde{ \mathbf{f} } ( T ^i , \widetilde{ \mathbf{Z} } _h ^i  ) . \label{eq:rk_tilde_step}
  \end{align}
\end{subequations}
Here, $ a _{ i j } $, $ b _i $, and $ c _i $ are given coefficients
specifying the method, often displayed as a \emph{Butcher tableau},
\begin{equation*}
  \begin{array}[b]{c|ccc}
    c _1 & a _{ 1 1 } & \cdots & a _{ 1 s } \\
    \vdots & \vdots & \ddots & \vdots \\
    c _s & a _{ s 1 } & \cdots & a _{ s s } \\
    \hline
    & b _1 & \cdots & b _s 
  \end{array} \quad .
\end{equation*}
Note that the ``dots'' in \eqref{eq:rk_tilde_dot} are not time
derivatives, since none of the variables are continuous-time paths;
rather, this is simply suggestive notation indicating the relationship
to the vector field $ \widetilde{ \mathbf{f} } $.\footnote{One
  exception to this warning: for RK methods corresponding to
  collocation methods,
  $ \smash{\dot{ \widetilde{ \mathbf{Z} } }} _h ^i $ is indeed the
  time derivative of the collocation polynomial at time $ T ^i $
  \citep[Chapter II]{HaLuWa2006}.}

We now establish a Runge--Kutta version of \cref{lem:ftilde}, showing
that the method can be implemented by solving a discrete-time
approximation of the weak problem \eqref{eq:weak}. We strengthen
\cref{assumption:theta_symplectic} slightly by assuming that
$ \widetilde{ \mathbf{z} } _h \mapsto \mathbf{z} _h $ is a
\emph{linear} symplectic map. This holds for all the methods we have
discussed, where
$ \widetilde{ \mathbf{W} } _h \hookrightarrow \mathbf{W} _h $ is the
inclusion map of a symplectic subspace.

\begin{theorem}
  \label{thm:rk_weak}
  Suppose \cref{assumption:theta} holds, with the additional condition
  that the map $ \widetilde{ \mathbf{z} } _h \mapsto \mathbf{z} _h $
  is linear. Then \eqref{eq:rk_tilde_stage} holds if and only if
  $ ( \mathbf{Z} _h ^i , \widehat{ \mathbf{Z} } _h ^i ) =
  \mathbf{\Theta} ( T ^i, \widetilde{ \mathbf{Z} } _h ^i ) $ satisfies
  \begin{subequations}
    \label{eq:rk_weak}
    \begin{align}
      ( \mathbf{J} \mathbf{Z}  _h ^i , \mathbf{w} _h ) _{ \mathcal{T} _h } + \Delta t \sum _{ j = 1 } ^s a _{ i j } \Bigl( ( \mathbf{Z} _h ^j , \mathbf{D} \mathbf{w} _h ) _{ \mathcal{T} _h } + [ \widehat{ \mathbf{Z} } _h ^j , \mathbf{w} _h ] _{ \partial \mathcal{T} _h }  \Bigr) &= ( \mathbf{J} \mathbf{z} _h ^0 , \mathbf{w} _h ) _{ \mathcal{T} _h } + \Delta t \sum _{ j = 1 } ^s a _{ i j } \bigl( \mathbf{f} ( T ^j , \mathbf{Z} _h ^j ) , \mathbf{w} _h \bigr) _{ \mathcal{T} _h } , \label{eq:rk_stage_w} \\
      \bigl\langle \mathbf{\Phi} ( \mathbf{Z} _h ^i , \widehat{ \mathbf{Z} } _h ^i ), \widehat{ \mathbf{w} } _h ^{\mathrm{nor}} \bigr\rangle _{ \partial \mathcal{T} _h } &= 0 , \label{eq:rk_stage_wnor} \\
      \langle \widehat{ \mathbf{Z} } _h ^{i, \mathrm{nor}} , \widehat{ \mathbf{w} } _h ^{\mathrm{tan}} \rangle _{ \partial \mathcal{T} _h } &= 0 , \label{eq:rk_stage_wtan}\\
      \intertext{for all $ \mathbf{w} _h \in \mathbf{W} _h $,
      $ \widehat{ \mathbf{w} } _h ^{\mathrm{nor}} \in \widehat{
      \mathbf{W} } _h ^{\mathrm{nor}} $, and
      $ \widehat{ \mathbf{w} } _h ^{\mathrm{tan}} \in \ringhat{
      \mathbf{V} } _h ^{\mathrm{tan}} $. Subsequently, \eqref{eq:rk_tilde_step}
      holds if and only if}
      ( \mathbf{J} \mathbf{z} _h ^1 , \mathbf{w} _h ) _{ \mathcal{T} _h } + \Delta t \sum _{ i = 1 } ^s b _i \Bigl( ( \mathbf{Z} _h ^i , \mathbf{D} \mathbf{w} _h ) _{ \mathcal{T} _h } + [ \widehat{ \mathbf{Z} } _h ^i , \mathbf{w} _h ] _{ \partial \mathcal{T} _h }  \Bigr) &= ( \mathbf{J} \mathbf{z} _h ^0 , \mathbf{w} _h ) _{ \mathcal{T} _h } + \Delta t \sum _{ i = 1 } ^s b _i \bigl( \mathbf{f} ( T ^i , \mathbf{Z} _h ^i ) , \mathbf{w} _h \bigr) _{ \mathcal{T} _h } , \label{eq:rk_step_w}
    \end{align}
  \end{subequations}
  for all $ \mathbf{w} _h \in \mathbf{W} _h $.
\end{theorem}

\begin{proof}
  First, since the linear map
  $ \widetilde{ \mathbf{z} } _h \mapsto \mathbf{z} _h $ is symplectic,
  by \cref{assumption:theta_symplectic}, it must be injective. Indeed,
  if $ \mathbf{z} _h = 0 $, then
  $ ( \widetilde{ \mathbf{J} } \widetilde{ \mathbf{z} } _h ,
  \widetilde{ \mathbf{w} } _h ) _{ \mathcal{T} _h } = ( \mathbf{J}
  \mathbf{z} _h , \mathbf{w} _h ) _{ \mathcal{T} _h } = 0 $ for all
  $ \widetilde{ \mathbf{w} } _h \in \widetilde{ \mathbf{W} } _h $, so
  nondegeneracy of the symplectic form
  $ ( \widetilde{ \mathbf{J} } \cdot , \cdot ) _{ \mathcal{T} _h } $
  on $ \widetilde{ \mathbf{W} } _h $ implies
  $ \widetilde{ \mathbf{z} } _h = 0 $. (This is a special case of the
  standard result that symplectic maps are immersions \citep[Exercise
  5.2-3]{MaRa1999}.) Applying this map to \eqref{eq:rk_tilde_dot}
  gives
  \begin{subequations}
    \label{eq:rk_dot}
    \begin{align}
      \mathbf{Z} _h ^i = \mathbf{z} _h ^0 + \Delta t \sum _{ j = 1 } ^s a _{ i j } \dot{ \mathbf{Z} } _h ^j , \label{eq:rk_dot_stage}\\
      \mathbf{z} _h ^1  = \mathbf{z} _h ^0 + \Delta t \sum _{ i = 1 } ^s b _i \dot{ \mathbf{Z} } _h ^i , \label{eq:rk_dot_step}
    \end{align}
  \end{subequations}
  which is thus equivalent to \eqref{eq:rk_tilde_dot} by
  injectivity. We emphasize the importance of the linearity assumption
  for this step, since it allows us to apply the map term-by-term.

  Next, by \cref{lem:ftilde},
  $ \widetilde{ \mathbf{J} } \smash{\dot{ \widetilde{ \mathbf{Z} } }} _h ^i =
  \widetilde{ \mathbf{f} } ( T ^i , \widetilde{ \mathbf{Z} } _h ^i ) $ is
  equivalent to
  \begin{equation*}
    ( \mathbf{J} \dot{\mathbf{Z}} _h ^i , \mathbf{w} _h ) _{ \mathcal{T} _h } + ( \mathbf{Z} _h ^i , \mathbf{D} \mathbf{w} _h ) _{ \mathcal{T} _h } + [ \widehat{ \mathbf{Z} } _h ^i , \mathbf{w} _h ] _{ \partial \mathcal{T} _h } = \bigl( \mathbf{f} ( T ^i, \mathbf{Z} _h ^i ) , \mathbf{w} _h \bigr)  _{ \mathcal{T} _h }  , \quad \forall \mathbf{w} _h \in \mathbf{W} _h .
  \end{equation*}
  Therefore, applying
  $ ( \mathbf{J} \cdot , \mathbf{w} _h ) _{ \mathcal{T} _h } $ to
  \eqref{eq:rk_dot} gives \eqref{eq:rk_stage_w} and
  \eqref{eq:rk_step_w}, which are thus equivalent to
  \eqref{eq:rk_tilde}. Finally,
  \eqref{eq:rk_stage_wnor}--\eqref{eq:rk_stage_wtan} hold by
  \cref{assumption:theta_constraints}, which completes the proof.
\end{proof}

\begin{remark}
  Note that the method \eqref{eq:rk_weak} involves numerical traces
  only for the internal stages $ \widehat{ \mathbf{Z} } _h ^i $, and we
  do not need to compute $ \widehat{ \mathbf{z} } _h ^0 $ or
  $ \widehat{ \mathbf{z} } _h ^1 $.
\end{remark}

\begin{example}
  \label{ex:midpoint}
  The implicit midpoint method is a $1$-stage RK method with tableau
  \begin{equation*}
    \def\arraystretch{1.4}
    \begin{array}[b]{c|c}
      \frac{1}{2} & \frac{1}{2} \\
      \hline
      & 1
    \end{array} \quad .
  \end{equation*}
  The internal stage at time $ T ^1 = \frac{1}{2} ( t ^0 + t ^1 ) $
  corresponds to the midpoint
  $ \mathbf{Z} _h ^1 = \frac{1}{2} ( \mathbf{z} _h ^0 + \mathbf{z} _h
  ^1 ) $. Denoting $ t ^{ 1/2} \coloneqq T ^1 $ and
  $ \mathbf{z} _h ^{ 1/2 } \coloneqq \mathbf{Z} _h ^1 $ to make this
  clear, \eqref{eq:rk_stage_w}--\eqref{eq:rk_stage_wtan} may then be
  written as
  \begin{align*}
    ( \mathbf{J} \mathbf{z} _h ^{1/2} , \mathbf{w} _h ) _{ \mathcal{T} _h } + \frac{1}{2} \Delta t \Bigl( ( \mathbf{z} _h ^{1/2} , \mathbf{D} \mathbf{w} _h ) _{ \mathcal{T} _h } + [ \widehat{ \mathbf{z} } _h ^{1/2} , \mathbf{w} _h ] _{ \partial \mathcal{T} _h }  \Bigr) &= ( \mathbf{J} \mathbf{z} _h ^0 , \mathbf{w} _h ) _{ \mathcal{T} _h } + \frac{1}{2} \Delta t \bigl( \mathbf{f} ( t ^{1/2} , \mathbf{z} _h ^{1/2} ) , \mathbf{w} _h \bigr) _{ \mathcal{T} _h } , \\
    \bigl\langle \mathbf{\Phi} ( \mathbf{z} _h ^{1/2} , \widehat{ \mathbf{z} } _h ^{1/2} ), \widehat{ \mathbf{w} } _h ^{\mathrm{nor}} \bigr\rangle _{ \partial \mathcal{T} _h } &= 0 , \\
    \langle \widehat{ \mathbf{z} } _h ^{1/2, \mathrm{nor}} , \widehat{ \mathbf{w} } _h ^{\mathrm{tan}} \rangle _{ \partial \mathcal{T} _h } &= 0 ,\\
    \intertext{for all $ \mathbf{w} _h \in \mathbf{W} _h $,
    $ \widehat{ \mathbf{w} } _h ^{\mathrm{nor}} \in \widehat{
    \mathbf{W} } _h ^{\mathrm{nor}} $, and \eqref{eq:rk_step_w} becomes}
    ( \mathbf{J} \mathbf{z} _h ^1 , \mathbf{w} _h ) _{ \mathcal{T} _h } + \Delta t \Bigl( ( \mathbf{z} _h ^{1/2} , \mathbf{D} \mathbf{w} _h ) _{ \mathcal{T} _h } + [ \widehat{ \mathbf{z} } _h ^{1/2} , \mathbf{w} _h ] _{ \partial \mathcal{T} _h }  \Bigr) &= ( \mathbf{J} \mathbf{z} _h ^0 , \mathbf{w} _h ) _{ \mathcal{T} _h } + \Delta t \bigl( \mathbf{f} ( t ^{1/2} , \mathbf{z} _h ^{1/2} ) , \mathbf{w} _h \bigr) _{ \mathcal{T} _h } ,
  \end{align*}
  for all $ \mathbf{w} _h \in \mathbf{W} _h $.
\end{example}

We next consider partitioned Runge--Kutta methods, which allow
different coefficients for the $q$ and $p$ components. Let
$ \mathbf{z} _h = ( q _h , p _h ) $ with $ q _h , p _h \in W _h $, and
let
$ \widetilde{ \mathbf{z} } _h = ( \widetilde{ q } _h , \widetilde{ p }
_h ) $ with
$ \widetilde{ q } _h , \widetilde{ p } _h \in \widetilde{ W } _h
$. Before introducing the methods, we first prove that, if the map
$ \widetilde{ \mathbf{z} } _h \mapsto \mathbf{z} _h $ partitions as
$ \widetilde{ q } _h \mapsto q _h $ and
$ \widetilde{ p } _h \mapsto p _h $, then \cref{assumption:theta}
translates to statements about these individual components.

\begin{lemma}
  \label{lem:thetapart}
  Suppose
  $ \mathbf{\Theta} \colon ( t, \widetilde{ \mathbf{z} } _h ) \mapsto
  ( \mathbf{z} _h , \widehat{ \mathbf{z} } _h ) $ satisfies
  \cref{assumption:theta}, and denote its components by
  $ \Theta _q \colon ( t, \widetilde{ q } _h , \widetilde{ p } _h )
  \mapsto ( q _h , \widehat{ q } _h ) $ and
  $ \Theta _p \colon ( t, \widetilde{ q } _h , \widetilde{ p } _h )
  \mapsto ( p _h , \widehat{ p } _h ) $. If
  $ \widetilde{ \mathbf{z} } _h \mapsto \mathbf{z} _h $ partitions as
  $ \widetilde{ q } _h \mapsto q _h $ and
  $ \widetilde{ p } _h \mapsto p _h $, then the following hold:
  \begin{lemmaenumerate}
  \item The equality
    $ ( s _h , r _h ) _{ \mathcal{T} _h } = ( \widetilde{ s } _h ,
    \widetilde{ r } _h ) _{ \mathcal{T} _h } $ holds with
    $ s _h = \frac{ \partial q _h }{ \partial \widetilde{ q } _h }
    \widetilde{ s } _h $ and
    $ r _h = \frac{ \partial p _h }{ \partial \widetilde{ p } _h }
    \widetilde{ r } _h $ for all
    $ \widetilde{ s } _h , \widetilde{ r } _h \in \widetilde{ W } _h
    $.\label{lem:thetapart_symplectic}

  \item If $ \dot{ \widetilde{ q } } _h \in \widetilde{ W } _h $ is
    such that
    \begin{subequations}
      \label{eq:weak_qp}
      \begin{align}
        ( \dot{q} _h , r _h ) _{ \mathcal{T} _h } + ( p _h , \mathrm{D} r _h ) _{ \mathcal{T} _h } + [ \widehat{ p } _h , r _h ] _{ \partial \mathcal{T} _h }
        &= \bigl( f _p (t, q _h , p _h ) , r _h \bigr) _{ \mathcal{T} _h } \label{eq:weak_r} \\
          \intertext{holds with
          $ \dot{q} _h = \frac{ \partial q _h }{ \partial \widetilde{ q } _h
          } \dot{ \widetilde{ q } } _h $ and
          $ r _h = \frac{ \partial p _h }{ \partial \widetilde{ p } _h }
          \widetilde{ r } _h $ for all
          $ \widetilde{ r } _h \in \widetilde{ W } _h $, then it holds for
          all $ r _h \in W _h $. Likewise, if
          $ \dot{ \widetilde{ p } } _h \in \widetilde{ W } _h $ is such that}
        -( \dot{p} _h , s _h ) _{ \mathcal{T} _h } + ( q _h , \mathrm{D} s _h ) _{ \mathcal{T} _h } + [ \widehat{ q } _h , s _h ] _{ \partial \mathcal{T} _h }
        &= \bigl( f _q (t, q _h , p _h ) , s _h \bigr) _{ \mathcal{T} _h } \label{eq:weak_s}
      \end{align}
    \end{subequations}
    holds with
    $ \dot{p} _h = \frac{ \partial p _h }{ \partial \widetilde{ p } _h
    } \dot{ \widetilde{ p } } _h $ and
    $ s _h = \frac{ \partial q _h }{ \partial \widetilde{ q } _h }
    \widetilde{ s } _h $ for all
    $ \widetilde{ s } _h \in \widetilde{ W } _h $, then it holds for
    all $ s _h \in W _h $.

  \item \label{lem:thetapart_constraints} If
    $ \mathbf{\Phi} ( \mathbf{z} _h , \widehat{ \mathbf{z} } _h )
    = \begin{bmatrix}
      \Phi _q ( q _h , \widehat{ q } _h ) \\
      \Phi _p ( p _h , \widehat{ p } _h )
    \end{bmatrix} $, then for all $ \widetilde{ q } _h , \widetilde{ p } _h \in \widetilde{ W } _h $, we have
    \begin{align*}
      \bigl\langle \Phi _q ( q _h , \widehat{ q } _h ) , \widehat{ s } _h ^{\mathrm{nor}} \bigr\rangle _{ \partial \mathcal{T} _h } &= 0 , \quad \forall \widehat{ s } _h ^{\mathrm{nor}} \in \widehat{ W } _h ^{\mathrm{nor}} , & \bigl\langle \Phi _p ( p _h , \widehat{ p } _h ) , \widehat{ r } _h ^{\mathrm{nor}} \bigr\rangle _{ \partial \mathcal{T} _h } &= 0 , \quad \forall \widehat{ r } _h ^{\mathrm{nor}} \in \widehat{ W } _h ^{\mathrm{nor}} , \\
      \langle \widehat{ q } _h ^{\mathrm{nor}} , \widehat{ s } _h ^{\mathrm{tan}} \rangle _{ \partial \mathcal{T} _h } &= 0 , \quad \forall \widehat{ s } _h ^{\mathrm{tan}} \in \ringhat{V} _h ^{\mathrm{tan}} , & \langle \widehat{ p } _h ^{\mathrm{nor}} , \widehat{ r } _h ^{\mathrm{tan}} \rangle _{ \partial \mathcal{T} _h } &= 0 , \quad \forall \widehat{ r } _h ^{\mathrm{tan}} \in \ringhat{V} _h ^{\mathrm{tan}} .
    \end{align*}
  \end{lemmaenumerate}
\end{lemma}

\begin{proof}\
  \begin{enumerate}[label=(\roman*)]
  \item This follows directly from \cref{assumption:theta_symplectic}
    with $ \widetilde{ \mathbf{w} } _1 = ( \widetilde{ s } _h , 0 ) $
    and $ \widetilde{ \mathbf{w} } _2 = ( 0 , \widetilde{ r } _h ) $.

  \item For the first statement, given
    $ \dot{ \widetilde{ q } } _h \in \widetilde{ W } _h $ such that
    \eqref{eq:weak_r} holds for all
    $ \widetilde{ r } _h \in \widetilde{ W } _h $, it follows from
    \cref{lem:thetapart_symplectic} that there exists unique
    $ \dot{ \widetilde{ p } } _h \in \widetilde{ W } _h $ satisfying
    \eqref{eq:weak_s} for all
    $ \widetilde{ s } _h \in \widetilde{ W } _h $. Hence,
    $ \dot{ \widetilde{ \mathbf{z} } } _h = ( \dot{ \widetilde{ q } }
    _h , \dot{ \widetilde{ p } } _h ) $ satisfies \eqref{eq:weak_w}
    for all
    $ \widetilde{ \mathbf{w} } _h = ( \widetilde{ s } _h , \widetilde{
      r } _h ) $, so \cref{assumption:theta_wtilde} implies
    \eqref{eq:weak_w} holds for all
    $ \mathbf{w} _h = ( s _h , r _h ) $. In particular,
    \eqref{eq:weak_r} holds for all $ r _h \in W _h $. The proof of
    the second statement, starting with \eqref{eq:weak_s}, is
    essentially the same.

  \item This is immediate from \cref{assumption:theta_constraints}.\qedhere
  \end{enumerate} 
\end{proof}

\begin{remark}
  We do not necessarily assume that $ \mathbf{\Theta} $ partitions
  into
  $ ( t, \widetilde{ q } _h) \mapsto ( q _h , \widehat{ q } _h ) $ and
  $ (t, \widetilde{ p } _h) \mapsto ( p _h , \widehat{ p } _h ) $,
  since $ \Theta _q $ and $ \Theta _p $ may each depend on all of
  $ ( t, \widetilde{ q } _h , \widetilde{ p } _h ) $. However, there
  are many cases in which it does, in particular:
  \begin{itemize}
  \item For the AFW-H method, the map $ \mathbf{\Theta} $ described in
    \cref{ex:afw_discrete_hamiltonian} generally does not partition,
    since $ \widehat{ \mathbf{z} } _h ^{\mathrm{nor}} $ depends on
    $ \mathbf{f} ( t, \mathbf{z} _h ) $, i.e.,
    $ \widehat{ q } _h ^{\mathrm{nor}} $ depends on
    $ f _q ( t, q _h , p _h ) $ and
    $ \widehat{ p } _h ^{\mathrm{nor}} $ depends on
    $ f _p ( t, q _h , p _h ) $. However, $ \mathbf{\Theta} $ does
    partition when $ \mathbf{f} $ is \emph{separable}, meaning that
    $ f _q = f _q ( t, q _h ) $ is independent of $p _h $ and
    $ f _p = f _p ( t, p _h ) $ is independent of $q _h $.

  \item For the LDG-H method with $ \boldsymbol{\alpha} =
    \begin{bmatrix}
      \alpha _q \\
      & \alpha _p
    \end{bmatrix} $, the map $ \mathbf{\Theta} $ described in \cref{ex:ldg-h_discrete_hamiltonian} always partitions, even for non-separable $ \mathbf{f} $. This form of $ \boldsymbol{\alpha} $ is also needed for $ \mathbf{\Phi} $ to partition.

  \item For the semilinear Hodge wave equation, $ \mathbf{f} $ is
    separable. Hence, the maps $ \mathbf{\Theta} $ described in
    \cref{sec:hodge_wave_hamiltonian} partition for both the AFW-H
    method and the multisymplectic LDG-H method.
  \end{itemize}
\end{remark}

As a consequence of \cref{lem:thetapart}, we get a partitioned version
of \cref{lem:ftilde}.

\begin{corollary}
  \label{cor:partitioned_ftilde}
  Under the hypotheses of \cref{lem:thetapart}, \eqref{eq:weak_r}
  holds for all $ r _h \in W _h $ if and only if
  \begin{subequations}
    \label{eq:qptilde_ODE}
    \begin{equation}
      \label{eq:qtilde_ODE}
      \dot{ \widetilde{ q } _h } = \widetilde{ f } _p ( t , \widetilde{ q } _h , \widetilde{ p } _h ) ,
    \end{equation}
    where
    $ \widetilde{ f } _p \colon I \times \widetilde{ W } _h \rightarrow
    \widetilde{ W } _h $ is defined by
    \begin{equation*}
      \bigl( \widetilde{ f } _p ( t, \widetilde{ q } _h , \widetilde{ p } _h ), \widetilde{ r } _h \bigr) _{ \mathcal{T} _h } = \bigl( f _p ( t, q _h , p _h ) , r _h \bigr) _{ \mathcal{T} _h } - ( p _h , \mathrm{D} r _h ) _{ \mathcal{T} _h } - [ \widehat{ p } _h , r _h ] _{ \partial \mathcal{T} _h } , \quad \forall \widetilde{ r } _h \in \widetilde{ W } _h ,
    \end{equation*}
    with
    $ r _h = \frac{ \partial p _h }{ \partial \widetilde{ p } _h }
    \widetilde{ r } _h $. Likewise, \eqref{eq:weak_s} holds for all
    $ s _h \in W _h $ if and only if
    \begin{equation}
      \label{eq:ptilde_ODE}
      - \dot{ \widetilde{ p } _h } = \widetilde{ f } _q ( t , \widetilde{ q } _h , \widetilde{ p } _h ) ,
    \end{equation}
  \end{subequations}
  where $ \widetilde{ f } _q \colon I \times \widetilde{ W } _h \rightarrow
  \widetilde{ W } _h $ is defined by
  \begin{equation*}
    \bigl( \widetilde{ f } _q ( t, \widetilde{ q } _h , \widetilde{ p } _h ), \widetilde{ s } _h \bigr) _{ \mathcal{T} _h } = \bigl( f _q ( t, q _h , p _h ) , s _h \bigr) _{ \mathcal{T} _h } - ( q _h , \mathrm{D} s _h ) _{ \mathcal{T} _h } - [ \widehat{ q } _h , s _h ] _{ \partial \mathcal{T} _h } , \quad \forall \widetilde{ s } _h \in \widetilde{ W } _h ,
  \end{equation*}
  with
  $ s _h = \frac{ \partial q _h }{ \partial \widetilde{ s } _h }
  \widetilde{ s } _h $.
\end{corollary}

Now, an $s$-stage partitioned Runge--Kutta (PRK) method for
\eqref{eq:qptilde_ODE} takes the form
\begin{subequations}
  \label{eq:prk_tilde}
  \begin{align}
    \widetilde{ Q } _h ^i &= \widetilde{ q } _h ^0 + \Delta t \sum _{ j = 1 } ^s a _{ i j } \smash{\dot{ \widetilde{ Q } }} _h ^j , & \widetilde{ P } _h ^i &= \widetilde{ p }  _h ^0 + \Delta t \sum _{ j = 1 } ^s \bar{a} _{ i j } \smash{\dot{ \widetilde{ P } }} _h ^j , \label{eq:prk_tilde_stage} \\
    \widetilde{ q } _h ^1 &= \widetilde{ q } _h ^0 + \Delta t \sum _{ i = 1 } ^s b _i \smash{\dot{ \widetilde{ Q } }} _h ^i , & \widetilde{ p } _h ^1 &= \widetilde{ p } _h ^0 + \Delta t \sum _{ i = 1 } ^s \bar{b} _i \smash{\dot{ \widetilde{ P } }} _h ^i , \label{eq:prk_tilde_step}
  \end{align}
\end{subequations}
where
$ \smash{\dot{ \widetilde{ Q } }} _h ^i \coloneqq \widetilde{ f } _p ( T ^i ,
\widetilde{ Q } _h ^i , \widetilde{ P } _h ^i ) $ with
$ T ^i \coloneqq t ^0 + c _i \Delta t $, and
$ - \smash{\dot{ \widetilde{ P } }} _h ^i \coloneqq \widetilde{ f } _q ( \bar{T} ^i , \widetilde{ Q } _h ^i , \widetilde{ P } _h ^i ) $ with
$ \bar{T} _i \coloneqq t ^0 + \bar{c} _i \Delta t $. The coefficients
are generally presented as a pair of Butcher tableaux,
\begin{equation*}
  \begin{array}[b]{c|ccc}
    c _1 & a _{ 1 1 } & \cdots & a _{ 1 s } \\
    \vdots & \vdots & \ddots & \vdots \\
    c _s & a _{ s 1 } & \cdots & a _{ s s } \\
    \hline\\[-2.5ex]
    & b _1 & \cdots & b _s
  \end{array}
  \qquad \qquad
  \begin{array}[b]{c|ccc}
    \bar{c} _1 & \bar{a} _{ 1 1 } & \cdots & \bar{a} _{ 1 s } \\
    \vdots & \vdots & \ddots & \vdots \\
    \bar{c} _s & \bar{a} _{ s 1 } & \cdots & \bar{a} _{ s s } \\
    \hline\\[-2.5ex]
    & \bar{b} _1 & \cdots & \bar{b} _s
  \end{array} \quad .
\end{equation*}
The method reduces to an ordinary RK method when the two tableaux are
identical.

The following is a partitioned version of \cref{thm:rk_weak}, showing
that this class of methods may also be implemented by solving a weak
problem.

\begin{theorem}
  \label{thm:prk_weak}
  Suppose \cref{assumption:theta} holds, with the additional condition
  that the map $ \widetilde{ \mathbf{z} } _h \mapsto \mathbf{z} _h $
  partitions into linear maps $ \widetilde{ q } _h \mapsto q _h $ and
  $ \widetilde{ p } _h \mapsto p _h $. Furthermore, as in
  \cref{lem:thetapart_constraints}, suppose $ \mathbf{\Phi} $ partitions into
  $ \mathbf{\Phi} ( \mathbf{z} _h , \widehat{ \mathbf{z} } _h )
  = \begin{bmatrix}
    \Phi _q ( q _h , \widehat{ q } _h ) \\
    \Phi _p ( p _h , \widehat{ p } _h )
  \end{bmatrix} $. Then \eqref{eq:prk_tilde_stage} holds if and only if $ ( Q _h ^i , \widehat{ Q } _h ^i ) = \Theta _q ( T ^i , \widetilde{ Q } _h ^i , \widetilde{ P } _h ^i ) $ and $ ( P _h ^i , \widehat{ P } _h ^i ) = \Theta _p ( \bar{T} ^i , \widetilde{ Q } _h ^i , \widetilde{ P } _h ^i ) $ satisfy
  \begin{subequations}
    \label{eq:prk_weak}
    \begin{align}
      ( Q _h ^i , r _h ) _{ \mathcal{T} _h } + \Delta t \sum _{ j = 1 } ^s a _{ i j } \Bigl( ( P _h ^j , \mathrm{D} r _h ) _{ \mathcal{T} _h } + [ \widehat{ P } _h ^j , r _h ] _{ \partial \mathcal{T} _h } \Bigr) &= ( q _h ^0 , r _h ) _{ \mathcal{T} _h } + \Delta t \sum _{ j = 1 } ^s a _{ i j } \bigl( f _p ( T ^j , Q _h ^j , P _h ^j ), r _h \bigr) _{ \mathcal{T} _h } ,\\
      -( P _h ^i , s _h ) _{ \mathcal{T} _h } + \Delta t \sum _{ j = 1 } ^s \bar{a} _{ i j } \Bigl( ( Q _h ^j , \mathrm{D} s _h ) _{ \mathcal{T} _h } + [ \widehat{ Q } _h ^j , s _h ] _{ \partial \mathcal{T} _h } \Bigr) &= -( p _h ^0 , s _h ) _{ \mathcal{T} _h } + \Delta t \sum _{ j = 1 } ^s \bar{a} _{ i j } \bigl( f _q ( \bar{T} _j , Q _h ^j , P _h ^j ), s _h \bigr) _{ \mathcal{T} _h } ,\\
      \bigl\langle \Phi _q ( Q _h ^i , \widehat{ Q } _h ^i ) , \widehat{ s } _h ^{\mathrm{nor}} \bigr\rangle _{ \partial \mathcal{T} _h } &= 0 , \\
      \bigl\langle \Phi _p ( P _h ^i , \widehat{ P } _h ^i ) , \widehat{ r } _h ^{\mathrm{nor}} \bigr\rangle _{ \partial \mathcal{T} _h } &= 0 , \\
      \langle \widehat{ Q } _h ^{i, \mathrm{nor}} , \widehat{ s } _h ^{\mathrm{tan}} \rangle _{ \partial \mathcal{T} _h } &= 0 ,\\
      \langle \widehat{ P } _h ^{i, \mathrm{nor}} , \widehat{ r } _h ^{\mathrm{tan}} \rangle _{ \partial \mathcal{T} _h } &= 0 ,\\
      \intertext{for all $ s _h , r _h \in W _h $; $ \widehat{ s } _h ^{\mathrm{nor}} , \widehat{ r } _h ^{\mathrm{nor}} \in \widehat{ W } _h ^{\mathrm{nor}} $; and $ \widehat{ s } _h ^{\mathrm{tan}} , \widehat{ r } _h ^{\mathrm{tan}} \in \ringhat{V} _h ^{\mathrm{tan}} $. Subsequently, \eqref{eq:prk_tilde_step} holds if and only if}
      ( q _h ^1 , r _h ) _{ \mathcal{T} _h } + \Delta t \sum _{ i = 1 } ^s b _ i \Bigl( ( P _h ^i , \mathrm{D} r _h ) _{ \mathcal{T} _h } + [ \widehat{ P } _h ^i , r _h ] _{ \partial \mathcal{T} _h } \Bigr) &= ( q _h ^0 , r _h ) _{ \mathcal{T} _h } + \Delta t \sum _{ i = 1 } ^s b _i \bigl( f _p ( T ^i , Q _h ^i , P _h ^i ), r _h \bigr) _{ \mathcal{T} _h } ,\\
      -( p _h ^1 , s _h ) _{ \mathcal{T} _h } + \Delta t \sum _{ i = 1 } ^s \bar{b} _i \Bigl( ( Q _h ^i , \mathrm{D} s _h ) _{ \mathcal{T} _h } + [ \widehat{ Q } _h ^i , s _h ] _{ \partial \mathcal{T} _h } \Bigr) &= -( p _h ^0 , s _h ) _{ \mathcal{T} _h } + \Delta t \sum _{ i = 1 } ^s \bar{b} _i \bigl( f _q ( \bar{T} ^i , Q _h ^i , P _h ^i ), s _h \bigr) _{ \mathcal{T} _h } ,
    \end{align}
  \end{subequations}
  for all $ s _h , r _h \in W _h $.
\end{theorem}

\begin{proof}
  Similarly to the proof of \cref{thm:rk_weak}, we apply the linear
  maps $ \widetilde{ q } _h \mapsto q _h $ and
  $ \widetilde{ p } _h \mapsto p _h $, which are injective by
  \cref{lem:thetapart_symplectic}, to the corresponding parts of
  \eqref{eq:prk_tilde}, obtaining the equivalent system
  \begin{subequations}
    \label{eq:prk}
    \begin{align}
      Q _h ^i &= q _h ^0 + \Delta t \sum _{ j = 1 } ^s a _{ i j } \dot{ Q } _h ^j , & P _h ^i &= p  _h ^0 + \Delta t \sum _{ j = 1 } ^s \bar{a} _{ i j } \dot{ P } _h ^j , \label{eq:prk_stage} \\
      q _h ^1 &= q _h ^0 + \Delta t \sum _{ i = 1 } ^s b _i \dot{ Q } _h ^i , & p _h ^1 &= p _h ^0 + \Delta t \sum _{ i = 1 } ^s \bar{b} _i \dot{ P } _h ^i . \label{eq:prk_step}
    \end{align}
  \end{subequations}
  By \cref{cor:partitioned_ftilde},
  $ \smash{\dot{ \widetilde{ Q } }} _h ^i = \widetilde{ f } _p ( T ^i ,
  \widetilde{ Q } _h ^i , \widetilde{ P } _h ^i ) $ and
  $ - \smash{\dot{ \widetilde{ P } }} _h ^i = \widetilde{ f } _q ( \bar{T} ^i ,
  \widetilde{ Q } _h ^i , \widetilde{ P } _h ^i ) $ are equivalent to
  \begin{alignat*}{2}
    ( \dot{Q} _h ^i , r _h ) _{ \mathcal{T} _h } + ( P _h ^i , \mathrm{D} r _h ) _{ \mathcal{T} _h } + [ \widehat{ P } _h ^i , r _h ] _{ \partial \mathcal{T} _h }
    &= \bigl( f _p ( T ^i, Q _h ^i , P _h ^i ) , r _h \bigr) _{ \mathcal{T} _h }, \quad &\forall r _h &\in W _h ,\\
    -( \dot{P} _h ^i , s _h ) _{ \mathcal{T} _h } + ( Q _h ^i , \mathrm{D} s _h ) _{ \mathcal{T} _h } + [ \widehat{ Q } _h ^i , s _h ] _{ \partial \mathcal{T} _h }
    &= \bigl( f _q ( \bar{T} ^i , Q _h ^i , P _h ^i ) , s _h \bigr) _{ \mathcal{T} _h } , \quad &\forall s _h &\in W _h .
  \end{alignat*}
  We note the importance of having established \eqref{eq:weak_r} and
  \eqref{eq:weak_s} separately: this allows us to apply
  \cref{cor:partitioned_ftilde} with $ t = T ^i $ and
  $ t = \bar{T} ^i $, respectively, even when these times are
  distinct. Substituting into \eqref{eq:prk} and applying
  \cref{lem:thetapart_constraints} completes the proof.
\end{proof}

\begin{example}
  \label{ex:verlet}
  The St\"ormer/Verlet method is a PRK method with tableaux%
  \footnote{\citeauthor{HaLuWa2006} define St\"ormer/Verlet slightly
    differently, taking $ \bar{c} _1 = \bar{c} _2 = \frac{1}{2} $
    \citep[Table II.2.1]{HaLuWa2006}. Although these methods coincide
    for autonomous systems, \citet{Jay2021} has recently shown that
    the version above is preferred for non-autonomous systems---and in
    particular, that it is symplectic, whereas the version with
    $ \bar{c} _1 = \bar{c} _2 = \frac{1}{2} $ is not.}
  \begin{equation*}
    \def\arraystretch{1.4}
    \begin{array}[b]{c|cc}
      0 & 0 & 0  \\
      1 & \frac{1}{2} & \frac{1}{2} \\
      \hline
           & \frac{1}{2} & \frac{1}{2}
    \end{array}
    \qquad \qquad
    \begin{array}[b]{c|cc}
      0 & \frac{1}{2} & 0 \\
      1 & \frac{1}{2} & 0 \\
      \hline
        & \frac{1}{2} & \frac{1}{2}
    \end{array} \quad .
  \end{equation*}
  The expression and implementation of the method can be simplified by
  observing that the stages satisfy $ Q _h ^1 = q _h ^0 $ (since
  $ a _{ 1 j } = 0 $), $ Q _h ^2 = q _h ^1 $ (since
  $ a _{ 2 j } = b _j $), and $ P _1 = P _2 $ (since
  $ \bar{a} _{ 1 j } = \bar{a} _{ 2 j } $). Denoting
  $ p _h ^{ 1/2 } \coloneqq P _h ^1 = P _h ^2 $ and
  $ \widehat{ p } _h ^{ 1/2 } \coloneqq \frac{1}{2} ( \widehat{ P } _h
  ^1 + \widehat{ P } _h ^2 ) $, \eqref{eq:prk_weak} can be expressed
  as the following three-step ``leapfrog'' procedure: First, find
  $ p _h ^{ 1/2 } $ and $ \widehat{ q } _h ^0 $ satisfying
  \begin{align*}
    - ( p _h ^{ 1/2 } , s _h ) _{ \mathcal{T} _h } + \frac{1}{2} \Delta t \Bigl( ( q _h ^0, \mathrm{D} s _h ) _{ \mathcal{T} _h } + [ \widehat{ q } _h ^0 , s _h ] _{ \partial \mathcal{T} _h } \Bigr) &= - ( p _h ^0 , s _h ) _{ \mathcal{T} _h } + \frac{1}{2} \Delta t \bigl(  f _q ( t ^0 , q _h ^0, p _h ^{ 1/2} ), s _h \bigr) _{ \mathcal{T} _h } ,\\
    \bigl\langle \Phi _q ( q _h ^0 , \widehat{ q } _h ^0 ) , \widehat{ s } _h ^{\mathrm{nor}} \bigr\rangle _{ \partial \mathcal{T} _h } &= 0 , \\
    \langle \widehat{ q } _h ^{0, \mathrm{nor}} , \widehat{ s } _h ^{\mathrm{tan}} \rangle _{ \partial \mathcal{T} _h } &= 0 ,
  \end{align*}
  for all $ s _h $ and $ \widehat{ s } _h $. Next, find $ q _h ^1 $
  and $ \widehat{ p } _h ^{ 1/2 } $ satisfying
  \begin{align*}
    ( q _h ^1 , r _h ) _{ \mathcal{T} _h } + \Delta t \Bigr( ( p _h ^{ 1/2}, \mathrm{D} r _h ) _{ \mathcal{T} _h } + [ \widehat{ p } _h ^{ 1/2 } , r _h ] _{ \partial \mathcal{T} _h } \Bigr) &= ( q _h ^0, r _h ) _{ \mathcal{T} _h } + \frac{1}{2} \Delta t \bigl( f _p ( t ^0 , q _h ^0 , p _h ^{ 1/2 } ) , r _h \bigr) _{ \mathcal{T} _h } \\
    & \hphantom{{}= ( q _h ^0, r _h ) _{ \mathcal{T} _h } } + \frac{1}{2} \Delta t \bigl( f _p ( t ^1, q _h ^1, p _h ^{ 1/2 } ), r _h \bigr)_{ \mathcal{T} _h } , \\
    \bigl\langle \Phi _p ( p _h ^{1/2} , \widehat{ p } _h ^{1/2} ) , \widehat{ r } _h ^{\mathrm{nor}} \bigr\rangle _{ \partial \mathcal{T} _h } &= 0 , \\
    \langle \widehat{ p } _h ^{1/2, \mathrm{nor}} , \widehat{ r } _h ^{\mathrm{tan}} \rangle _{ \partial \mathcal{T} _h } &= 0 ,
  \end{align*}
  for all $ r _h $ and $ \widehat{ r } _h $. Finally, find $ p _h ^1 $
  and $ \widehat{ q } _h ^1 $ satisfying
  \begin{align*}
    - ( p _h ^1 , s _h ) _{ \mathcal{T} _h } + \frac{1}{2} \Delta t \Bigl( ( q _h ^1 , \mathrm{D} s _h ) _{ \mathcal{T} _h } + [ \widehat{ q } _h ^1 , s _h ] _{ \partial \mathcal{T} _h } \Bigr) &= - ( p _h ^{1/2} , s _h ) _{ \mathcal{T} _h } + \frac{1}{2} \Delta t \bigl( f _q ( t ^1 , q _h ^1, p _h ^{ 1/2 } ) , s _h \bigr) _{ \mathcal{T} _h } ,\\
    \bigl\langle \Phi _q ( q _h ^1 , \widehat{ q } _h ^1 ) , \widehat{ s } _h ^{\mathrm{nor}} \bigr\rangle _{ \partial \mathcal{T} _h } &= 0 , \\
    \langle \widehat{ q } _h ^{1, \mathrm{nor}} , \widehat{ s } _h ^{\mathrm{tan}} \rangle _{ \partial \mathcal{T} _h } &= 0 ,
  \end{align*}
  for all $ s _h $ and $ \widehat{ s } _h $. When $ \mathbf{f} $ is
  separable, each of these steps requires \emph{only a linear solve},
  even if $ \mathbf{f} $ is nonlinear. This corresponds to the fact
  that St\"ormer/Verlet is \emph{explicit} for separable systems.
\end{example}

\begin{example}
  To illustrate \cref{ex:verlet} more concretely, we now give an
  explicit description of a method for the semilinear Hodge wave
  equation that applies St\"ormer/Verlet time-stepping to the
  multisymplectic LDG-H semidiscretization \eqref{eq:ldg-h_wave}.  The
  method advances
  $ ( u _h ^0 , p _h ^0 ) \mapsto ( u _h ^1 , p _h ^1 ) $ according to
  the following procedure:

  \begin{enumerate}[leftmargin=0pt, itemindent=*, label=\textsc{Step \arabic*.}]
  \item As in \cref{thm:ldg-h_wave_definite}, find
    $ ( \sigma _h ^0, \widehat{ \sigma } _h ^{0, \mathrm{tan}} ) \in W
    _h ^{ k -1 } \times \ringhat{V} _h ^{k-1, \mathrm{tan}} $
    satisfying
    \begin{equation*}
      - ( \sigma _h ^0 , \tau _h ) _{ \mathcal{T} _h } + \bigl\langle \alpha ^{ k -1 } ( \widehat{ \sigma } _h ^{0, \mathrm{tan}} - \sigma _h ^{0, \mathrm{tan}} ) , \widehat{ \tau } _h ^{\mathrm{tan}} - \tau _h ^{\mathrm{tan}} \bigr\rangle _{ \partial \mathcal{T} _h } = ( \delta u _h ^0 , \tau _h ) _{ \mathcal{T} _h } + \langle u _h ^{0, \mathrm{nor}} , \widehat{ \tau } _h ^{\mathrm{tan}} \rangle _{ \partial \mathcal{T} _h } ,
    \end{equation*}
    for all
    $ ( \tau _h , \widehat{ \tau } _h ^{\mathrm{tan}} ) \in W _h ^{ k
      -1 } \times \ringhat{V} _h ^{k-1, \mathrm{tan}} $, and find
    $ ( \rho _h ^0 , \widehat{ u } _h ^{0, \mathrm{tan}} ) \in W _h
    ^{k+1} \times \ringhat{V} _h ^{k, \mathrm{tan}} $ satisfying
    \begin{equation*}
      ( \rho _h ^0 , \eta _h ) _{ \mathcal{T} _h } + \langle \alpha ^k \widehat{ u } _h ^{0, \mathrm{tan}} , \widehat{ v } _h ^{\mathrm{tan}} \rangle _{ \partial \mathcal{T} _h } + \langle \widehat{ u } _h ^{0, \mathrm{tan}} , \eta _h ^{\mathrm{nor}} \rangle _{ \partial \mathcal{T} _h } - \langle \rho _h ^{0, \mathrm{nor}} , \widehat{ v } _h ^{\mathrm{tan}} \rangle _{ \partial \mathcal{T} _h } = - ( u _h ^0 , \delta \eta _h ) _{ \mathcal{T} _h } + \langle \alpha ^k u _h ^{0, \mathrm{tan}} , \widehat{ v } _h ^{\mathrm{tan}} \rangle _{ \partial \mathcal{T} _h } ,
    \end{equation*}
    for all
    $ ( \eta _h , \widehat{ v } _h ^{\mathrm{tan}} ) \in W _h ^{k+1}
    \times \ringhat{V} _h ^{k, \mathrm{tan}} $. Then, find
    $ p _h ^{ 1/2 } \in W _h ^k $ satisfying
    \begin{multline*}
      ( p _h ^{ 1/2 } , v _h ) _{ \mathcal{T} _h } = ( p _h ^0 , v _h ) _{ \mathcal{T} _h } + \frac{1}{2} \Delta t \Bigl( - \bigl( f ( t ^0 , u _h ^0 ) , v _h \bigr) _{ \mathcal{T} _h } \\
      + ( \sigma _h ^0 , \delta v _h ) _{ \mathcal{T} _h } + ( \delta \rho _h ^0 , v _h ) _{ \mathcal{T} _h } + \langle \widehat{ \sigma } _h ^{0, \mathrm{tan}} , v _h ^{\mathrm{nor}} \rangle _{ \partial \mathcal{T} _h } + \bigl\langle \alpha ^k ( \widehat{ u } _h ^{0, \mathrm{tan}} - u _h ^{0, \mathrm{tan}} ) , v _h ^{\mathrm{tan}} \bigr\rangle _{ \partial \mathcal{T} _h } \Bigr) 
      ,
    \end{multline*}
    for all $ v _h \in W _h ^k $.

  \item Take $ u _h ^1 = u _h ^0 + \Delta t \, p _h ^{ 1/2 } $.

  \item Similarly to the first step, find
    $ ( \sigma _h ^1, \widehat{ \sigma } _h ^{1, \mathrm{tan}} ) \in W
    _h ^{ k -1 } \times \ringhat{V} _h ^{k-1, \mathrm{tan}} $
    satisfying
    \begin{equation*}
      - ( \sigma _h ^1 , \tau _h ) _{ \mathcal{T} _h } + \bigl\langle \alpha ^{ k -1 } ( \widehat{ \sigma } _h ^{1, \mathrm{tan}} - \sigma _h ^{1, \mathrm{tan}} ) , \widehat{ \tau } _h ^{\mathrm{tan}} - \tau _h ^{\mathrm{tan}} \bigr\rangle _{ \partial \mathcal{T} _h } = ( \delta u _h ^1 , \tau _h ) _{ \mathcal{T} _h } + \langle u _h ^{1, \mathrm{nor}} , \widehat{ \tau } _h ^{\mathrm{tan}} \rangle _{ \partial \mathcal{T} _h } ,
    \end{equation*}
    for all
    $ ( \tau _h , \widehat{ \tau } _h ^{\mathrm{tan}} ) \in W _h ^{ k
      -1 } \times \ringhat{V} _h ^{k-1, \mathrm{tan}} $, and
    $ ( \rho _h ^1 , \widehat{ u } _h ^{1, \mathrm{tan}} ) \in W _h
    ^{k+1} \times \ringhat{V} _h ^{k, \mathrm{tan}} $ satisfying
    \begin{equation*}
      ( \rho _h ^1 , \eta _h ) _{ \mathcal{T} _h } + \langle \alpha ^k \widehat{ u } _h ^{1, \mathrm{tan}} , \widehat{ v } _h ^{\mathrm{tan}} \rangle _{ \partial \mathcal{T} _h } + \langle \widehat{ u } _h ^{1, \mathrm{tan}} , \eta _h ^{\mathrm{nor}} \rangle _{ \partial \mathcal{T} _h } - \langle \rho _h ^{1, \mathrm{nor}} , \widehat{ v } _h ^{\mathrm{tan}} \rangle _{ \partial \mathcal{T} _h } = - ( u _h ^1 , \delta \eta _h ) _{ \mathcal{T} _h } + \langle \alpha ^k u _h ^{1, \mathrm{tan}} , \widehat{ v } _h ^{\mathrm{tan}} \rangle _{ \partial \mathcal{T} _h } ,
    \end{equation*}
    for all
    $ ( \eta _h , \widehat{ v } _h ^{\mathrm{tan}} ) \in W _h ^{k+1}
    \times \ringhat{V} _h ^{k, \mathrm{tan}} $. Then, find
    $ p _h ^1 \in W _h ^k $ satisfying
    \begin{multline*}
      ( p _h ^1 , v _h ) _{ \mathcal{T} _h } = ( p _h ^{1/2} , v _h ) _{ \mathcal{T} _h } + \frac{1}{2} \Delta t \Bigl( - \bigl( f ( t ^1 , u _h ^1 ) , v _h \bigr) _{ \mathcal{T} _h } \\
      + ( \sigma _h ^1 , \delta v _h ) _{ \mathcal{T} _h } + ( \delta \rho _h ^1 , v _h ) _{ \mathcal{T} _h } + \langle \widehat{ \sigma } _h ^{1, \mathrm{tan}} , v _h ^{\mathrm{nor}} \rangle _{ \partial \mathcal{T} _h } + \bigl\langle \alpha ^k ( \widehat{ u } _h ^{1, \mathrm{tan}} - u _h ^{1, \mathrm{tan}} ) , v _h ^{\mathrm{tan}} \bigr\rangle _{ \partial \mathcal{T} _h } \Bigr) 
      ,
    \end{multline*}
    for all $ v _h \in W _h ^k $.
  \end{enumerate} 
  Note that this method is explicit, in that it only requires solving
  linear variational problems, even when $f$ is nonlinear. Moreover,
  \textsc{Step 3} can be combined with the subsequent \textsc{Step 1},
  e.g.,
    \begin{multline*}
      ( p _h ^{3/2} , v _h ) _{ \mathcal{T} _h } = ( p _h ^{1/2} , v _h ) _{ \mathcal{T} _h } + \Delta t \Bigl( - \bigl( f ( t ^1 , u _h ^1 ) , v _h \bigr) _{ \mathcal{T} _h } \\
      + ( \sigma _h ^1 , \delta v _h ) _{ \mathcal{T} _h } + ( \delta \rho _h ^1 , v _h ) _{ \mathcal{T} _h } + \langle \widehat{ \sigma } _h ^{1, \mathrm{tan}} , v _h ^{\mathrm{nor}} \rangle _{ \partial \mathcal{T} _h } + \bigl\langle \alpha ^k ( \widehat{ u } _h ^{1, \mathrm{tan}} - u _h ^{1, \mathrm{tan}} ) , v _h ^{\mathrm{tan}} \bigr\rangle _{ \partial \mathcal{T} _h } \Bigr) ,
    \end{multline*}
    resulting in a ``leapfrog'' procedure where $ u _h $ is computed
    at integer steps and $ p _h $ at half-integer steps. In practice,
    this means that---except for the very first half-step---the linear
    variational system only needs to be solved once rather than twice
    per time step.
\end{example}

\subsection{Symplectic integrators and the discrete multisymplectic
  conservation law}

We now show that, when a multisymplectic semidiscretization method in
space is combined with a symplectic (P)RK method, the resulting
numerical scheme satisfies a discrete multisymplectic conservation
law. The argument is a direct application of the theory of
\emph{functional equivariance} developed in
\citet{McSt2024}---specifically, quadratic functional equivariance for
symplectic RK methods and bilinear functional equivariance for
symplectic PRK methods---and extends the results of \citep[Section
4.5]{McSt2024} for the special case of time-dependent de~Donder--Weyl
systems.

We recall that an RK method conserves quadratic invariants
\citep{Cooper1987} and is therefore symplectic
\citep{Lasagni1988,Sanz-Serna1988,BoSc1994} if its coefficients satisfy
\begin{equation}
  \label{eq:srk}
  b _i b _j - b _i a _{ i j } - b _j a _{ j i } = 0 , \quad \forall i, j = 1, \ldots, s ,
\end{equation}
with implicit midpoint and other Gauss--Legendre
collocation methods being primary examples. A PRK method conserves
bilinear invariants and is therefore symplectic if its coefficients
satisfy
\begin{subequations}
  \label{eq:sprk}
  \begin{alignat}{2}
    b _i \bar{b} _j - b _i \bar{a} _{ i j } - \bar{b} _j a _{ j i } &= 0 , \quad &\forall i, j &= 1, \ldots, s , \label{eq:sprk_ab}\\
    b _i &= \bar{b} _i , \quad &\forall i &= 1, \ldots, s , \label{eq:sprk_b}\\
    c _i &= \bar{c} _i , \quad &\forall i &= 1, \ldots, s , \label{eq:sprk_c}
  \end{alignat}
\end{subequations}
with St\"ormer/Verlet and other Lobatto IIIA--IIIB methods being
primary examples. The conditions \eqref{eq:sprk_ab}--\eqref{eq:sprk_b}
for autonomous systems appear in \citep[Equation 2.2]{Suris1990} and
\citep[Equation 2.5]{Sun1993}; the addition of \eqref{eq:sprk_c} for
non-autonomous systems can be found in \citep[Equation
2.17]{Sanz-Serna2016}. See also \citep[Sections IV.2 and
VI.4]{HaLuWa2006} and references therein.

Now, suppose $ \partial \mathbf{f} / \partial \mathbf{z} _h $ is
symmetric and $ \mathbf{\Phi} $ is multisymplectic, and consider the
system
\begin{equation*}
  \widetilde{ \mathbf{J} } \dot{ \widetilde{ \mathbf{z} } } _h = \widetilde{ \mathbf{f} } ( t, \widetilde{ \mathbf{z} } _h ), \qquad
  \widetilde{ \mathbf{J} } \dot{ \widetilde{ \mathbf{w} } } _1 = \frac{ \partial \widetilde{ \mathbf{f} }}{ \partial \widetilde{ \mathbf{z} } _h } \widetilde{ \mathbf{w} } _1, \qquad
  \widetilde{ \mathbf{J} } \dot{ \widetilde{ \mathbf{w} } } _2 = \frac{ \partial \widetilde{ \mathbf{f} }}{ \partial \widetilde{ \mathbf{z} } _h } \widetilde{ \mathbf{w} } _2, \qquad
  \dot{ \zeta } = - [ \widehat{ \mathbf{w} } _1 , \widehat{ \mathbf{w} } _2 ] _{ \partial K } ,
\end{equation*}
for any $ K \in \mathcal{T} _h $. This describes the simultaneous
evolution of a solution to \eqref{eq:weak}, by \cref{lem:ftilde}; two
variations satisfying \eqref{eq:weakvar}, by \cref{cor:ftilde_var};
and the observable
$ \zeta = ( \mathbf{J} \mathbf{w} _1, \mathbf{w} _2 ) _K $, by
\cref{thm:local_ms}. We define the vector field
\begin{equation*}
  \widetilde{ \mathbf{g} } ( t, \widetilde{ \mathbf{z} } _h , \widetilde{ \mathbf{w} } _1 , \widetilde{ \mathbf{w} } _2 , \zeta ) = \biggl( \widetilde{ \mathbf{f} } ( t, \widetilde{ \mathbf{z} } _h ), \frac{ \partial \widetilde{ \mathbf{f} }}{ \partial \widetilde{ \mathbf{z} } _h } \widetilde{ \mathbf{w} } _1, \frac{ \partial \widetilde{ \mathbf{f} }}{ \partial \widetilde{ \mathbf{z} } _h } \widetilde{ \mathbf{w} } _2, - [ \widehat{ \mathbf{w} } _1 , \widehat{ \mathbf{w} } _2 ] _{ \partial K } \biggr),
\end{equation*}
corresponding to this system. If $ \Psi $ is a numerical integrator
with time-step size $ \Delta t $, we denote its application to the
vector field $ \widetilde{ \mathbf{f} } $ by
$ \Psi _{ \widetilde{ \mathbf{f} } } \colon I \times \widetilde{
  \mathbf{W} } _h \rightarrow \widetilde{ \mathbf{W} } _h $,
$ ( t ^0, \widetilde{ \mathbf{z} } _h ^0 ) \mapsto \widetilde{
  \mathbf{z} } _h ^1 $, and likewise its application to the vector
field $ \widetilde{ \mathbf{g} } $ by
$ \Psi _{ \widetilde{ \mathbf{g} } } \colon I \times ( \widetilde{
  \mathbf{W} } _h ) ^3 \times \mathbb{R} \rightarrow ( \widetilde{
  \mathbf{W} } _h ) ^3 \times \mathbb{R} $. When $ \Psi $ is a PRK
method, we partition all three copies of
$ \widetilde{ \mathbf{W} } _h $ in the same way, i.e.,
$ ( \widetilde{ q } _h , \widetilde{ s } _1 , \widetilde{ s } _2 ) $
in one part and
$ ( \widetilde{ p } _h , \widetilde{ r } _1 , \widetilde{ r } _2 ) $
in the other.

\begin{theorem}
  \label{thm:discrete_mscl}
  Let $ \partial \mathbf{f} / \partial \mathbf{z} _h $ be symmetric
  and $ \mathbf{\Phi} $ be multisymplectic. Suppose that either
  \begin{enumerate}[label=(\roman*)]
  \item the hypotheses of \cref{thm:rk_weak} hold, and $ \Psi $ is an
    RK method satisfying \eqref{eq:srk};
    or \label{thm:discrete_mscl_rk}
  \item the hypotheses of \cref{thm:prk_weak} hold, and $ \Psi $ is a
    PRK method satisfying
    \eqref{eq:sprk}. \label{thm:discrete_mscl_prk}
  \end{enumerate}
  Then we have
  \begin{equation}
    \label{eq:Psi_g}
    \Bigl(  \widetilde{ \mathbf{z} } _h ^1 , \widetilde{ \mathbf{w} } _1 ^1 , \widetilde{ \mathbf{w} } _2 ^1 , ( \mathbf{J} \mathbf{w} _1 ^1 , \mathbf{w} _2 ^1 ) _K \Bigr) = \Psi _{ \widetilde{ \mathbf{g} } } \Bigl( t ^0, \widetilde{ \mathbf{z} } _h ^0 , \widetilde{ \mathbf{w} } _1 ^0 , \widetilde{ \mathbf{w} } _2 ^0 , ( \mathbf{J} \mathbf{w} _1 ^0 , \mathbf{w} _2 ^0 ) _K \Bigr),
  \end{equation}
  where
  \begin{equation*}
    \widetilde{ \mathbf{z} } _h ^1 = \Psi _{ \widetilde{ \mathbf{f} } } ( t ^0 , \widetilde{ \mathbf{z} } _h ^0 ) , \qquad \widetilde{ \mathbf{w} } _1 ^1 = \frac{ \partial \widetilde{ \mathbf{z} } _h ^1 }{ \partial \widetilde{ \mathbf{z} } _h ^0 } \widetilde{ \mathbf{w} } _1 ^0 , \qquad \widetilde{ \mathbf{w} } _2 ^1 = \frac{ \partial \widetilde{ \mathbf{z} } _h ^1 }{ \partial \widetilde{ \mathbf{z} } _h ^0 } \widetilde{ \mathbf{w} } _2 ^0 .
  \end{equation*}
  The equality of the last components of \eqref{eq:Psi_g} can be
  written as
  \begin{equation}
    \label{eq:discrete_mscl}
    ( \mathbf{J} \mathbf{w} _1 ^1 , \mathbf{w} _2 ^1 ) _K =  ( \mathbf{J} \mathbf{w} _1 ^0 , \mathbf{w} _2 ^0 ) _K - \Delta t \sum _{ i = 1 } ^s b _i [ \widehat{ \mathbf{W} } _1 ^i , \widehat{ \mathbf{W} } _2 ^i ] _{ \partial K } ,
  \end{equation}
  which we call the \emph{discrete multisymplectic conservation law}.
\end{theorem}

\begin{proof}
  If \ref{thm:discrete_mscl_rk} holds, then linearity of
  $ \widetilde{ \mathbf{z} } _h \mapsto \mathbf{z} _h $ implies that
  $ ( \mathbf{J} \mathbf{w} _1 , \mathbf{w} _2 ) _K $ is quadratic in
  $ \widetilde{ \mathbf{w} } _1 , \widetilde{ \mathbf{w} } _2 $. Since
  RK methods satisfying \eqref{eq:srk} conserve quadratic invariants
  and are therefore quadratic functionally equivariant
  \citep[Corollary 2.10(b)]{McSt2024}, the conclusion follows from the
  general results in \citet[Section 2.4.4]{McSt2024}: compare
  \eqref{eq:Psi_g} with \citep[Equation 5]{McSt2024} and
  \eqref{eq:discrete_mscl} with \citep[Equation 6]{McSt2024}.

  Similarly, if \ref{thm:discrete_mscl_prk} holds, then linearity of
  $ \widetilde{ q } _h \mapsto q _h $ and
  $ \widetilde{ p } _h \mapsto p _h $ implies that
  $ ( \mathbf{J} \mathbf{w} _1 , \mathbf{w} _2 ) _K = ( s _1 , r _2 )
  _K - ( r _1 , s _2 ) _K $ is bilinear in
  $ \widetilde{ s } _1 , \widetilde{ s } _2 $ and
  $ \widetilde{ r } _1 , \widetilde{ r } _2 $. Since PRK methods
  satisfying \eqref{eq:sprk} conserve bilinear invariants and are
  therefore bilinear functionally equivariant \citep[Example
  5.18]{McSt2024}, the conclusion follows as in \citep[Section
  5.3]{McSt2024}.
\end{proof}

\begin{remark}
  Although the results in \citet{McSt2024} are stated for autonomous
  systems, they are readily extended to non-autonomous systems with
  only minor modifications. In particular, \eqref{eq:sprk_b} is
  sufficient to ensure that a PRK method is affine functionally
  equivariant for autonomous systems \citep[Example 5.17]{McSt2024},
  while the additional condition \eqref{eq:sprk_c} allows the argument
  to extend to non-autonomous systems. The sufficiency of
  \eqref{eq:sprk} for bilinear functional equivariance of PRK methods
  can also be seen directly from \citet[Lemma 2.5]{Sanz-Serna2016}.
\end{remark}

\begin{example}
  For the implicit midpoint method in \cref{ex:midpoint}, the discrete
  multisymplectic conservation law on $ K \in \mathcal{T} _h $ takes
  the form
  \begin{equation*}
    (  \mathbf{J} \mathbf{w} _1 ^1 , \mathbf{w} _2 ^1 ) _K = ( \mathbf{J} \mathbf{w} _1 ^0 , \mathbf{w} _2 ^0 ) _K - \Delta t [ \widehat{ \mathbf{w} } _1 ^{ 1/2 } , \widehat{ \mathbf{w} } _2 ^{ 1/2 } ] _{ \partial K } .
  \end{equation*}
\end{example}

\begin{remark}
  For strongly multisymplectic methods, as in \cref{sec:strong_ms} we
  may replace $ K \in \mathcal{T} _h $ by any collection of elements
  $ \mathcal{K} \subset \mathcal{T} _h $, obtaining the discrete
  multisymplectic conservation law
  \begin{equation*}
    ( \mathbf{J} \mathbf{w} _1 ^1 , \mathbf{w} _2 ^1 ) _{ \mathcal{K} } = ( \mathbf{J} \mathbf{w} _1 ^0 , \mathbf{w} _2 ^0 ) _{ \mathcal{K} } - \Delta t \sum _{ i = 1 } ^s b _i [ \widehat{ \mathbf{W} } _1 ^i , \widehat{ \mathbf{W} } _2 ^i ] _{\partial (\overline{\bigcup\mathcal{K}})} .
  \end{equation*}
\end{remark}

\section{Numerical examples}
\label{sec:numerical_examples}

We illustrate the behavior of these methods by considering the
$ k = 1 $ semilinear Hodge wave equation in dimension $ n = 2 $. For
$ \Omega \subset \mathbb{R}^2 $, we may identify
$ H \Lambda (\Omega) $ with the complex of scalar and vector
``proxies''
\begin{equation*}
  0 \rightarrow H ^1 (\Omega) \xrightarrow{ \operatorname{curl} } H ( \operatorname{div} ; \Omega ) \xrightarrow{ \operatorname{div} } L ^2 (\Omega) \rightarrow 0
\end{equation*}
and $ H ^\ast \Lambda (\Omega) $ with the dual complex
\begin{equation*}
  0 \leftarrow L ^2 (\Omega) \xleftarrow{ \operatorname{rot} } H ( \operatorname{rot} ; \Omega ) \xleftarrow{ -\operatorname{grad} } H ^1 (\Omega) \leftarrow 0 ,
\end{equation*}
where $\operatorname{curl} \tau \coloneqq (\partial_y \tau, -\partial_x \tau)$ and
$\operatorname{rot} v \coloneqq \partial_x v_y - \partial_y v_x$. As in \citep[Table
1]{AwFaGuSt2023}, proxies for tangential and normal traces on
$ \partial K $ are
\begin{equation*}
  \tau ^{\mathrm{tan}} = \tau \rvert _{ \partial K } , \qquad v ^{\mathrm{nor}} = v \times \widehat{ n } , \qquad v ^{\mathrm{tan}} = ( v \cdot \widehat{ n } ) \widehat{ n } , \qquad \eta ^{\mathrm{nor}} = \eta \widehat{ n } ,
\end{equation*}
where $ \widehat{ n } $ is the outer unit normal vector and
$ v \times \widehat{ n } \coloneqq v _x \widehat{ n } _y - v _y
\widehat{ n } _x $.

Using this vector calculus correspondence, \eqref{eq:hodge_wave} gives
a first-order formulation of the semilinear vector wave equation,
\begin{subequations}
  \begin{align}
    \dot{\sigma} + \operatorname{rot} p &= 0, \\
    \dot{u} &= p, \\
    \dot{\rho} + \operatorname{div} p &= 0, \\
    \operatorname{rot} u &= -\sigma, \\
    -\dot{p} + \operatorname{curl} \sigma - \operatorname{grad} \rho &= \frac{\partial F}{\partial u}, \\
    \operatorname{div} u &= -\rho .
  \end{align}
\end{subequations}
Recall that the dynamics of $ ( u , p ) $ correspond to the global Hamiltonian
\begin{equation*}
  \mathcal{H} ( t, u , p ) = \frac{1}{2} \bigl( \lVert \sigma \rVert _\Omega ^2 + \lVert p \rVert _\Omega ^2 + \lVert \rho \rVert _\Omega ^2 \bigr) + \int _\Omega F ( t, x, u ) \,\mathrm{vol},
\end{equation*}
which we previously denoted by
$ \widetilde{ \mathcal{H} } _{ \mathbf{S} } $ in
\cref{ex:global_hamiltonian_wave}.

For the discretization, we employ the LDG-H methods introduced in
\cref{sec:hodge_wave_methods}, using equal-order spaces and
piecewise-constant penalties. To satisfy the hypothesis that
$ \alpha ^0 $ be negative-definite and $ \alpha ^1 $ be
positive-definite, required throughout
\cref{sec:hodge_wave_methods,sec:hodge_wave_hamiltonian}, the penalty
constants must satisfy $ \alpha ^0 _e < 0 $ and $ \alpha ^1 _e > 0 $
on each facet $ e \subset \partial \mathcal{T} _h $. The
multisymplectic LDG-H method \eqref{eq:ldg-h_wave} reads: Find
\begin{equation*}
  ( u _h , p _h , \sigma _h , \rho _h , \widehat{ \sigma } _h ^{\mathrm{tan}} , \widehat{ u } _h ^{\mathrm{tan}} ) \colon I \rightarrow W _h ^1 \times W _h ^1 \times W _h ^0 \times W _h ^2 \times \ringhat{ V } _h ^{0, \mathrm{tan}} \times \ringhat{ V } _h ^{1, \mathrm{tan}}
\end{equation*}
satisfying the dynamical equations
\begin{alignat*}{2}
  ( \dot{u} _h , r _h ) _{ \mathcal{T} _h } &= ( p _h , r _h ) _{ \mathcal{T} _h } , \quad &\forall r _h &\in W _h ^1 , \\
  \begin{multlined}[b]
    - ( \dot{p} _h , v _h ) _{ \mathcal{T} _h } + ( \sigma _h , \operatorname{rot} v _h ) _{ \mathcal{T} _h } - ( \operatorname{grad} \rho _h , v _h ) _{ \mathcal{T} _h } \\
    + \langle \widehat{ \sigma } _h ^{\mathrm{tan}} , v _h \times \widehat{ n } \rangle _{ \partial \mathcal{T} _h } + \bigl\langle \alpha ^1 ( \widehat{ u } _h ^{\mathrm{tan}} - u _h ) \cdot \widehat{ n } , v _h \cdot \widehat{ n } \bigr\rangle _{ \partial \mathcal{T} _h }
  \end{multlined} &= \biggl( \frac{ \partial F }{ \partial u _h } , v _h \biggr) _{ \mathcal{T} _h } , \quad &\forall v _h &\in W _h ^1 , \\
  \intertext{together with the constraints}
  ( \operatorname{rot} u _h , \tau _h ) _{ \mathcal{T} _h } + \bigl\langle \alpha ^0 ( \widehat{ \sigma } _h ^{\mathrm{tan}} - \sigma _h ) , \tau _h \bigr\rangle _{ \partial \mathcal{T} _h } &= - ( \sigma _h , \tau _h ) _{ \mathcal{T} _h } , \quad &\forall \tau _h &\in W _h ^0 , \\
  -( u _h , \operatorname{grad} \eta _h ) _{ \mathcal{T} _h } + \langle \widehat{ u } _h ^{\mathrm{tan}} \cdot  \widehat{ n } , \eta _h \rangle _{ \partial \mathcal{T} _h } &= - ( \rho _h , \eta _h )  _{ \mathcal{T} _h } , \quad &\forall \eta _h &\in W _h ^2 , \\
  \intertext{and the conservativity conditions}
  \bigl\langle u _h \times \widehat{ n } - \alpha ^0 ( \widehat{ \sigma } _h ^{\mathrm{tan}} - \sigma _h ) , \widehat{ \tau } _h ^{\mathrm{tan}} \bigr\rangle _{ \partial \mathcal{T} _h } &= 0 , \quad &\forall \widehat{ \tau } _h ^{\mathrm{tan}} &\in \ringhat{V} _h ^{0, \mathrm{tan}}, \\
  \bigl\langle \rho _h - \alpha ^1 ( \widehat{ u } _h ^{\mathrm{tan}} - u _h ) \cdot \widehat{ n } , \widehat{ v } _h ^{\mathrm{tan}} \cdot \widehat{ n } \bigr\rangle _{ \partial \mathcal{T} _h } &= 0 , \quad &\forall \widehat{ v } _h ^{\mathrm{tan}} &\in \ringhat{V} _h ^{1, \mathrm{tan}} .
\end{alignat*}
By \cref{thm:ldg-h_discrete_hamiltonian}, the semidiscrete dynamics of
$ ( u _h , p _h ) $ correspond to the discrete Hamiltonian
\begin{multline*}
  \mathcal{H} _h ( t, u _h , p _h ) = \frac{1}{2} \Bigl( \lVert \sigma _h \rVert _\Omega ^2 + \lVert p _h \rVert _\Omega ^2 + \lVert \rho _h \rVert _\Omega ^2 \Bigr) + \int _\Omega F ( t, x, u _h ) \,\mathrm{vol} \\
  - \frac{1}{2} \bigl\langle \alpha ^0 ( \widehat{ \sigma } _h ^{\mathrm{tan}} - \sigma _h ), \widehat{ \sigma } _h ^{\mathrm{tan}} - \sigma _h \bigr\rangle _{ \partial \mathcal{T} _h } + \frac{1}{2} \bigl\langle \alpha ^1 ( \widehat{ u } _h ^{\mathrm{tan}} - u _h ) \cdot \widehat{ n } , ( \widehat{ u } _h ^{\mathrm{tan}} - u _h ) \cdot \widehat{ n } \bigr\rangle _{ \partial \mathcal{T} _h } .
\end{multline*}
On the other hand, the non-multisymplectic LDG-H method
\eqref{eq:ldg-h_wave_mixed} reads: Find
\begin{equation*}
  ( \sigma _h , u _h , \rho _h , p _h , \widehat{ \sigma } _h ^{\mathrm{tan}} , \widehat{ p } _h ^{\mathrm{tan}} ) \colon I \rightarrow W _h ^0 \times W _h ^1 \times W _h ^2 \times W _h ^1 \times \ringhat{ V } _h ^{0, \mathrm{tan}} \times \ringhat{ V } _h ^{1, \mathrm{tan}}
\end{equation*}
satisfying the dynamical equations
\begin{alignat*}{2}
  ( \dot{ \sigma } _h , \tau _h ) _{ \mathcal{T} _h } + ( \operatorname{rot} p _h , \tau _h ) _{ \mathcal{T} _h } + \bigl\langle \alpha ^0 ( \widehat{ \sigma } _h ^{\mathrm{tan}} - \sigma _h ) , \tau _h \bigr\rangle _{ \partial \mathcal{T} _h } &= 0, \quad &\forall \tau _h &\in W _h ^0 , \\
  ( \dot{ u } _h , r _h ) _{ \mathcal{T} _h } &= ( p _h , r _h ) _{ \mathcal{T} _h } , \quad &\forall r _h &\in W _h ^1 , \\
  ( \dot{ \rho } _h , \eta _h ) _{ \mathcal{T} _h } - ( p _h , \operatorname{grad} \eta _h ) _{ \mathcal{T} _h } + \langle \widehat{ p } _h ^{\mathrm{tan}} \cdot \widehat{ n } , \eta _h \rangle _{ \partial \mathcal{T} _h } &= 0, \quad &\forall \eta _h &\in W _h ^2 , \\
  \begin{multlined}[b]
    - ( \dot{ p } _h , v _h ) _{ \mathcal{T} _h } + ( \sigma _h , \operatorname{rot} v _h ) _{ \mathcal{T} _h } - ( \operatorname{grad} \rho _h , v _h ) _{ \mathcal{T} _h } \\
    + \langle \widehat{ \sigma } _h ^{\mathrm{tan}} , v _h \times \widehat{ n } \rangle _{ \partial \mathcal{T} _h } + \bigl\langle \alpha ^1 ( \widehat{ p } _h ^{\mathrm{tan}} - p _h ) \cdot \widehat{ n } , v _h \cdot \widehat{ n }  \bigr\rangle _{ \partial \mathcal{T} _h }
  \end{multlined} &= \biggl( \frac{ \partial F }{ \partial u _h } , v _h \biggr)  _{ \mathcal{T} _h } ,\quad &\forall v _h &\in W _h ^1 , \\
  \intertext{together with the conservativity conditions}
  \bigl\langle p _h \times \widehat{ n } - \alpha ^0 ( \widehat{ \sigma } _h ^{\mathrm{tan}} - \sigma _h ) , \widehat{ \tau } _h ^{\mathrm{tan}} \bigr\rangle _{ \partial \mathcal{T} _h } &= 0 , \quad &\forall \widehat{ \tau } _h ^{\mathrm{tan}} &\in \ringhat{V} _h ^{0, \mathrm{tan}}, \\
  \bigl\langle \rho _h - \alpha ^1 ( \widehat{ p } _h ^{\mathrm{tan}} - p _h ) \cdot \widehat{ n } , \widehat{ v } _h ^{\mathrm{tan}} \cdot \widehat{ n } \bigr\rangle _{ \partial \mathcal{T} _h } &= 0 , \quad &\forall \widehat{ v } _h ^{\mathrm{tan}} &\in \ringhat{V} _h ^{1, \mathrm{tan}} .
\end{alignat*}
All numerical computations were performed using NGSolve
\citep{Schoeberl2014}. The code used to conduct these experiments is
freely available from
\url{https://github.com/EnricoZampa/HamiltonianLDG}.

\subsection{The linear homogeneous case}
We first consider the linear homogeneous vector wave equation
($ F = 0 $) on the domain $ \Omega = [0, 1] \times [0, 0.1] $ with
periodic boundary conditions. We compute numerical solutions to the
problem whose exact solution is the traveling plane wave
\begin{equation*}
  u ( t, x , y ) = \Bigl( 0, \sin \bigl( 4 \pi ( x - t ) \bigr) \Bigr) .
\end{equation*}
Notice that $ u _y $ is a traveling-wave solution to the
one-dimensional scalar wave equation, allowing comparison with
\citet[Example 4.3]{SaCiNgPe17}. In contrast to \citep{SaCiNgPe17},
however, we are computing a solution to the vector wave equation on a
two-dimensional strip discretized by an unstructured triangle mesh,
rather than the scalar wave equation on a one-dimensional grid.

We now compare the behavior of the multisymplectic and
non-multisymplectic LDG-H methods. For both methods, we semidiscretize
using equal-order spaces with polynomial degree $ r = 1 $, constant
penalties $ - \alpha ^0 = \alpha ^1 = 0.05 $, and mesh size
$ h = 0.025 $. Following \citep[Example 4.3]{SaCiNgPe17}, we then
integrate in time with $ \Delta t = h $ using a symplectic diagonally
implicit RK method of order $6$, obtained by composing several steps
of implicit midpoint with weights discovered by \citet[Table 1,
Solution A]{Yoshida1990}; see also \citet[Equation V.3.11]{HaLuWa2006}
and \citep[Table A.1]{SaCiNgPe17}.

\begin{figure}
  \centering
  \includegraphics{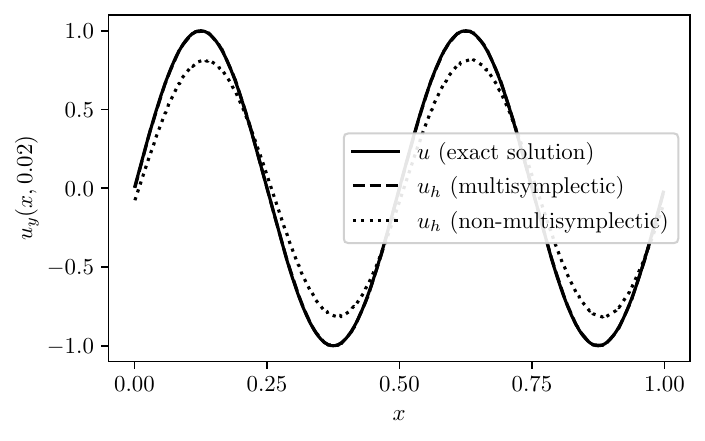}
  \caption{Cross-section of the $y$-component of the exact and
    numerical solutions, taken along the line $y = 0.02 $ at time
    $ T = 20 $. The multisymplectic LDG-H method nearly matches the
    exact solution, whereas the non-multisymplectic LDG-H method shows
    amplitude decrease due to energy dissipation, as well as phase
    error.\label{fig:u_comparison_linear}}
\end{figure}

\Cref{fig:u_comparison_linear} shows a cross-section of the exact
solution $u$ and the numerical solutions $ u _h $ for the
multisymplectic and non-multisymplectic LDG-H methods at time
$ T = 20 $. The multisymplectic LDG-H solution is barely
distinguishable from the exact solution with amplitude $1$. By
contrast, the non-multisymplectic LDG-H solution shows substantial
amplitude decrease due to the energy dissipation described in
\cref{lem:ldg-h_wave_mixed_energy}, as well as visible phase
error. Compare \citep[Figures 1 and 2]{SaCiNgPe17}.

\begin{figure}
  \centering
  \includegraphics{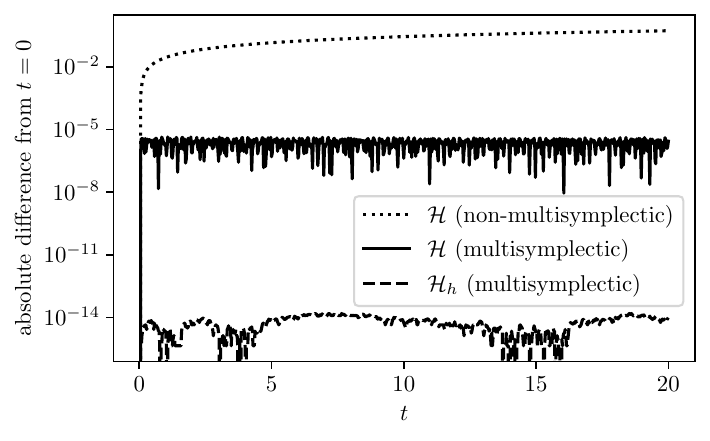}
  \caption{Absolute error in the global Hamiltonian $ \mathcal{H} $
    and discrete Hamiltonian $ \mathcal{H} _h $ along numerical
    solutions. The multisymplectic LDG-H method conserves
    $ \mathcal{H} _h $ up to floating-point error and nearly conserves
    $ \mathcal{H} $, whereas the non-multisymplectic LDG-H method
    shows significant drift due to its
    dissipativity.\label{fig:H_comparison_linear}}
\end{figure}

\Cref{fig:H_comparison_linear} shows the evolution of the global
Hamiltonian $\mathcal{H}$ for both LDG-H methods and of the discrete
Hamiltonian $ \mathcal{H} _h $ for the multisymplectic LDG-H
method. Since we are applying a symplectic RK method, and these
Hamiltonians are quadratic in the linear case, their conservation or
lack thereof is due to the spatial semidiscretization rather than the
time discretization. The multisymplectic LDG-H method conserves
$ \mathcal{H} _h $ in exact arithmetic, and the
$ \approx 10 ^{ - 14 } $ errors seen here are on the order of
accumulated floating-point error. The multisymplectic method also
nearly conserves $\mathcal{H}$ within $ \approx 10 ^{ - 5 } $, with
bounded errors reflecting the difference between $\mathcal{H} _h $ and
$ \mathcal{H} $. On the other hand, the dissipativity of the
non-multisymplectic LDG-H method leads to large energy drift, with
error $ \approx 10 ^0 $ by the final time $ T = 20 $. Compare
\citep[Figure 3]{SaCiNgPe17}.

\subsection{The nonlinear case}
We now consider a nonlinear example with
\begin{equation*}
  F ( t, x, u ) = \frac{1}{2} \lvert u \rvert ^2 - \frac{ 1 }{ 4 } \lvert u \rvert ^4 , \qquad f ( t, x, u ) = \frac{ \partial F }{ \partial u } = \bigl( 1 - \lvert u \rvert ^2 \bigr) u ,
\end{equation*}
which is a cubic nonlinear vector Klein--Gordon equation, akin to
\citet[Example 2]{SaVa24}. We take the domain to be the unit square
$ \Omega = [ 0, 1 ] ^2 $ with periodic boundary conditions and compute
a numerical solution to the problem whose exact solution is the
traveling plane wave
\begin{equation*}
  u ( t, x, y ) = \frac{1}{2} \Bigl( \cos \bigl( 2 \pi ( x + y ) - \theta t \bigr) , \sin \bigl( 2 \pi ( x + y ) - \theta t \bigr) \Bigr) ,
\end{equation*}
where $ \theta ^2 = 8 \pi ^2 + 3 / 4 $. We semidiscretize using the
multisymplectic LDG-H method with equal-order spaces of polynomial
degree $ r = 3 $, constant penalties $ - \alpha ^0 = \alpha ^1 = 1 $,
and mesh size $ h = 0.1 $. We integrate in time using the
St\"ormer/Verlet method, which requires only a linear (rather than
nonlinear) solve at each step due to the separability of the system,
taking step size $ \Delta t = 0.01 h $.

\begin{figure}
  \centering
  \includegraphics[width=5in]{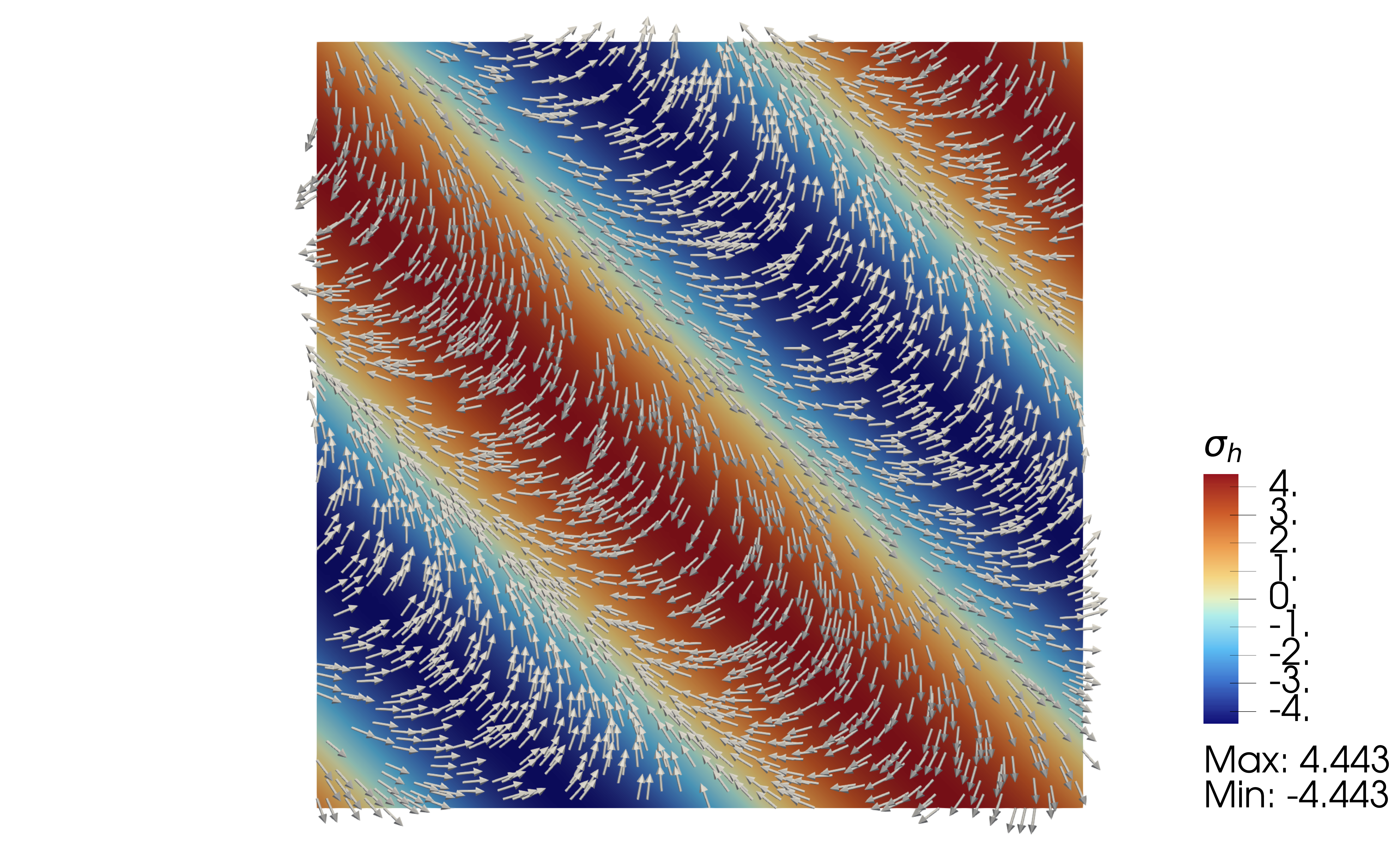}
  \caption{Visualization of the vector field $ u _h $ (arrows) and
    scalar field $ \sigma _h $ (color) at time $ T = 10 $, computed
    using multisymplectic LDG-H semidiscretization with
    St\"ormer/Verlet time stepping. Note that the exact solution
    $\sigma$ has amplitude $ \pi \sqrt{ 2 } \approx 4.443 $, closely
    matched by that of the numerical solution
    $ \sigma _h $.\label{fig:u_semilinear}}
\end{figure}

\begin{figure}
  \centering
  \includegraphics{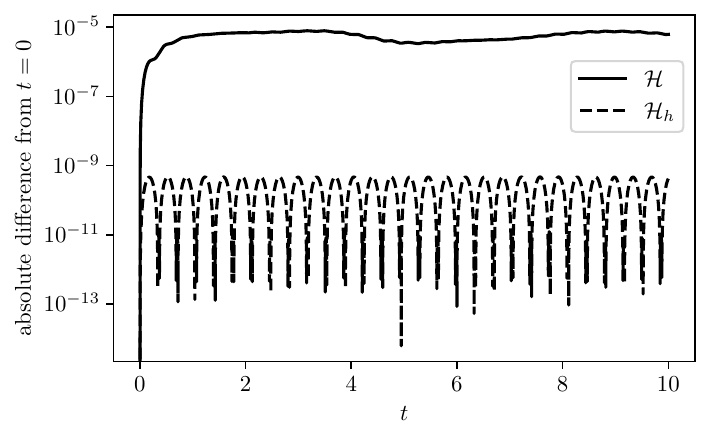}
  \caption{Absolute error in the global Hamiltonian $\mathcal{H}$ and
    discrete Hamiltonian $ \mathcal{H} _h $ for the multisymplectic
    LDG-H method with St\"ormer/Verlet time stepping. For this
    nonlinear problem, as in the linear case, both quantities are
    nearly conserved. \label{fig:H_semilinear}}
\end{figure}

\Cref{fig:u_semilinear} shows the numerical solution computed at time
$ T = 10 $, evincing near-preservation of the amplitude of the plane
wave. \Cref{fig:H_semilinear} shows the evolution of the global
Hamiltonian $\mathcal{H}$ and discrete Hamiltonian $ \mathcal{H} _h
$. We observe near-conservation of $ \mathcal{H} _h $ within
$ \approx 10 ^{ - 9 } $ and of $ \mathcal{H} $ within
$ \approx 10 ^{ - 5 } $, with bounded errors reflecting the difference
between $ \mathcal{H} _h $ and $\mathcal{H}$. Unlike the linear case,
since $ \mathcal{H} _h $ is not bilinear (not even quadratic), we
should not expect the St\"ormer/Verlet method to conserve it exactly,
even in exact arithmetic. However, since the method is symplectic, we
observe bounded oscillation of $ \mathcal{H} _h $, rather than drift,
due to conservation of a nearby modified discrete Hamiltonian;
cf.~\citet[Section IX.3]{HaLuWa2006} and references therein.


\footnotesize
\begin{thebibliography}{55}
\providecommand{\natexlab}[1]{#1}

\bibitem[{Arnold(2018)}]{Arnold2018}
\textsc{D.~N. Arnold}, \emph{Finite element exterior calculus}, vol.~93 of
  CBMS-NSF Regional Conference Series in Applied Mathematics, Society for
  Industrial and Applied Mathematics (SIAM), Philadelphia, PA, 2018.

\bibitem[{Arnold et~al.(2006)Arnold, Falk, and Winther}]{ArFaWi2006}
\textsc{D.~N. Arnold, R.~S. Falk, and R.~Winther}, \emph{Finite element
  exterior calculus, homological techniques, and applications}, Acta Numer., 15
  (2006), pp. 1--155.

\bibitem[{Arnold et~al.(2010)Arnold, Falk, and Winther}]{ArFaWi2010}
\leavevmode\vrule height 2pt depth -1.6pt width 23pt, \emph{Finite element
  exterior calculus: from {H}odge theory to numerical stability}, Bull. Amer.
  Math. Soc. (N.S.), 47 (2010), pp. 281--354.

\bibitem[{Awanou et~al.(2023)Awanou, Fabien, Guzm\'{a}n, and
  Stern}]{AwFaGuSt2023}
\textsc{G.~Awanou, M.~Fabien, J.~Guzm\'{a}n, and A.~Stern}, \emph{Hybridization
  and postprocessing in finite element exterior calculus}, Math. Comp., 92
  (2023), pp. 79--115.

\bibitem[{Belishev and Sharafutdinov(2008)}]{BeSh2008}
\textsc{M.~Belishev and V.~Sharafutdinov}, \emph{Dirichlet to {N}eumann
  operator on differential forms}, Bull. Sci. Math., 132 (2008), pp. 128--145.

\bibitem[{Bochev and Scovel(1994)}]{BoSc1994}
\textsc{P.~B. Bochev and C.~Scovel}, \emph{On quadratic invariants and
  symplectic structure}, BIT, 34 (1994), pp. 337--345.

\bibitem[{Bridges(1997)}]{Bridges1997}
\textsc{T.~J. Bridges}, \emph{Multi-symplectic structures and wave
  propagation}, Math. Proc. Cambridge Philos. Soc., 121 (1997), pp. 147--190.

\bibitem[{Bridges(2006)}]{Bridges2006}
\leavevmode\vrule height 2pt depth -1.6pt width 23pt, \emph{Canonical
  multi-symplectic structure on the total exterior algebra bundle}, Proc. R.
  Soc. Lond. Ser. A Math. Phys. Eng. Sci., 462 (2006), pp. 1531--1551.

\bibitem[{Bridges and Reich(2001)}]{BrRe2001}
\textsc{T.~J. Bridges and S.~Reich}, \emph{Multi-symplectic integrators:
  numerical schemes for {H}amiltonian {PDE}s that conserve symplecticity},
  Phys. Lett. A, 284 (2001), pp. 184--193.

\bibitem[{Brink et~al.(2017)Brink, Izs\'ak, and van~der Vegt}]{BrFeVdV2017}
\textsc{F.~Brink, F.~Izs\'ak, and J.~J.~W. van~der Vegt}, \emph{Hamiltonian
  finite element discretization for nonlinear free surface water waves}, J.
  Sci. Comput., 73 (2017), pp. 366--394.

\bibitem[{Celledoni and Jackaman(2021)}]{CeJa21}
\textsc{E.~Celledoni and J.~Jackaman}, \emph{Discrete conservation laws for
  finite element discretisations of multisymplectic {PDE}s}, J. Comput. Phys.,
  444 (2021), pp. Paper No. 110520, 26.

\bibitem[{Chen(2008)}]{Chen2008}
\textsc{J.-B. Chen}, \emph{Variational integrators and the finite element
  method}, Appl. Math. Comput., 196 (2008), pp. 941--958.

\bibitem[{Cockburn et~al.(2009)Cockburn, Gopalakrishnan, and
  Lazarov}]{CoGoLa2009}
\textsc{B.~Cockburn, J.~Gopalakrishnan, and R.~Lazarov}, \emph{Unified
  hybridization of discontinuous {G}alerkin, mixed, and continuous {G}alerkin
  methods for second order elliptic problems}, SIAM J. Numer. Anal., 47 (2009),
  pp. 1319--1365.

\bibitem[{Cooper(1987)}]{Cooper1987}
\textsc{G.~J. Cooper}, \emph{Stability of {R}unge--{K}utta methods for
  trajectory problems}, IMA J. Numer. Anal., 7 (1987), pp. 1--13.

\bibitem[{de~Donder(1935)}]{deDonder1935}
\textsc{T.~de~Donder}, \emph{Th\'eorie Invariantive du Calcul des Variations},
  Gauthier-Villars, second ed., 1935.

\bibitem[{Frank et~al.(2006)Frank, Moore, and Reich}]{FrMoRe2006}
\textsc{J.~Frank, B.~E. Moore, and S.~Reich}, \emph{Linear {PDE}s and numerical
  methods that preserve a multisymplectic conservation law}, SIAM J. Sci.
  Comput., 28 (2006), pp. 260--277.

\bibitem[{Guo et~al.(2001)Guo, Ji, Li, and Wu}]{GuJiLiWu2001}
\textsc{H.-Y. Guo, X.-M. Ji, Y.-Q. Li, and K.~Wu}, \emph{A note on symplectic,
  multisymplectic scheme in finite element method}, Commun. Theor. Phys.
  (Beijing), 36 (2001), pp. 259--262.

\bibitem[{Hairer et~al.(2006)Hairer, Lubich, and Wanner}]{HaLuWa2006}
\textsc{E.~Hairer, C.~Lubich, and G.~Wanner}, \emph{Geometric numerical
  integration}, vol.~31 of Springer Series in Computational Mathematics,
  Springer-Verlag, Berlin, second ed., 2006.

\bibitem[{Jay(2021)}]{Jay2021}
\textsc{L.~O. Jay}, \emph{Symplecticness conditions of some low order
  partitioned methods for non-autonomous {H}amiltonian systems}, Numer.
  Algorithms, 86 (2021), pp. 495--514.

\bibitem[{Lasagni(1988)}]{Lasagni1988}
\textsc{F.~M. Lasagni}, \emph{Canonical {R}unge--{K}utta methods}, Z. Angew.
  Math. Phys., 39 (1988), pp. 952--953.

\bibitem[{Lehrenfeld(2010)}]{Lehrenfeld2010}
\textsc{C.~Lehrenfeld}, \emph{Hybrid discontinuous {G}alerkin methods for
  solving incompressible flow problems}, Master's thesis, RWTH Aachen, 2010.

\bibitem[{Lehrenfeld and Sch\"{o}berl(2016)}]{LeSc2016}
\textsc{C.~Lehrenfeld and J.~Sch\"{o}berl}, \emph{High order exactly
  divergence-free hybrid discontinuous {G}alerkin methods for unsteady
  incompressible flows}, Comput. Methods Appl. Mech. Engrg., 307 (2016), pp.
  339--361.

\bibitem[{Leimkuhler and Reich(2004)}]{LeRe2004}
\textsc{B.~Leimkuhler and S.~Reich}, \emph{Simulating {H}amiltonian dynamics},
  vol.~14 of Cambridge Monographs on Applied and Computational Mathematics,
  Cambridge University Press, Cambridge, 2004.

\bibitem[{Leopardi and Stern(2016)}]{LeSt2016}
\textsc{P.~Leopardi and A.~Stern}, \emph{The abstract {H}odge-{D}irac operator
  and its stable discretization}, SIAM J. Numer. Anal., 54 (2016), pp.
  3258--3279.

\bibitem[{Lew et~al.(2003)Lew, Marsden, Ortiz, and West}]{LeMaOrWe2003}
\textsc{A.~Lew, J.~E. Marsden, M.~Ortiz, and M.~West}, \emph{Asynchronous
  variational integrators}, Arch. Ration. Mech. Anal., 167 (2003), pp. 85--146.

\bibitem[{Marsden et~al.(1998)Marsden, Patrick, and Shkoller}]{MaPaSh1998}
\textsc{J.~E. Marsden, G.~W. Patrick, and S.~Shkoller}, \emph{Multisymplectic
  geometry, variational integrators, and nonlinear {PDE}s}, Comm. Math. Phys.,
  199 (1998), pp. 351--395.

\bibitem[{Marsden et~al.(2001)Marsden, Pekarsky, Shkoller, and
  West}]{MaPeShWe2001}
\textsc{J.~E. Marsden, S.~Pekarsky, S.~Shkoller, and M.~West},
  \emph{Variational methods, multisymplectic geometry and continuum mechanics},
  J. Geom. Phys., 38 (2001), pp. 253--284.

\bibitem[{Marsden and Ratiu(1999)}]{MaRa1999}
\textsc{J.~E. Marsden and T.~S. Ratiu}, \emph{Introduction to mechanics and
  symmetry}, vol.~17 of Texts in Applied Mathematics, Springer-Verlag, New
  York, second ed., 1999.

\bibitem[{Marsden and Shkoller(1999)}]{MaSh1999}
\textsc{J.~E. Marsden and S.~Shkoller}, \emph{Multisymplectic geometry,
  covariant {H}amiltonians, and water waves}, Math. Proc. Cambridge Philos.
  Soc., 125 (1999), pp. 553--575.

\bibitem[{McLachlan and Stern(2020)}]{McSt2020}
\textsc{R.~I. McLachlan and A.~Stern}, \emph{Multisymplecticity of hybridizable
  discontinuous {G}alerkin methods}, Found. Comput. Math., 20 (2020), pp.
  35--69.

\bibitem[{McLachlan and Stern(2024)}]{McSt2024}
\leavevmode\vrule height 2pt depth -1.6pt width 23pt, \emph{Functional
  equivariance and conservation laws in numerical integration}, Found. Comput.
  Math., 24 (2024), pp. 149--177.

\bibitem[{Moiola and Perugia(2018)}]{MoPe2018}
\textsc{A.~Moiola and I.~Perugia}, \emph{A space-time {T}refftz discontinuous
  {G}alerkin method for the acoustic wave equation in first-order formulation},
  Numer. Math., 138 (2018), pp. 389--435.

\bibitem[{Moore and Reich(2003)}]{MoRe2003}
\textsc{B.~Moore and S.~Reich}, \emph{Backward error analysis for
  multi-symplectic integration methods}, Numer. Math., 95 (2003), pp. 625--652.

\bibitem[{Nguyen et~al.(2011)Nguyen, Peraire, and Cockburn}]{NgPeCo2011_wave}
\textsc{N.~C. Nguyen, J.~Peraire, and B.~Cockburn}, \emph{High-order implicit
  hybridizable discontinuous {G}alerkin methods for acoustics and
  elastodynamics}, J. Comput. Phys., 230 (2011), pp. 3695--3718.

\bibitem[{N\'u\~nez and S\'anchez(2025)}]{NuSa2025}
\textsc{C.~N\'u\~nez and M.~A. S\'anchez}, \emph{Symplectic {H}amiltonian
  hybridizable discontinuous {G}alerkin methods for linearized shallow water
  equations}, Comput. Methods Appl. Mech. Engrg., 447 (2025), pp. Paper No.
  118383, 19.

\bibitem[{Oikawa(2015)}]{Oikawa2015}
\textsc{I.~Oikawa}, \emph{A hybridized discontinuous {G}alerkin method with
  reduced stabilization}, J. Sci. Comput., 65 (2015), pp. 327--340.

\bibitem[{Oikawa(2016)}]{Oikawa2016}
\leavevmode\vrule height 2pt depth -1.6pt width 23pt, \emph{Analysis of a
  reduced-order {HDG} method for the {S}tokes equations}, J. Sci. Comput., 67
  (2016), pp. 475--492.

\bibitem[{Quenneville-Belair(2015)}]{Quenneville-Belair2015}
\textsc{V.~Quenneville-Belair}, \emph{A new approach to finite element
  simulations of general relativity}, Ph.D. thesis, University of Minnesota,
  2015. Available at \url{https://hdl.handle.net/11299/175309}.

\bibitem[{Reich(2000{\natexlab{a}})}]{Reich2000a}
\textsc{S.~Reich}, \emph{Finite volume methods for multi-symplectic {PDE}s},
  BIT, 40 (2000{\natexlab{a}}), pp. 559--582.

\bibitem[{Reich(2000{\natexlab{b}})}]{Reich2000b}
\leavevmode\vrule height 2pt depth -1.6pt width 23pt, \emph{Multi-symplectic
  {R}unge-{K}utta collocation methods for {H}amiltonian wave equations}, J.
  Comput. Phys., 157 (2000{\natexlab{b}}), pp. 473--499.

\bibitem[{S\'anchez et~al.(2017)S\'anchez, Ciuca, Nguyen, Peraire, and
  Cockburn}]{SaCiNgPe17}
\textsc{M.~A. S\'anchez, C.~Ciuca, N.~C. Nguyen, J.~Peraire, and B.~Cockburn},
  \emph{Symplectic {H}amiltonian {HDG} methods for wave propagation phenomena},
  J. Comput. Phys., 350 (2017), pp. 951--973.

\bibitem[{S\'anchez et~al.(2022)S\'anchez, Du, Cockburn, Nguyen, and
  Peraire}]{SaDuCo22}
\textsc{M.~A. S\'anchez, S.~Du, B.~Cockburn, N.-C. Nguyen, and J.~Peraire},
  \emph{Symplectic {H}amiltonian finite element methods for electromagnetics},
  Comput. Methods Appl. Mech. Engrg., 396 (2022), pp. Paper No. 114969, 27.

\bibitem[{S\'anchez and Valenzuela(2024)}]{SaVa24}
\textsc{M.~A. S\'anchez and J.~Valenzuela}, \emph{Symplectic {H}amiltonian
  finite element methods for semilinear wave propagation}, J. Sci. Comput., 99
  (2024), pp. Paper No. 62, 27.

\bibitem[{Sanz-Serna(1988)}]{Sanz-Serna1988}
\textsc{J.~M. Sanz-Serna}, \emph{Runge--{K}utta schemes for {H}amiltonian
  systems}, BIT, 28 (1988), pp. 877--883.

\bibitem[{Sanz-Serna(2016)}]{Sanz-Serna2016}
\leavevmode\vrule height 2pt depth -1.6pt width 23pt, \emph{Symplectic
  {R}unge-{K}utta schemes for adjoint equations, automatic differentiation,
  optimal control, and more}, SIAM Rev., 58 (2016), pp. 3--33.

\bibitem[{Sanz-Serna and Calvo(1994)}]{SaCa1994}
\textsc{J.~M. Sanz-Serna and M.~P. Calvo}, \emph{Numerical {H}amiltonian
  problems}, vol.~7 of Applied Mathematics and Mathematical Computation,
  Chapman \& Hall, London, 1994.

\bibitem[{Sch\"{o}berl(2014)}]{Schoeberl2014}
\textsc{J.~Sch\"{o}berl}, \emph{C++11 implementation of finite elements in
  {NGSolve}}, ASC Report 30/2014, Institute for Analysis and Scientific
  Computing, Vienna University of Technology, 2014. Available from
  \url{https://ngsolve.org/_static/ngs-cpp11.pdf}.

\bibitem[{Stern and Zampa(2025)}]{StZa2025}
\textsc{A.~Stern and E.~Zampa}, \emph{Multisymplecticity in finite element
  exterior calculus}, Found. Comput. Math.,  (2025).
  \href{https://doi.org/10.1007/s10208-025-09720-y}{doi:10.1007/s10208-025-09720-y}.

\bibitem[{Sun(1993)}]{Sun1993}
\textsc{G.~Sun}, \emph{Symplectic partitioned {R}unge-{K}utta methods}, J.
  Comput. Math., 11 (1993), pp. 365--372.

\bibitem[{Suris(1990)}]{Suris1990}
\textsc{Y.~B. Suris}, \emph{Hamiltonian methods of {R}unge-{K}utta type and
  their variational interpretation}, Mat. Model., 2 (1990), pp. 78--87.

\bibitem[{Weck(2004)}]{Weck2004}
\textsc{N.~Weck}, \emph{Traces of differential forms on {L}ipschitz
  boundaries}, Analysis (Munich), 24 (2004), pp. 147--169.

\bibitem[{Weyl(1935)}]{Weyl1935}
\textsc{H.~Weyl}, \emph{Geodesic fields in the calculus of variation for
  multiple integrals}, Ann. of Math. (2), 36 (1935), pp. 607--629.

\bibitem[{Xu et~al.(2008)Xu, van~der Vegt, and Bokhove}]{XuVdVBo2008}
\textsc{Y.~Xu, J.~J.~W. van~der Vegt, and O.~Bokhove}, \emph{Discontinuous
  {H}amiltonian finite element method for linear hyperbolic systems}, J. Sci.
  Comput., 35 (2008), pp. 241--265.

\bibitem[{Yoshida(1990)}]{Yoshida1990}
\textsc{H.~Yoshida}, \emph{Construction of higher order symplectic
  integrators}, Phys. Lett. A, 150 (1990), pp. 262--268.

\bibitem[{Zhen et~al.(2003)Zhen, Bai, Li, and Wu}]{ZhBaLiWu2003}
\textsc{L.~Zhen, Y.~Bai, Q.~Li, and K.~Wu}, \emph{Symplectic and
  multisymplectic schemes with the simple finite element method}, Phys. Lett.
  A, 314 (2003), pp. 443--455.

\end{thebibliography}
\end{document}